\documentclass[reqno,a4paper,12pt]{amsart}

\usepackage[all,poly]{xy}
\usepackage{amsfonts,stmaryrd}
\usepackage[mathcal]{eucal}
\usepackage{amssymb}
\usepackage{amsmath}
\usepackage{mathrsfs}
\usepackage{color}
\usepackage[pagebackref,colorlinks]{hyperref}
\usepackage{enumerate}

\theoremstyle{plain}
\newtheorem {Lem}{Lemma}

\newtheorem {The}[Lem]{Theorem}
\newtheorem {LemA}{Lemma}

\newtheorem {LemB}{Lemma}

\newtheorem {LemC}{Lemma}

\newtheorem {TheA}{Theorem}

\newtheorem {TheB}{Theorem}

\newtheorem {TheC}{Theorem}

\newtheorem {CorA}{Corollary}

\newtheorem {CorB}{Corollary}

\newtheorem {CorC}{Corollary}

\newtheorem {Cor}[Lem]{Corollary}
\newtheorem {Prop}[Lem]{Proposition}

\newtheorem {Prob}{Problem}

\theoremstyle{remark}

\theoremstyle{definition}

\def\ord{\operatorname{ord}}
\def\prim{\frak p}

\def\Int{{\Bbb Z}}

\def\bar{\overline}
\def\map{\longrightarrow}

\def\a{\alpha}
\def\b{\beta}

\def\e{\varepsilon}

\def\GL{\operatorname{GL}}

\def\SL{\operatorname{SL}}
\def\SK{\operatorname{SK}}
\def\Sp{\operatorname{Sp}}

\def\Ep{\operatorname{Ep}}

\def\EU{\operatorname{EU}}
\def\FU{\operatorname{FU}}
\def\CU{\operatorname{CU}}
\def\St{\operatorname{St}}

\def\Max{\operatorname{Max}}
\def\Ann{\operatorname{Ann}}

\def\Co{{\Bbb C}}
\def\E{\operatorname{E}}
\def\A{\operatorname{A}}
\def\B{\operatorname{B}}
\def\C{\operatorname{C}}

\def\E{\operatorname{E}}

\def\G{\operatorname{G}}
\def\K{\operatorname{K}}

\def\rk{\operatorname{rk}}

\def\unlhd{\trianglelefteq}
\def\aeq{\Longleftrightarrow}
\def\seq{\Longrightarrow}
\def\map{\longrightarrow}
\def\Form{{A,\Lambda}}
\def\ma{{\frak a}}
\def\mm{{\frak m}}

\newcommand{\FormR}{A,\Lambda}

\newcommand{\LF}{\lfloor}
\newcommand{\RF}{\rfloor}

\def\Cent{\operatorname{Cent}}
\def\Max{\operatorname{Max}}

\def\sr{\operatorname{sr}}

\def\map{\longrightarrow}
\def\bar{\overline}
\def\epsilon{\varepsilon}
\def\e{\varepsilon}
\def\a{\alpha}
\def\b{\beta}

\newcommand{\gm}{\mathfrak m}

\newcommand\Label[1]{\label{#1}}

\newcommand{\ep}{\varepsilon}

\newcommand\maxx{\operatorname{max}}
\newcommand\minn{\operatorname{min}}
\newcommand\Spec{\operatorname{Spec}}
\newcommand\asr{\operatorname{asr}}

\newcommand\lam{\Lambda}

\newcommand\gam{\Gamma}

\newcommand\formr{(A,\lam)}
\newcommand\formi{(I,\gam)}

\newcommand\KU{\operatorname{KU}}

\def\Ker{\text{\rm Ker}}

\def\rk{\operatorname{rk}}

\def\map{\longrightarrow}

\def\e{\varepsilon}
\def\a{\alpha}
\def\b{\beta}

\def\bar{\overline}
\def\tilde{\widetilde}
\def\A{\operatorname{A}}
\def\B{\operatorname{B}}
\def\C{\operatorname{C}}

\def\E{\operatorname{E}}

\def\G{\operatorname{G}}
\def\K{\operatorname{K}}
\def\SSL#1{\operatorname{S}^{#1}\!\operatorname{L}}

\def\Sp{\operatorname{Sp}}
\def\Ep{\operatorname{Ep}}

\def\SL{\operatorname{SL}}
\def\GL{\operatorname{GL}}

\def\GU{\operatorname{GU}}
\def\SU{\operatorname{SU}}

\def\Max{\operatorname{Max}}

\def\sr{\operatorname{sr}}
\def\Int{{\Bbb Z}}

\long\def\forget#1\forgotten{}

\title{The commutators of classical groups}

\author{R.~Hazrat}
\address{
Centre for Research in Mathematics,
Western Sydney University,
Australia}
\email{r.hazrat@WesternSydney.edu.au}

\author{N.~Vavilov}
\address{Department of Mathematics and Mechanics,
St.~Petersburg State University, St.~Petersburg, Russia}
\email{nikolai-vavilov@yandex.ru}

\author{Z.~Zhang}
\address{Department of  Mathematics, Beijing Institute
of Technology, Beijing, China}
\email{zuhong@gmail.com}

\keywords{
General linear groups, unitary groups, Chevalley groups, elementary subgroups,
elementary generators, localisation, relative subgroups,
conjugation calculus, commutator calculus, Noetherian reduction,
the Quillen--Suslin lemma, localisation-completion, commutator
formulae, commutator width, nilpotency of $\K_1$, nilpotent
filtration}

\begin{thanks}
{The second author started this research within the
framework of the RFFI/Indian Academy cooperation project
10-01-92651 ``Higher composition laws, algebraic $K$-theory and
algebraic groups'' (SPbGU--Tata Institute) and the RFFI/BRFFI
cooperation  project 10-01-90016 ``The structure of forms of
reductive groups, and behaviour of small unipotent elements in
representations of algebraic groups'' (SPbGU--Mathematics
Institute of the Belorussian Academy). Currently the work of the
second author is supported by the RFFI research
project 11-01-00756 (RGPU) and by the State Financed research
task 6.38.74.2011 at the Saint Petersburg State University
``Structure theory and geometry of algebraic groups and their
applications in representation theory and algebraic $K$-theory''.
The second author is also supported by the RFFI research project
12-01-00947 (POMI).}
\end{thanks}

\begin{document}

\begin{abstract}
In his seminal paper, half a century ago, Hyman Bass established a commutator formula in the setting of (stable) general linear group which was the key step in defining the $K_1$ group. Namely, he proved that for an associative ring $A$ with identity, 
\[ E(A)=[E(A),E(A)]=[\GL(A),\GL(A)], \]
where $\GL(A)$ is the stable general linear group and $E(A)$ is its elementary subgroup. Since then, various commutator formulas have been studied in stable and non-stable settings, and for  a range of classical and algebraic like-groups, mostly in relation to subnormal subgroups of these groups. 
The major classical theorems and methods developed include some of the splendid results of the heroes of classical algebraic $K$-theory; Bak, Quillen, Milnor, Suslin, Swan and Vaserstein, among others.

One of the dominant techniques in establishing commutator type results is localisation. 
In this note we describe some recent applications of
localisation methods to the study  (higher/relative) commutators in the groups of
points of algebraic and algebraic-like groups, such as general linear groups, $\GL(n,A)$,
unitary groups $\GU(2n,A,\Lambda)$ and Chevalley groups
$G(\Phi,A)$. 
We also state some of the intermediate
results as well as some corollaries of these results.

This note provides a general overview of the subject and covers the current activities. It contains complete proofs of several main results to give the reader a self-contained source. We have borrowed the proofs from our previous papers and expositions 
~\cite{33}--\cite{47},\cite{87,88},\cite{112}--\cite{115}. 

\end{abstract}

\maketitle

\setcounter{tocdepth}{1}
\tableofcontents

\begin{flushright}
{\it Everybody knows there is no fineness or accuracy of suppression;\\ if you hold down one thing you hold down the adjoining.}\\
Saul Bellow
\end{flushright}

\section{Introduction}

Let $A$ be a ring and $I$ be a two sided ideal of $A$. In his seminal paper~\cite{14}, fifty years ago, Bass laid out a theory now known as the classical algebraic $K$-theory (as opposed to the higher algebraic $K$-theory introduced by Quillen~\cite{Quillen72}). 
He considered the stable general linear group $\GL(A)=\bigcup_{n=1}^\infty \GL(n,A)$ and its stable elementary subgroup 
$E(A)=\bigcup_{n=1}^\infty E(n,A)$ and defined the stable $K_1(A)$ as the quotient $\GL(A)/E(A)$ (see~\S\ref{sec22} for details). Relating the group structure of $\GL(A)$ to the ideal structure of $A$, he went on to establish an exact sequence naturally relating $K_1$ to the group $K_0$, previously defined by Grothendieck and Serre. In order the coset space $K_1(A)$ to be a well-defined group, Bass proved his famous ``Whitehead
lemma'' (\cite[Theorem~3.1]{14}, see Lemma~\ref{whiteee}), i.e., 
\[ E(A,I)=[E(A),E(A,I)]=[\GL(A),\GL(A,I)]. \]
In particular when $I=A$, it follows that $E(A)$ is a normal subgroup of $\GL(A)$. 

He further proved that if $n\ge \max\{\sr(A),3\}$, where $\sr(A)$ is the stable range of $A$, then 
\begin{equation}\label{hggsy2}
 E(n,A,I)=[\GL(n,A),E(n,A,I)].
 \end{equation}
Again, when $I=A$, it follows that $E(n,A)$ is a normal subgroup of $\GL(n,A)$.

The next natural question arose was whether $E(n,A)$ is a normal subgroup of $\GL(n,A)$ below the stable range as well. 
In the non-stable case,  there is no ``room'' available for
manoeuvring as in the  stable case (see the proof of Whitehead
Lemma~\ref{whiteee}). Thus, one is forced to put
some finiteness assumption on the ring. Indeed, Gerasimov~\cite{28} produced examples of rings $A$
for which, for any $n\geq 2$,  $E(n,A)$ is as far from being normal in $\GL(n,R)$, as one
can imagine.

A major contribution in this direction came with the work of Suslin~\cite{90},
\cite{TUL} who showed that if $A$ is a {\it module finite ring}, namely, a ring
that is finitely generated as module over its centre, and $n\ge 3$, 
then $E(n,A)$ is a normal subgroup of $\GL(n,A)$. That Suslin's
normality theorem (and the methods developed to prove it) implies the
standard commutator formulae of the type~(\ref{hggsy2}) in full force was somewhat later
observed independently by Borewicz--Vavilov \cite{borvav} and
Vaserstein~\cite{98}.  In these work it was established that, for a 
 module finite ring $A$ and a two-sided ideal $I$ of $A$ and  $n\ge 3$, we have (see~\S\ref{hyhyhy})
\[
\big[E(n,A),\GL(n,A,I)\big]=E(n,A,I).
\] 
The focus then shifted to the relative commutators with two ideals. 
In his paper, Bass already proved that for a ring $A$ and two sided ideals $I,J$, and $n \geq \max(\sr(A)+1,3)$, 
\begin{equation}\label{uu11}
 \big [E(n,A,I),\GL(n,A,J) \big ]=\big [E(n,A,I),E(n,A,J) \big ]. 
\end{equation}

Mason and Stothers, building on Bass' result improved the formula, with the same assumptions, to 
\[ \big [\GL(n,A,I),\GL(n,A,J) \big ]=\big [E(n,A,I),E(n,A,J) \big ]. \]
Later, in a series of the papers, the authors with A.~Stepanov proved that the commutator formula~(\ref{uu11}) is valid for any module finite ring $A$ and $n\geq 3$ (see Theorem~\ref{gcformula1}). 

Since Suslin's work,  five major noticeably different methods have been developed for arbitrary rings to prove such commutator formulae results (and carried out in different classical groups):

\smallskip\noindent
$\bullet$ Suslin's direct factorisation method~\cite{90},
\cite{91}, \cite{56} (see also \cite{32});

\smallskip\noindent $\bullet$ Suslin's factorisation and patching
method~\cite{TUL}, \cite{khleb}, \cite{13};

\smallskip\noindent
$\bullet$ Quillen--Suslin--Vaserstein's localisation and patching
method, \cite{90}, \cite{98}, \cite{95}, \cite{92};

\smallskip\noindent
$\bullet$ Bak's localisation-completion method~\cite{8},
\cite{33}, \cite{12};

\smallskip\noindent
$\bullet$ Stepanov--Vavilov--Plotkin's decomposition of unipotents
\cite{107}, \cite{111}, \cite{87}, \cite{108}.

Suslin's result makes it possible to define the non-stable
$K_{1,n}:=\GL(n,A)/E(n,A)$, when $n \geq 3$, for module finite rings. The study of
these non-stable $K_1$'s is known to be very difficult. There are
examples due to van der Kallen~\cite{53} and Bak~\cite{8} which
show that non-stable $K_1$ can be non-abelian and the natural
question is how non-abelian it can be?

The breakthrough came with the work of Bak~\cite{8}, who showed
that $K_{1,n}$ is nilpotent by abelian 
if $n\geq 3$ and the ring satisfies some dimension condition (e.g.
has a centre with finite Krull dimension).
His method  consists of some ``conjugation calculus'' on
elementary elements, plus simultaneously applying
localisation-patching and completion. This is the method which opened doors to establishing the so called, higher commutator formulas and will be employed in this paper. 

Localisation is one of the most powerful ideas in the study of classical groups over rings. 
It allows to reduce many important problems over
various classes of rings subject to commutativity conditions, to
similar problems for semi-local rings. 
Both methods --the Quillen-Suslin and Bak's approach (particularly the latter)-- rely on a large body of common calculations,
and technical facts, known as {\it conjugation calculus} and
{\it commutator calculus}. Often times these calculations are even referred
to as the {\it yoga of conjugation}, and the {\it yoga of commutators},
to stress the overwhelming feeling of technical strain and exertion. We use variations of these methods to prove multiple commutator formulas for general linear group  of the following type (see~\S\ref{nervsue}):
\begin{multline}\label{hhhyhy}
\Big [E(n,A,I_0),\GL(n,A,I_1), \GL(n,A, I_2),\ldots, \GL(n,A, I_m) \Big]=\\
 \Big[E(n,A,I_0),E(n,A,I_1),E(n,A, I_2),\ldots, E(n,A, I_m) \Big].
\end{multline}
First note that one can produce examples of a commutative ring $A$ and ideals $I,J$ and $K$ such that (see~\S\ref{gbdtmu43})
\[ [E(n,A,I),E(n,A,J)] \not = E(n,A,IJ),\]
and (see~\S\ref{ziuso1})
\begin{equation*} 
\Big [[E(A,I),E(A,J)],E(A,K) \Big ] \not =\Big [E(A,I),[E(A,J),E(A,K)]\Big].
\end{equation*}
So higher commutator formulas of the form (\ref{hhhyhy}) is far from trivial. We will observe that using some commutator calculus, and induction, the proof of  (\ref{hhhyhy}) reduces to prove the base of induction, i.e., to prove 
\begin{equation}\label{uijiuj}
 \Big[[E(n,A,I),\GL(n,A,J)],\GL(n,A,K)\Big]=
\Big[[E(n,A,I),E(n,A,J)],E(n,A,K)\Big]. 
\end{equation}
The proof of~(\ref{uijiuj}) constitutes the bulk of work and uses a variation of  localisation method first developed in~\cite{8}.

The path to full-scale generalisation of these results from general linear groups to other
classical groups was anything but straightforward. For instance, in the unitary case, due to the following circumstances,
\par\smallskip
$\bullet$ the presence of long and short roots,
\par\smallskip
$\bullet$ complicated elementary relations, 
\par\smallskip
$\bullet$ non-commutativity,
\par\smallskip
$\bullet$ non-trivial involution,
\par\smallskip
$\bullet$ non-trivial form parameter,
\par\smallskip\noindent
these yoga calculations tend to be especially lengthy, and highly involved. In this paper, for a comparison, we only provide one proof in the case of unitary groups (which has not been appeared before). Namely, whereas the proof of 
Lemma~\ref{Lem:Habdank} in the setting of general linear groups is only a half of a page, the proof of its counterpart in the unitary setting, Lemma~\ref{yyqq1}, constitutes more than 4 pages.

The aim of this note is to start with the original Bass' Whitehead lemma and continue to establish  the (higher) commutator formulas. We trace the literature on this theme, provide proofs to the main results in the setting of the general linear group and formulate the results in  
other classical-like groups. We aim to provide a self-contained source from the results scattered in the literature.

\section{The groups, an overview}

In this paper we consider algebraic-like
or classical-like group functors $G$. We let $G(A)$ to be
the group of points of $G$ over a ring $A$. Note that
groups of types other than $\A_l$ only exist over commutative
rings. Typically, $G(A)$ is one of the following groups.
\begin{itemize}
\item[{\bf A.}]  General linear group $\GL(n,A)$ of degree $n$
over a ring $A$.
\end{itemize}

In this context the ring $A$ does not have to be commutative.
However, we have to impose {\it some\/}
commutativity conditions for our results to hold. One of
the well behaved classes is the class of quasi-finite rings. Recall, that
a ring $A$ is called {\it module finite\/} if it is finitely
generated as a module over its centre. {\it Quasi-finite}
rings are direct limits of inductive systems of module finite rings (see~\S\ref{sub:1.2}).
To avoid unnecessary repetitions, in the sequel, speaking of
ideals of an associative ring $A$, we always mean
{\it two-sided} ideals of $A$.
\par

\begin{itemize}
\item[{\bf B.}]  Unitary groups $\GU(2n,A,\Lambda)$ over a
form ring $(A,\Lambda)$.
\end{itemize}

In this setting A is a [not necessarily commutative] ring with involution $\bar{\phantom{\a}}:A\rightarrow A $ and a 
form parameter $\Lambda$ (see~\S\ref{sec3}). As in the case of general linear groups, we usually assume that A is module finite over a commutative ring R. 
In general, $\Lambda$ is not an $R$-module. Thus, $R$ has to be replaced by its subring $R_0$, generated by all $\xi\bar\xi$ with $\xi\in R$. 

\begin{itemize}
\item [{\bf C.}]  Chevalley groups $G(\Phi,A)$ of type $\Phi$ over a commutative ring $A$.
\end{itemize}
Chevalley groups are indeed {\it algebraic\/}, and the ground
rings are {\it commutative\/} in this case, which usually makes
life easier. 

Together with the algebraic-like group $G(A)$ we consider the
following subgroups.
\begin{itemize}
\item[$\bullet$] First of all, the elementary group $E(A)$,
generated by elementary unipotents.
\item[$\circ$] In the linear case, the elementary generators are
elementary [linear] transvec\-tions $e_{ij}(\xi)$,
$1\le i\neq j\le n$, $\xi\in A$.
\item[$\circ$] In the unitary case, the elementary generators are elementary
unitary transvec\-tions $T_{ij}(\xi)$, $1\le i\neq j\le -1$, $\xi\in A$.
In the even hyperbolic case they come in two modifications. They can
be short root type, $i\neq\pm j$, when the parameter $\xi$ can be
any element of $A$. On the other hand, for the long root type $i=-j$
and the parameter $\xi$ must belong to [something defined in terms of]
the form parameter $\Lambda$.
\item[$\circ$] Finally, for Chevalley groups, the elementary generators are
the elementary root unipotents $x_{\a}(\xi)$ for a root $\a\in\Phi$
and a ring element $\xi\in A$.
\end{itemize}

Further, let $I\unlhd A$ be an ideal of $A$. We also consider
the following relative subgroups.
\begin{itemize}
\item[$\bullet$] The elementary group $E(I)$ of level $I$, generated by
elementary unipotents of level $I$.
\item[$\bullet$] The relative elementary group $E(A,I)=E(I)^{E(A)}$
of level $I$.
\item[$\bullet$] The principal congruence subgroups $G(A,I)$ of level $I$,
the kernel of reduction homomorphism $\rho_I:G(A)\map G(A/I)$.
\item[$\bullet$] The full congruence subgroups $C(A,I)$ of level $I$, the
inverse image of the centre of $G(A/I)$ with respect to $\rho_I$.
\end{itemize}
We use the usual notation for these groups in the above contexts
A--C as shown below.
$$ \begin{array}{cccc}
G(A)\qquad\qquad\hfill&\GL(n,A)\qquad\hfill&\GU(n,A,\Lambda)\qquad\hfill&G(\Phi,A)\hfill\\
\noalign{\vskip 5truept}
E(A)\qquad\qquad\hfill&E(n,A)\qquad\hfill&\EU(n,A,\Lambda)\qquad\hfill&E(\Phi,A)\hfill\\
\noalign{\vskip 5truept}
E(I)\qquad\qquad\hfill&E(n,I)\qquad\hfill&\FU(n,I,\Gamma)\qquad\hfill&E(\Phi,I)\hfill\\
\noalign{\vskip 5truept}
E(A,I)\qquad\qquad\hfill&E(n,A,I)\qquad\hfill&\EU(n,I,\Gamma)\qquad\hfill&E(\Phi,A,I)\hfill\\
\noalign{\vskip 5truept}
G(A,I)\qquad\qquad\hfill&\GL(n,A,I)\qquad\hfill&\GU(n,I,\Gamma)\qquad\hfill&G(\Phi,A,I)\hfill\\
\noalign{\vskip 5truept}
C(A,I)\qquad\qquad\hfill&C(n,A,I)\qquad\hfill&\CU(n,I,\Gamma)\qquad\hfill&C(\Phi,A,I)\hfill\\
\end{array} $$
\par
There are two more general contexts, where 
localisation methods have been successfully used, in particular, 
\begin{itemize}
\item[{\bf D.}]  Isotropic reductive groups $G(A)$,
\smallskip
\item[{\bf E.}]  Odd unitary groups $U(V,q)$,
\end{itemize}
however we don't pursue these groups here (see~\cite{75,76,77,78}, \cite{68}, \cite{84}). 


\section{Preliminaries}

We gather here basic results in group and ring theory, which will be used throughout this note.

\subsection{Commutators}\Label{sub:1.4} Let $G$ be a group. For any $x,y\in G$,  $^xy=xyx^{-1} $  denotes the left $x$-conjugate of $y$. Let  $[x,y]=xyx^{-1}y^{-1}$ denote the commutator of $x$ and $y$.  Sometimes the double commutator $[[x, y], z]$ will be denoted simply by $[x, y, z]$ and $$\big[[A,B],C\big]=[A,B,C].$$ Thus we write $[A_1,A_2,A_3,\dots,A_n]$ for $\big[\dots\big[[A_1,A_2],A_3\big],\dots,A_n\big]$ and call it the {\it standard form} of the multiple commutator formulas.

The following formulas will be used frequently (sometimes without giving a reference to them),
\begin{itemize}
\item[(C1)] $[x,yz]=[x,y]({}^y[x,z])$;
\smallskip
\item[(C$1^+$)] 
An easy induction, using identity (C1), shows that 
$$\big [x,\prod_{i=1}^k u_i]=\prod_{i=1}^k {}^{\prod_{j=1}^{i-1}u_j}[x,u_{i}],$$ where by convention $\prod_{j=1}^0 u_j=1$.
\item[(C2)] $[xy,z]=(^x[y,z])[x,z]$;
\smallskip
\item[(C$2^+$)] 
As in (C$1^+$), we have  
$$\big [\prod_{i=1}^k u_i,x\big]=\prod_{i=1}^k {}^{\prod_{j=1}^{k-i}u_j}[u_{k-i+1},x].$$ 
\smallskip
\item[(C3)] (the Hall-Witt identity): 
 ${}^{x}\big[[x^{-1},y],z\big]\, \, {}^{z}\big[[z^{-1},x],y\big]\, \, {}^{y}\big[[y^{-1},z],x\big]=1$;
\smallskip
\item[(C4)] $[x,^yz]=^y[^{y^{-1}}x,z]$;
\smallskip
\item[(C5)] $[^yx,z]=^{y}[x,^{y^{-1}}z]$.
\smallskip
\item[(C6)] If $H$ and $K$ are subgroups of $G$, then $[H,K]=[K,H]$. 
\smallskip
\item[(C7)] If $F$, $H$ and $K$ are subgroups of $G$, then
\[ \Big[[F, H], K\Big] \le  \Big[[F, K], H\Big]\Big[F, [H, K]\Big].\]

In~\S\ref{jalg7} we will provide an example that even in the setting of elementary subgroups of a linear group \[\Big[[F, H], K\Big] \not = \Big[F, [H, K]\Big].\] 

\item[(C8)] $(xy)^2=x^2y^2 \Big[y^{-1},x^{-1}\Big]\Big[[x^{-1},y^{-1}]y^{-1}\Big].$
\end{itemize}

One can write numerous identities involving commutators. The reader is referred to~\cite{48,49} for more samples of these identities.

\subsection{}\label{symhf483}
Let $A$ be a ring and $I,J$ and $K$ be two sided ideals. We denote by 
\[I\circ J:=IJ+JI,\] the symmetrised product
of ideals $I,J\unlhd A$. In the commutative case it coincides
with their usual product. In general, the symmetrised product
is not associative. Thus, when writing something like 
$I\circ J\circ K$ we have to specify the order in which 
products are formed.

\subsection{Limit of rings}\Label{sub:1.2} An $R$-algebra $A$ is called {\em module finite} over $R$, if $A$ is finitely generated as an $R$-module. An $R$-algebra $A$ is  called {\em quasi-finite} over $R$ if there is a direct system of module finite  $R$-subalgebras $A_i$ of $A$ such that $\varinjlim A_i=A$. 

Suppose $A$ is an $R$-algebra and $I$ is an index set. By a direct system of subalgebras $A_i/R_i$, $i \in I$, of $A$, we shall mean a set of subrings $R_i$ of $R$ and a set of subrings $A_i$ of $A$ such that each $A_i$ is naturally an $R_i$-algebra and such that given 
$i,j \in I$, there is a $k\in I$ such that $R_i\le  R_k$, $R_j\le  R_k$, $A_i\le  A_k$, and $A_j \le  A_k$.

\begin{Prop}\Label{Prop:01}
An $R$-algebra $A$ is quasi-finite over R if and only if it satisfies the following equivalent conditions:
\begin{itemize}
\item[(1)] There is a direct system of subalgebras $A_i/R_i$ of $A$ such that each $A_i$ is module finite over $R_i$ and such that 
$\varinjlim R_i=R$ and $\varinjlim A_i=A$.
\item[(2)] There is a direct system of subalgebras $A_i/R_i$ of $A$ such that each $A_i$ is module finite over $R_i$ and each $R_i$ is finitely generated as a $\mathbb Z$-algebra and such that $\varinjlim R_i=R$ and $\varinjlim A_i=A$.
\end{itemize}
\end{Prop}

\subsection{Stable rank of rings}\label{srmhgy1}
Let us recall the linear case first. These results are most conveniently
stated in terms of the new type of dimension for rings, introduced by
Bass, stable rank. Since later we shall discuss generalisations of this notion,
we recall here its definition.
\par
A row $(a_1,\ldots,a_n)\in{}^{n}A$ is called {\it unimodular\/}
if the elements $a_1,\ldots,a_n$ generate $A$ as a {\it right\/} ideal,
i.e.\ $a_1A+\cdots+a_nA=A$, or, what is the same, there exist
$b_1,\ldots,b_n\in A$ such that $a_1b_1+\cdots+a_nb_n=1$.
\par
A row $(a_1,\ldots,a_{n+1})\in{}^{n+1}A$ is called {\it stable\/},
if there exist  $b_1,\ldots,b_n\in A$ such that the {\it right\/}
ideal generated by $a_1+a_{n+1}b_1,\ldots,a_n+a_{n+1}b_n$
coincides with the {\it right\/} ideal generated by
$a_1,\ldots,a_{n+1}$.
\par
One says that the {\it stable rank\/} of the ring $A$ equals $n$ and
writes $\sr(A)=n$ if every unimodular row of length $n+1$ is stable,
but there exists a non-stable unimodular row of length $n$. If such
$n$ does not exist (i.e.\ there are non-stable unimodular rows of arbitrary
length) we say that the stable rank of $A$ is infinite.
\par
It turned out that stable rank, on one hand,  most naturally arises
in the proof of results pertaining to linear groups and, on the other
hand, it can be easily estimated in terms of other known dimensions
of a commutative ring $A$, say of its Krull dimension $\dim(A)$,
or its Jacobson dimension $j(A)=\dim(\Max(A))$. Here, $\Max(A)$ is
the subspace of all maximal ideals of the topological space $\Spec(A)$, the set of all
prime ideals of $A$, equipped with the  Zariski Topology. Then $j(A)$ is the
dimension of the topological space $\Max(A)$. Let us state a typical
result in this spirit due to Bass.
\begin{The}\label{bassbound}
Let $A$ be a ring finitely generated as a module
over a commutative ring $R$. Then $\sr(A)\le\dim(\Max(R))+1$.
\end{The}
 The right hand side
should be thought of as a condition expressing (a weaker form of)
stability for not necessarily unimodular rows. In \cite{estesohm}
and \cite{kallenmagurn} it is shown that already
$\asr(A)\le\dim(\Max(R))+1$, where $\asr(A)$ stands for the {\it
absolute stable rank}.

\section{General Linear Groups}\label{sec22}

Let $G=\GL(n,A)$ be the general linear group of degree $n$ over an
associative ring $A$ with 1. Recall that $\GL(n,A)$ is the group of
all two-sided invertible square matrices of degree $n$ over $A$,
or, in other words, the multiplicative group of the full matrix
ring $M(n,A)$. When one thinks of $A\mapsto\GL(n,A)$ as a functor from
rings to groups, one writes $\GL_n$. In the sequel for a matrix
$g\in G$ we denote by $g_{ij}$ its matrix entry in the position
$(i,j)$, so that $g=(g_{ij})$, $1\le i,j\le n$. The inverse of
$g$ will be denoted by $g^{-1}=(g'_{ij})$, $1\le i,j\le n$.
\par

A crucial role is played by the elementary subgroup of $\GL(n,A)$.
As usual we denote by $e$ (or sometimes 1) the identity matrix of degree $n$
and by $e_{ij}$ a standard matrix unit, i.e., the matrix
that has 1 in the position $(i,j)$ and zeros elsewhere.
An elementary matrices $e_{i,j}(\xi)$ is a matrix of the
form
$$ e_{i,j}(\xi)=e+\xi e_{ij},\qquad \xi\in A,\quad 1\le i\neq j\le n. $$
\noindent
An elementary matrices $e_{i,j}(\xi)$ only differs from the
identity matrix in the position $(i,j)$, $i\neq j$, where it has
$\xi$ instead of 0.  In other words, multiplication by
an elementary matrix on the left/right performs
what in an undergraduate linear algebra course would be called
a row/column elementary transformation `of the first kind'.

If there is no danger we simply write $e_{ij}(\xi)$ instead of $e_{i,j}(\xi)$.

The {\it elementary subgroup} $E(n,A)$ of the
general linear group $\GL(n,A)$ is generated by all the
elementary matrices. That is, 
$$ E(n,A)=\langle e_{ij}(\xi),\ \xi\in A,\ 1\le i\neq j\le n\rangle. $$

Both for the development of the theory and for the sake of
applications one has to extend these definitions to include
relative groups. For a two-sided ideal $I$ of $A$, one defines
the corresponding reduction homomorphism
$$ \pi_I:\GL(n,A)\map\GL(n,A/I),\qquad
(g_{ij})\mapsto(g_{ij}+I). $$
\noindent
Now the {\it principal congruence subgroup} $\GL(n,A,I)$ of level $I$
is the kernel of reduction homomorphism $\pi_I$, while the {\it full
congruence subgroup} $C(n,A,I)$ of level $I$ is the inverse image
of the centre of $\GL(n,A/I)$ with respect to this homomorphism.
Clearly both are normal subgroups of $\GL(n,A)$.
\par
Again, let $I\unlhd A$ be a two-sided ideal of $A$, and let $x=e_{ij}(\xi)$ be an
elementary matrix. Somewhat loosely we say that $x$ is of
level $I$, provided $\xi\in I$. One can consider the subgroup
generated in $\GL(n,A)$ by all the elementary matrices of level~$I$:
$$ E(n,I)=\langle e_{ij}(\xi), \xi\in I, 1\le i\neq j\le n\rangle. $$
\noindent This group is contained in the absolute elementary
subgroup $E(n,A)$ and does not depend on the choice of an ambient
ring $A$ with 1. However, in general $E(n,I)$ has little chances to
be normal in $\GL(n,A)$. The relative elementary subgroup $E(n,A,I)$
is defined as the normal closure of $E(n,I)$ in $E(n,A)$:
$$ E(n,A,I)=\langle e_{ij}(\xi), \xi\in I, 1\le i\neq j\le n\rangle^{E(n,A)}. $$

We have the following relations among elementary matrices which will be used in the paper. We refer to these relations in the text by (E). 
\begin{itemize}
\item[(E1)] $e_{i,j}(a)e_{i,j}(b)=e_{i,j}(a+b).$
\par\smallskip
\item[(E2)] $[e_{i,j}(a),e_{k,l}(b)]=1$ if $i\not = l, j \not = k$.
\par\smallskip
\item[(E3)] $[e_{i,j}(a),e_{j,k}(b)]=e_{i,k}(ab)$ if $i \not = k$.
\end{itemize}

{\it Essentially\/}, the following result was first established in
the context of Chevalley groups by Michael Stein \cite{85}.
The next approximation is the paper by Jacques Tits \cite{96},
where it is proven that $E(n,A,I)$ is generated by its intersections
with the fundamental $\SL_2$. Nevertheless, the earliest reference,
where we could trace this result, was the paper by Leonid Vaserstein
and Andrei Suslin \cite{104}. We follow the proof given in ~\cite[Lemma 4.8]{8} (see also 
in~\cite[Theorem 11]{115}).

\begin{Lem}\label{Engenerator}
Let $A$ be a ring and $I$ be a two-sided ideal of $A$. Then $E(n,A,I)$ is generated as a group by the elements 
$$z_{ij}(a,\alpha):={}^{e_{ji}(a)}e_{ij}(\alpha)=e_{ji}(a)e_{ij}(\alpha)e_{ji}(-a),$$
where $i\ne j$, $a\in A$ and $\alpha\in I$.
\end{Lem}
\begin{proof}
By definition, $E(n,A, I)$ is generated by the elements ${}^{e}e_{ij}(\alpha)$, where $i\not =j$, $e\in E(n,A)$, and $\alpha \in I$. If $e$ is the identity matrix, let $l(e) = 0$ and otherwise, let $l(e)$ denote the least number of elementary matrices required to write $e$ as a product of elementary matrices. The proof is by induction on $l(e)$.

We need the following identity in order to reduce the length of $e$ in the induction proof. 
Let $i$, $j$, $k$ be distinct natural numbers and $a, b \in  A$ and $\alpha \in  I$.  Then one can check by straightforward multiplication that
 \begin{multline}\label{gbgdq1}
{}^{e_{ij}(a)e_{ji}(b)}e_{ij}(\alpha)
= e_{kj}(- \alpha(1 + ba)) e_{ki}(\alpha b) e_{ik} (-ab \alpha b)e_{ij}(ab \alpha) \times\\
 ({}^{e_{jk}(b)}e_{kj}(\alpha))e_{ij}(\alpha)e_{ik}((ab -1) \alpha b)e_{jk}(b \alpha b)({}^{e_{ij}(a)}e_{ji}(-b \alpha b)) \times\\
({}^{e_{ki}(1)}e_{ik}( \alpha b))e_{kj}(\alpha ba)e_{ij}(\alpha ba).
\end{multline}

We proceed by induction.  If $l(e) = 0$, there is nothing to prove.
Suppose $l(e)= 1$. Then $e =e_{kl}(a)$ for some $1\leq k\not = l \leq n$. If $(k,l) =(j,i)$, there is nothing to prove. If $(k,l) \not = (j,i)$ then by (E), ${}^{e_{kl}(a)} e_{ij}(\alpha)$ is either $e_{ij}(\alpha)$ or $e_{i'j'}(\alpha')e_{ij}(\alpha)$ for an elementary matrix $e_{i'j'}(\alpha')$ such that $\alpha'\in I$.

Suppose $l(e)\ge 2$. Write $e = e' e_{mn}(b)e_{kl}(a)$, where $l(e') = l(e) - 2$. If $(k, l) \not = (j, i)$, then applying the paragraph above, one can finish by induction on $l(e)$. Suppose $(k, l) = (l, i)$. If $(m, n) = (i,j)$ then applying~(\ref{gbgdq1}), one can finish by induction on $l(e)$. Suppose $(m,n)\not = (i,j)$. If $m \not = i$ and $n \not = j$ then by (E) \[e_{mn}(b)e_{ji}(a) = e_{ji}(a) e_{mn}(b).\] It is not possible that $(m, n) = (j, i)$, because then it would follow that $e = e' e_{ji}(b + a)$ and thus, that $l(e)\leq  l(e')+ 1$. Since $(m, n)\not = (j, i)$, it follows from (E) that
${}^{e_{mn}(b)}e_{ij}(q)$ is either  $e_{ij}(\alpha)$ or $e_{i'j'}(\alpha')e_{ij}(\alpha)$, for an elementary matrix $e_{i'j'}(\alpha')$, where
$\alpha' \in  I$ and one is done again by induction on $l(e)$. There remain now two cases to check; namely, $(m, n) = (m,j)$ with $m \not = i$ and $(m, n) = (i, n)$ with $n \not = j$. In the first case,

\begin{align*}
{}^{e_{mj}(b) e_{ji}(a)} e_{ij}(\alpha)&=
{}^{e_{mi}(ba) e_{ji}(a) e_{mj}(b)} e_{ij}(\alpha)\\
&={}^{e_{mi}(ba) e_{ji}(a)} e_{ij}(\alpha)\\
&={}^{e_{ji}(a) e_{mi}(ba)} e_{ij}(\alpha)\\
&= {}^{e_{ji}(a)}(e_{mj}(ba\alpha)e_{ij}(\alpha)).
\end{align*}

Thus, one can finish by induction on $l(e)$. The second case is checked similarly.
\end{proof}

Using Lemma~\ref{Engenerator}, it is not hard to prove that $E(n,A, I^2)\le  E(n,I)$ (see~\cite[Corollary~4.9]{8}  and~\cite[Proposition~2]{96}).
This containment can be slightly generalised to the case of two ideals. This will be established in Lemma~\ref{Lem:Habdank} which  will be used throughout the paper.

The first step in the construction of algebraic $K$-theory was done
by Hyman Bass in \cite{14} almost half century ago. There is a standard embedding
\begin{equation}\label{pfsy12}
 \GL(n,A)\map\GL(n+1,A),\qquad g\mapsto
\left(\begin{array}{cc} g&0 \\ 0&1 \\ \end{array}\right), 
\end{equation}
called the stabilisation map, which allows us to identify
$\GL(n,A)$ with a subgroup in $\GL(n+1,A)$. Now we can consider
the stable general linear group
$$ \GL(A)=\varinjlim_n\GL(n,A), $$
\noindent which is the direct limit (effectively the union) of the
$\GL(n,A)$ under the stabilisation embeddings.

Since the stabilisation map sends $E(n,A)$ to
$E(n+1,A)$, we can define the stable elementary group $E(A)=
\varinjlim E(n,A)$.  This subgroup is called the (absolute) {\it elementary group} of
degree $n$ over $A$.

Applying the stabilisation embeddings to the families $\GL(n,A,I)$
and $E(n,A,I)$ generates stable versions $\GL(A,I)$ and $E(A,I)$,
respectively, of these groups. There is no stable version of
$C(n,A,I)$, though, as the stability map does not send $C(n,A,I)$
into $C(n+1,A,I)$. 

A crucial observation known as the Whitehead
lemma, asserts that modulo $E(A)$ the product of two matrices in
$\GL(n,A)$ is the same as their direct sum, and in particular,
$E(A)=[\GL(A),\GL(A)]$. Such identities in the stable case can be established easily, as there is enough room to arrange the matrices inside $\GL(A)$. For the pedagogical reason we include the proof of the following identity (see Lemma~\ref{whiteee})
\begin{equation*}
 E(A,I)=[E(A),E(A,I)]=[\GL(A),E(A,I)]=[\GL(A),\GL(A,I)]. 
 \end{equation*} 
The main theme of this note is to establish the non-stable identities of this type.  

First, we need some lemmas.

\begin{Lem}\label{triang5}
Let $A$ be a ring and $I$ be a two sided ideal of $A$. Any $n\times n$ upper/lower triangular matrix with $1$ on the main diagonal and elements of $I$ as non-zero entires belong to $E(n,I)$.  
\end{Lem}
\begin{proof}
Let $x$ be an upper triangular matrix with $1$ on the diagonal and elements of $I$ as non-zero entires, i.e., 
$x=(a_{ij})\in M_n(A)$ with $a_{ii}=1$, $1\leq i \leq n$ and $a_{ij} \in I$ for $j>i$. 
Then the matrix 
\begin{equation}\label{gbgtrd2}
x'=(a'_{ij})=xe_{12}(-a_{12})e_{23}(-a_{23})\dots e_{n-1,n}(-a_{n-1,n})
\end{equation}
is still upper triangular with $1$ on the main diagonals and $0$ on $j-i=1$. Note that since $a_{ij}\in I$, all the elementary matrices in~(\ref{gbgtrd2}) are in $E(n,I)$. 

Now the matrix 
\[x{''}=(a{''}_{ij})=x'e_{13}(-a'_{13})e_{24}(-a'_{24})\dots e_{n-2,n}(-a'_{n-2,n}),\]
is again upper triangular with $1$ on the main diagonals and $0$ on $j-i=1,2$. Here also $a'_{ij}\in I$ and so all the elementary matrices are in $E(n,I)$. 
Continuing in this fashion,  by induction, $x^{(n-1)}$ is the identity matrix. Note that all elementary matrices involved are in $E(n,I)$. It follows that $A\in E(n,I)$. The lower triangular case is similar.
\end{proof}

\begin{Lem}\label{Whitehead}
Let $A$ be an associative ring and let $I\trianglelefteq A$
be a two-sided ideal of $A$. Then for any $x,y\in\GL(n,A,I)$
one has
\begin{equation}\label{uhbd41}
\begin{pmatrix} xyx^{-1}y^{-1}&0\\ 0&1\\ \end{pmatrix} \in 
E(2n,A,I). 
\end{equation}
\end{Lem}
\begin{proof}
Following Bass~\cite[Lemma~1.7]{14}, we first show that 
\begin{equation}\label{uhbd412}
\begin{pmatrix} xy&0\\ 0&1\\ \end{pmatrix}\equiv
\begin{pmatrix} x&0\\ 0&y\\ \end{pmatrix}\pmod{E(2n,A,I)},
\end{equation}
and 
\begin{equation}\label{uhbd413}
\begin{pmatrix} yx&0\\ 0&1\\ \end{pmatrix}\equiv
\begin{pmatrix} x&0\\ 0&y\\ \end{pmatrix}\pmod{E(2n,A,I)}, 
\end{equation}
which then immediately implies~(\ref{uhbd41}).

Write $y=1+q$, where $q\in M_n(I)$. Furthermore, let 
\begin{align*}
\alpha=\begin{pmatrix} yx&0\\ 0&1\\ \end{pmatrix}, \beta=
\begin{pmatrix} x&0\\ 0&y\\ \end{pmatrix},\tau_1=\begin{pmatrix} 1&(yx)^{-1}q\\ 0&1\\ \end{pmatrix},\\
\tau_2=\begin{pmatrix} 1&-x^{-1}q\\ 0&1\\ \end{pmatrix},\tau_3=\begin{pmatrix} 1&0\\ -y^{-1}qx&1\\ \end{pmatrix},
\sigma=\begin{pmatrix} 1&0\\ x&1\\ \end{pmatrix}.
\end{align*}
By Lemma~\ref{triang5}, $\tau_1,\tau_2,\tau_3 \in E(2n,I)$, $\sigma \in E(2n,A)$  and thus by definition
$\sigma^{-1}\tau_2\sigma \in E(2n,A,I)$. We get $\tau:=\tau_1\sigma^{-1}\tau_2\sigma\tau_3 \in E(2n,A,I)$. 
Now a simple matrix calculation shows  
\begin{align*}
\alpha\tau_1=\begin{pmatrix} yx&q\\ 0&1\\ \end{pmatrix}, \alpha\tau_1\sigma^{-1}=\begin{pmatrix} yx-qa&q\\ -x&1\\ \end{pmatrix}=\begin{pmatrix} x&q\\ -x&1\\ \end{pmatrix},\\
\alpha\tau_1\sigma^{-1}\tau_2=\begin{pmatrix} x&-q+q\\ -x&1+q\\ \end{pmatrix}=\begin{pmatrix} x&0\\ -x&y\\ \end{pmatrix},\\
\alpha\tau_1\sigma^{-1}\tau_2\sigma=\begin{pmatrix} x&0\\ yx-x&y\\ \end{pmatrix}=
\begin{pmatrix} x&0\\ qx&b\\ \end{pmatrix}.
\end{align*}
Finally
\begin{equation*}
\alpha\tau=\alpha\tau_1\sigma^{-1}\tau_2\sigma\tau_3=\begin{pmatrix} x&0\\ 0&y\\ \end{pmatrix}=\beta
\end{equation*}
This shows the Identity~(\ref{uhbd413}). Plugging $x=y^{-1}$ into this identity we obtain
\[\left(\begin{array}{cc} y^{-1}&0 \\ 0&y \\ \end{array}\right) \in E(2n,A,I).\]
Thus 
\begin{equation*}
\begin{pmatrix} xy&0\\ 0&1\\ \end{pmatrix}\equiv \begin{pmatrix} xy&0\\ 0&1\\ \end{pmatrix} \begin{pmatrix} y^{-1}&0\\ 0&y\\ \end{pmatrix}\equiv \begin{pmatrix} x&0\\ 0&y\\ \end{pmatrix}, 
\end{equation*}
which is Identity~(\ref{uhbd412}). 
\end{proof}


\begin{Lem}\label{whiteee}
For an associative ring $A$ and an ideal $I\unlhd A$ one has
\begin{equation}\label{fluck1}
 E(A,I)=[E(A),E(A,I)]=[\GL(A),E(A,I)]=[\GL(A),\GL(A,I)]. 
 \end{equation}
\end{Lem}
\begin{proof}
The elements of $E(A,I)$ are generated by $xe_{ij}(\alpha)x^{-1}$, where $e_{ij}(\alpha)\in E(I)$ and $ x\in E(A)$. Writing
\[xe_{ij}(\alpha)x^{-1}=[x,e_{ij}(\alpha)]e_{ij}(\alpha)=[x,e_{ij}(\alpha)][e_{ik}(1),e_{kj}(\alpha)] \in [E(A),E(A,I)],\] it follows that 
\[E(A,I) \le   [E(A),E(A,I)].\] Thus we have 
\begin{equation*}
 E(A,I)\le  [E(A),E(A,I)] \le  [\GL(A),E(A,I)] \le  [\GL(A),\GL(A,I)]. 
 \end{equation*}
 We show $[\GL(A),\GL(A,I)] \le  E(A,I)$. 
 Let $x\in \GL(A)$ and $y\in \GL(A,I)$. Then for a sufficiently large $n$,
$x\in \GL(n,A)$ and $y \in \GL(n,A,I)$. By Lemma~\ref{Whitehead}, 
\begin{equation*}
\begin{pmatrix} xyx^{-1}y^{-1}&0\\ 0&1\\ \end{pmatrix} \in 
E(2n,A,I) \le  E(A,I). 
\end{equation*}
This finishes the proof. 
\end{proof}




At this point Bass defines
$$ K_1(A)=\GL(A)/E(A)=\GL(A)/[\GL(A),\GL(A)] $$
\noindent as the abelianisation of $\GL(A)$. Indeed algebraic
$K$-theory was born as Bass observed that the functors $K_0$ and
$K_1$ together with their relative versions fit into a unified
theory with important applications in algebra, algebraic geometry
and number theory. In the same manner, 
the relative $K_1$-functor of a pair $(A,I)$ defined as 
$$ K_1(A,I)=\GL(A,I)/E(A,I). $$
\noindent

As one of important applications in algebra, Bass \cite{14}
relates the normal subgroup structure of $\GL(A)$ to the ideal
structure of $A$. This leap in generality is  considered
as the starting point of the modern theory of linear groups.
\begin{The}
Let $A$ be an arbitrary associative ring and $H\le\GL(A)$ be
a subgroup normalised by the elementary group $E(A)$. Then
there exists a unique ideal $I\unlhd A$ such, that
$$ E(A,I)\le H\le\GL(A,I). $$
\noindent Conversely, any subgroup $H$ satisfying these inclusions
is (by Lemma~\ref{whiteee}) normal in $\GL(A)$.
\end{The}
Quite remarkably this result holds
for arbitrary associative rings. Thus, an explicit enumeration of
all normal subgroups of $\GL(A)$ amounts to the calculation of
$K_1(A,I)$ for all ideals $I$ in $A$.
\par
The group $K_1$ answers essentially the question as to how far
$\GL(n,A)$ falls short of being spanned by elementary generators.
A few years later Milnor \cite{milnor70}, \cite{milnor},
building on the work of Steinberg \cite{steinberg62},
\cite{steinberg67}
and Moore \cite{moore}, introduced
the group $K_2$, which measures essentially to which extent all
relations among elementary generators follow from the obvious ones.

For any associative ring $A$, a two-sided ideal $I\unlhd A$
and a fixed $n$ we consider the quotient
$$ K_1(n,A,I)=\GL(n,A,I)/E(n,A,I). $$
\noindent
In general, the elementary subgroup $E(n,R,I)$ does not have
to be normal in the congruence subgroup $\GL(n,A,I)$. In 
particular, $K_1(n,A,I)$ is a pointed set, rather than a group.
However, we will see when $A$ is quasi-finite 
and $n\ge 3$, the $K_1(n,A,I)$ is a group. 
Similarly, we
define 
$$ \SK_1(n,A,I)=\SL(n,A,I)/E(n,A,I), $$
\noindent
consult~\cite{8} for the definition of $\SL(n,A,I)$
for quasi-finite rings.
\par
The stability embedding of the general linear groups (see~(\ref{pfsy12}))
sends $E(n,A,I)$ inside $E(n+1,A,I)$. In particular, by the
homomorphism theorem it induces stability map
$$ \psi_n:K_1(n,A,I)\map K_1(n+1,A,I), $$
which is a group homomorphism when both sides are groups. 
Clearly, $\psi_n$ restricts to a map between $\SK_1(n,A,I)$'s.
\par
The following results, known as the surjective and injective
stability for $K_1$ are due to Bass and to Bass---Vaserstein,
respectively.

\begin{Lem}\label{BassBV}
Let $A$ be an associative ring and let $I\trianglelefteq A$
be a two-sided ideal of $A$. Consider the stability map 
\begin{equation*}
\psi_n: K_1(n,A,I)\map K_1(n+1,A,I). 
\end{equation*}
Then 
\begin{enumerate}
\item 
If $n\ge\sr(A)$ , then $\psi_n$ is surjective. In other words
$$ \GL(n+1,A,I)=\GL(n,A,I)E(n+1,A,I). $$

\item If $n\ge\sr(A)+1$, then $\psi_n$ is injective. In other words
$$ \GL(n,A,I)\cap E(n+1,A,I)=E(n,A,I). $$
\end{enumerate}
\end{Lem}

\section{Unitary Groups}\label{sec3}

The notion of $\Lambda$-quadratic forms, quadratic modules and generalised
unitary groups over a form ring $(A,\Lambda)$ were introduced by Anthony
Bak in his Thesis who studied
their $K$-theory (see~\cite{B1,B2}).

Although the quadratic setting
is much more complicated than the linear one, it is being gradually
established that most results
concerning the $K$-theory of general linear groups can be carried over to the
$K$-theory of general quadratic groups.

In this section we
{\it briefly\/} review the most fundamental notation and results
that will be constantly used in the present paper. We refer to
\cite{B2,32,55,13,33,41,tang,lavrenov} for details, proofs, and
further references.


\subsection{}\label{form algebra}
Let $R$ be a commutative ring with $1$, and $A$ be an (not necessarily
commutative) $R$-algebra. An involution, denoted by $\bar{\phantom{\a}}$, 
is an
anti-homomorphism of $A$ of order $2$. Namely, for $\alpha,\beta\in A$,
one has $\overline{\alpha+\beta}=\bar\alpha+\bar\beta$, \
$\overline{\alpha\beta}=\bar\beta\bar\alpha$ and $\bar{\bar\alpha}=\alpha$.
Fix an element $\lambda\in\Cent(A)$ such that $\lambda\bar\lambda=1$. One may
define two additive subgroups of $A$ as follows:
$$ \Lambda_{\min}=\{\alpha-\lambda\bar\alpha\mid\alpha\in A\}, \qquad
\Lambda_{\max}=\{\alpha\in A\mid \alpha=-\lambda\bar\alpha\}. $$
\noindent
A {\em form parameter} $\Lambda$ is an additive subgroup of $A$ such that
\begin{itemize}
\item[(1)] $\Lambda_{\min}\le \Lambda\le \Lambda_{\max}$,
\smallskip
\item[(2)] $\alpha\Lambda\bar\alpha\le \Lambda$ for all $\alpha\in A$.
\end{itemize}
The pair $(A,\Lambda)$ is called a {\em form ring}.


\subsection{}\label{form ideals}
Let $I\unlhd A$ be a two-sided ideal of $A$. We assume $I$ to be
involution invariant, i.e., such that $\bar I=I$. Set
$$ \Gamma_{\max}(I)=I\cap \Lambda, \qquad
\Gamma_{\min}(I)=\{\xi-\lambda\bar\xi\mid\xi\in I\}+
\langle\xi\alpha\bar\xi\mid \xi\in I,\alpha\in\Lambda\rangle. $$
\noindent
A {\em relative form parameter} $\Gamma$ in $(\FormR)$ of level $I$ is an
additive group of $I$ such that
\begin{itemize}
\item[(1)] $\Gamma_{\min}(I)\le  \Gamma \le \Gamma_{\max}(I)$,
\smallskip
\item[(2)] $\alpha\Gamma\bar\alpha\le  \Gamma$ for all $\alpha\in A$.
\end{itemize}
The pair $(I,\Gamma)$ is called a {\em form ideal}.
\par
In the level calculations we will use sums and products of form
ideals. Let $(I,\Gamma)$ and $(J,\Delta)$ be two form ideals. Their sum
is artlessly defined as $(I+J,\Gamma+\Delta)$, it is immediate to verify
that this is indeed a form ideal.
\par
Guided by analogy, one is tempted to set
$(I,\Gamma)(J,\Delta)=(IJ,\Gamma\Delta)$. However, it is considerably
harder to correctly define the product of two relative form parameters.
The papers~\cite{30,31,33} introduce the following definition
$$ \Gamma\Delta=\Gamma_{\min}(IJ)+{}^J\Gamma+{}^I\Delta, $$
\noindent
where
$$ {}^J\Gamma=\big\langle \xi\Gamma\bar\xi\mid \xi\in J\big\rangle,\qquad
{}^I\Delta=\big\langle \xi\Delta\bar\xi\mid \xi\in I\big\rangle. $$
\noindent
One can verify that this is indeed a relative form parameter of level $IJ$
if $IJ=JI$.
\par
However, in the present paper we do not wish to impose any such
commutativity assumptions. Thus, we are forced to consider the
symmetrised products
$$ I\circ J=IJ+JI,\qquad
\Gamma\circ\Delta=\Gamma_{\min}(IJ+JI)+{}^J\Gamma+{}^I\Delta\big. $$
\noindent
The notation $\Gamma\circ\Delta$ -- as also $\Gamma\Delta$ is
slightly misleading, since in fact it depends on $I$ and $J$, not
just on $\Gamma$ and $\Delta$. Thus, strictly speaking, one should
speak of the symmetrised products of {\it form ideals\/}
$$ (I,\Gamma)\circ (J,\Delta)=
\big(IJ+JI,\Gamma_{\min}(IJ+JI)+{}^J\Gamma+{}^I\Delta\big). $$
\noindent
Clearly, in the above notation one has
$$ (I,\Gamma)\circ (J,\Delta)=(I,\Gamma)(J,\Delta)+(J,\Delta)(I,\Gamma). $$


\subsection{}\label{quasi-finite}
A {\em form algebra over a commutative ring $R$} is a form ring $(A,\Lambda)$,
where $A$ is an $R$-algebra and the involution leaves $R$ invariant, i.e.,
$\bar R=R$. A form algebra $(\FormR)$ is called {\it module finite}, if $A$ is
finitely generated as an $R$-module.
 A form algebra $(\FormR)$ is called {\it quasi-finite},
if there is a direct system of module finite  $R$-subalgebras $A_i$ of
$A$ such that $\varinjlim A_i=A$ (see~\S\ref{sub:1.2}).

In general $\Lambda$ is not an $R$-module. This forces us to
replace $R$ by its subring $R_0$, generated by all $\alpha\bar\alpha$
with $\alpha\in R$. Clearly, all elements in $R_0$ are invariant with
respect to the involution, i.~e.\ $\bar r=r$, for $r\in R_0$.
\par
It is immediate, that any form parameter $\Lambda$ is an $R_0$-module.
This simple fact will be used throughout. This is precisely why we have
to localise in multiplicative subsets of $R_0$, rather than in those of
$R$ itself (see~\S\ref{hgbdg43a}).



We now recall the basic notation and facts related to
Bak's generalised unitary groups and their elementary subgroups.


\subsection{}\Label{general}
Let, as above, $A$ be an associative ring with 1. For natural $m,n$
we denote by $M(m,n,A)$ the additive group of $m\times n$ matrices
with entries in $A$. In particular $M(m,A)=M(m,m,A)$ is the ring of
matrices of degree $m$ over $A$. For a matrix $x\in M(m,n,A)$ we
denote by $x_{ij}$, $1\le i\le m$, $1\le j\le n$, its entry in the
position $(i,j)$. Let $e$ be the identity matrix and $e_{ij}$,
$1\le i,j\le m$, be a standard matrix unit, i.e.\ the matrix which has
1 in the position $(i,j)$ and zeros elsewhere.
\par
As usual, $\GL(m,A)=M(m,A)^*$ denotes the general linear group
of degree $m$ over $A$. The group $\GL(m,A)$ acts on the free right
$A$-module $V\cong A^{m}$ of rank $m$. Fix a base $e_1,\ldots,e_{m}$
of the module $V$. We may think of elements $v\in V$ as columns with
components in $A$. In particular, $e_i$ is the column whose $i$-th
coordinate is 1, while all other coordinates are zeros.
\par
In the unitary setting, we are only interested in the case,
when $m=2n$ is even. We usually number the base
as follows: $e_1,\ldots,e_n,e_{-n},\ldots,e_{-1}$. All other
occurring geometric objects will be numbered accordingly. Thus,
we write \[v=(v_1,\ldots,v_n,v_{-n},\ldots,v_{-1})^t,\] where $v_i\in A,$
for vectors in $V\cong A^{2n}$.
\par
The set of indices will be always ordered accordingly,
$\Omega=\{1,\ldots,n,-n,\ldots,-1\}$. Clearly, $\Omega=\Omega^+\sqcup\Omega^-$,
where $\Omega^+=\{1,\ldots,n\}$ and $\Omega^-=\{-n,\ldots,-1\}$. For an
element $i\in\Omega$ we denote by $\e(i)$ the sign of $\Omega$, i.e.\
$\e(i)=+1$ if $i\in\Omega^+$, and $\e(i)=-1$ if $i\in\Omega^-$.

\subsection{}\Label{unitary} For a form ring $(\FormR)$, one considers the
{\it hyperbolic unitary group\/} $\GU(2n,\FormR)$, see~\cite[\S2]{13}.
This group is defined as follows:
\par
One fixes a symmetry $\lambda\in\Cent(A)$, $\lambda\bar\lambda=1$ and
supplies the module $V=A^{2n}$ with the following $\lambda$-hermitian form
$h:V\times V\map A$,
$$ h(u,v)=\bar u_1v_{-1}+\ldots+\bar u_nv_{-n}+
\lambda\bar u_{-n}v_n+\ldots+\lambda\bar u_{-1}v_1. $$
\noindent
and the following $\Lambda$-quadratic form $q:V\map A/\Lambda$,
$$ q(u)=\bar u_1 u_{-1}+\ldots +\bar u_n u_{-n} \mod\Lambda. $$
\noindent
In fact, both forms are engendered by a sesquilinear form $f$,
$$ f(u,v)=\bar u_1 v_{-1}+\ldots +\bar u_n v_{-n}. $$
\noindent
Now, $h=f+\lambda\bar{f}$, where $\bar f(u,v)=\bar{f(v,u)}$, and
$q(v)=f(u,u)\mod\Lambda$.
\par
By definition, the hyperbolic unitary group $\GU(2n,A,\Lambda)$ consists
of all elements from $\GL(V)\cong\GL(2n,A)$ preserving the $\lambda$-hermitian
form $h$ and the $\Lambda$-quadratic form $q$. In other words, $g\in\GL(2n,A)$
belongs to $\GU(2n,A,\Lambda)$ if and only if
$$ h(gu,gv)=h(u,v)\quad\text{and}\quad q(gu)=q(u),\qquad\text{for all}\quad u,v\in V. $$
\par
When the form parameter is not maximal or minimal, these groups are not
algebraic. However, their internal structure is very similar to that
of the usual classical groups. They are also often times called {\it general
quadratic groups}, or {\it classical-like groups}.

The groups introduced by Bak in his Thesis~\cite{B1} gather all {\it even\/}
classical groups under one umbrella. Linear groups, symplectic
groups, (even) orthogonal groups, (even) classical unitary groups,
are all special cases of his construction. Not only that, Bak's
construction allows to introduce a whole new range of {\it classical
like groups\/}, taking into account hybridisation, defect groups,
and other such phenomena in characteristic 2, which before~\cite{B1}  were
considered pathological, and required separate analysis outside of
the general theory. 

\par
To give the idea of how it works, let us illustrate how Bak's
construction specialises in the case of hyperbolic groups.
\par\smallskip
$\bullet$ In the case when involution is trivial, $\lambda=-1$,
$\Lambda=\Lambda_{\maxx}=R$, one gets the split symplectic group
$G(2n,R,\Lambda)=\Sp(2n,R)$.
\par\smallskip
$\bullet$ In the case when involution is trivial, $\lambda=1$,
$\Lambda=\Lambda_{\minn}=0$, one gets the split even orthogonal
group $G(2n,R,\Lambda)=O(2n,R)$.
\par\smallskip
$\bullet$ In the case when involution is non-trivial, $\lambda=-1$,
$\Lambda=\Lambda_{\maxx}$, one gets the classical quasi-split even
unitary group $G(2n,R,\Lambda)=U(2n,R)$.
\par\smallskip
$\bullet$ Let $R^o$ be the ring opposite to $R$ and
$R^e=R\oplus R^o$. Define an involution on $R^e$ by
$(x,y^o)\mapsto(y,x^o)$ and set $\lambda=(1,1^o)$.
Then there is a unique form parameter $\Lambda=\{(x,-x^o)\mid x\in R\}$.
The resulting unitary group
$$ G(2n,R^e,\Lambda)=\{(g,g^{-t})\mid g\in\GL(n,R)\} $$
\noindent
may be identified with the general linear group $\GL(n,R)$.
\par\smallskip
Thus, in particular the hyperbolic unitary groups cover Chevalley
groups of types $A_l$, $C_l$ and $D_l$. 

\subsection{}\Label{elementary1}
{\it Elementary unitary transvections\/} $T_{ij}(\xi)$
correspond to the pairs $i,j\in\Omega$ such that $i\neq j$. They come in
two stocks. Namely, if, moreover, $i\neq-j$, then for any $\xi\in A$ we set
$$ T_{ij}(\xi)=e+\xi e_{ij}-\lambda^{(\e(j)-\e(i))/2}\bar\xi e_{-j,-i}. $$
\noindent
These elements are also often called {\it elementary short root unipotents\/}.
\noindent
On the other side for $j=-i$ and $\a\in\lambda^{-(\e(i)+1)/2}\Lambda$ we set
$$ T_{i,-i}(\a)=e+\a e_{i,-i}. $$
\noindent
These elements are also often called {\it elementary long root elements\/}.
\par
Note that $\bar\Lambda=\bar\lambda\Lambda$. In fact, for any element
$\a\in\Lambda$ one has $\bar\a=-\bar\lambda\a$ and thus $\bar\Lambda$ coincides with
the set of products $\bar\lambda\a$, $\a\in\Lambda$. This means that in the
above definition $\a\in\bar\Lambda$ when $i\in\Omega^+$ and $\a\in\Lambda$
when $i\in\Omega^-$.
\par
Subgroups $X_{ij}=\{T_{ij}(\xi)\mid \xi\in A\}$, where $i\neq\pm j$, are
called {\it short root subgroups\/}. Clearly, $X_{ij}=X_{-j,-i}$.
Similarly, subgroups $X_{i,-i}=\{T_{ij}(\a)\mid
\a\in\lambda^{-(\e(i)+1)/2}\Lambda\}$ are called {\it long root subgroups\/}.
\par
The {\it elementary unitary group\/} $\EU(2n,\FormR)$ is generated by
elementary unitary transvections $T_{ij}(\xi)$, $i\neq\pm j$, $\xi\in A$,
and $T_{i,-i}(\a)$, $\a\in\lambda^{-(\e(i)+1)/2}\Lambda$, see~\cite[\S3]{13}.

\subsection{}\Label{elementary2}
Elementary unitary transvections $T_{ij}(\xi)$ satisfy the following
{\it elementary relations\/}, also known as {\it Steinberg relations\/}.
These relations will be used throughout this paper.
\par\smallskip
(R1) \ $T_{ij}(\xi)=T_{-j,-i}(-\lambda^{(\varepsilon(j)-\varepsilon (i))/2}\bar{\xi})$,
\par\smallskip
(R2) \ $T_{ij}(\xi)T_{ij}(\zeta)=T_{ij}(\xi+\zeta)$,
\par\smallskip
(R3) \ $[T_{ij}(\xi),T_{hk}(\zeta)]=e$, where $h\ne j,-i$ and $k\ne i,-j$,
\par\smallskip
(R4) \ $[T_{ij}(\xi),T_{jh}(\zeta)]=
T_{ih}(\xi\zeta)$, where $i,h\ne\pm j$ and $i\ne\pm h$,
\par\smallskip
(R5) \ $[T_{ij}(\xi),T_{j,-i}(\zeta)]=
T_{i,-i}(\xi\zeta-\lambda^{-\varepsilon(i)}\bar{\zeta}\bar{\xi})$, where $i\ne\pm j$,
\par\smallskip
(R6) \ $[T_{i,-i}(\alpha),T_{-i,j}(\xi)]=
T_{ij}(\alpha\xi)T_{-j,j}(-\lambda^{(\ep(j)-\ep(i))/2}\bar\xi\alpha\xi)$,
where $i\ne\pm j$.
\par\smallskip
Relation (R1) coordinates two natural parameterisations of the same short
root subgroup $X_{ij}=X_{-j,-i}$. Relation (R2) expresses additivity of
the natural parameterisations. All other relations are various instances
of the Chevalley commutator formula. Namely, (R3) corresponds to the
case, where the sum of two roots is not a root, whereas (R4), and (R5)
correspond to the case of two short roots, whose sum is a short root,
and a long root, respectively. Finally, (R6) is the Chevalley commutator
formula for the case of a long root and a short root, whose sum is a root.
Observe that any two long roots are either opposite, or orthogonal, so
that their sum is never a root.

\subsection{}
There is a standard embedding
$$ G(2n,A,\Lambda)\map G\Big(2(n+1),A,\Lambda\Big),\quad
\left(\begin{array}{cc}a&b\\c&d\\ \end{array}\right)
\mapsto
\left(\begin{array}{cccc}
a&0&0&b\\
0&1&0&0\\
0&0&1&0\\
c&0&0&d\\
\end{array}\right)$$
\noindent called the stabilisation map. In fact some other sources
may give a slightly different picture of the right hand side. How the right hand side  exactly looks, depends on the ordered basis we
choose. With the ordered basis which is used in~\cite{B2}, the standard embedding has the form
$$ G(2n,R,\Lambda) \map G\Big(2(n+1),R,\Lambda\Big),\quad
 \left(\begin{array}{cc}
  a & b \\
  c & d \\
\end{array}\right)
\mapsto \left(
\begin{array}{cccc}
  a & 0 & b& 0\\
  0 & 1 & 0& 0\\
  c & 0 & d& 0\\
  0 & 0 & 0& 1\\
\end{array}\right).$$

Define
$$G\formr=\varinjlim_n G(2n,R,\Lambda)$$ and $$E\formr=
\varinjlim_n E(2n,R,\Lambda).$$ The groups $G(I,\gam)$ and
$E(I,\gam)$ are defined similarly.

One could ask, whether one can carry over Bass' results discussed in~\S\ref{sec22} to the unitary case? Bak, and in a slightly narrower situation, Vaserstein, established unitary versions of Whitehead's lemma, which
in particular implies the following result.
\begin{The}
Let $\formr$ be an arbitrary form ring, and $\formi$
be its form ideal, then
$$ E\formi=[E\formr,E\formi]=[G\formr,E\formi]=[G\formr,G\formi]. $$
\end{The}
Now, similarly to the linear case, this allows one to introduce the
{\it unitary $K$-functor}
$$ K_1\formi=G\formi/E\formi. $$
\noindent
A version of unitary $K$-theory modelled upon the unitary
groups has been systematically developed by Bass in~\cite{bass73}. Note that, in some literature, the notation $\KU$ is used to denote the unitary $K$-groups. In other literature, the functor above is called a {\it quadratic $K$-functor} and the notation $\operatorname{KQ}$ is used. (For a lexicon of notations, see~\cite[\S14]{B2}).

As another piece of structure, parallel to the linear situation, let
us mention the description of normal subgroups in $G\formr$, that
holds over an arbitrary ring.
\begin{The}\label{GU-sandwich}
Let $\formr$ be an arbitrary form ring. If $H\le G\formr$
is a subgroup normalised by $E\formr$, then for a unique form
ideal $\formi$, one has
$$ E\formi\le H\le G\formi $$
\noindent
Conversely, these inclusions guarantee that $H$ is automatically
normal in $G\formr$.
\end{The}





\section{Towards non-stable $K$-theory}\label{hyhyhy}

One of the major contributions toward non-stable $K$-theory of rings is the work of Suslin~\cite{90,TUL}.  He proved that if $A$ is a  module finite ring, namely, a ring
that is finitely generated as a module over its center, and $n\ge 3$, 
then $E(n, A)$ is a normal subgroup of $\GL(n,A)$. Therefore the non-stable $K_1$-group, i.e., $\GL(n,A)/E(n,A)$,  is well-defined.  
Later, Borevich and Vavilov \cite{borvav} and Vaserstein~\cite{98}, building on Suslin's method,    
established the {\it standard commutator formula}:

\begin{The}[Suslin, Borevich-Vavilov, Vaserstein]\label{standard}
Let $A$ be a module finite ring, $I$ a two-sided ideal of $A$ 
 and $n\ge 3$. Then $E(n,A,I)$ is normal in $\GL(n,A)$, i.e., 
 \[
\big[E(n,A,I),\GL(n,A)\big]=E(n,A,I).
\] 
 Furthermore,
\[
\big[E(n,A),\GL(n,A,I)\big]=E(n,A,I).
\] 
\end{The}

One natural question that arises here is whether one has a ``finer'' mixed commutator formula involving two ideals. In fact this had already been established by Bass for general linear groups of degrees sufficiently larger than the stable rank, when he proved his celebrated classification of subgroups of $\GL_n$ normalized by $E_n$ (see \cite[Theorem~4.2]{14}):

\begin{The}[Bass]\label{hhggff}
Let $A$ be a ring, $I,J$ two-sided ideals of $A$ and $n \geq \max(\sr(A)+1,3)$. Then 
\[ \big [E(n,A,I),\GL(n,A,J) \big ]=\big [E(n,A,I),E(n,A,J) \big ]. \]
\end{The}

Later, Mason and Stothers, building on Bass' result, proved (\cite[Theorem~3.6, Corollary~3.9]{72}, and~\cite[Theorem~1.3]{70}):

\begin{The}[Mason-Stothers] \label{mat73}
Let $A$ be a ring, $I,J$ two-sided ideals of $A$ and $n \geq \max(\sr(A)+1,3)$. Then 
\[ \big [\GL(n,A,I),\GL(n,A,J) \big ]=\big [E(n,A,I),E(n,A,J) \big ]. \]
\end{The}

As the Bass Theorem~\ref{hhggff} and the Mason and Stothers 
Theorem~\ref{mat73} are the starting point of this paper, below we present a new 
proof of Theorem~\ref{mat73}. 

\begin{Lem}\label{MS-1}
For any $n\ge 1$ one has
$$ [\GL(n,A,I),\GL(n,A,J)]\le[\GL(n,A,I),E(2n,A,J)]. $$
\end{Lem}
\begin{proof}
Indeed, if $x\in\GL(n,A,I)$ and $y\in\GL(n,A,J)$. By Whitehead
lemma one has
$$ y=\begin{pmatrix} y&0\\ 0&e\\ \end{pmatrix}\equiv
\begin{pmatrix} e&0\\ 0&y\\ \end{pmatrix}\pmod{E(2n,A,J)}. $$
\noindent
Since $E(2n,A,J)$ is normal in $\GL(2n,A,J)$, one has
$$ y=\begin{pmatrix} y&0\\ 0&e\\ \end{pmatrix}=
\begin{pmatrix} e&0\\ 0&y\\ \end{pmatrix}z, $$
\noindent
for some $z\in E(2n,A,J)$. Since the first factor on
the right commutes with $x=\begin{pmatrix} x&0\\ 0&e\\ \end{pmatrix}$
one has $[x,y]=[x,z]$, as claimed.
\end{proof}
\begin{Lem}\label{MS-2}
For any $n\ge \max(\sr(A)+1,3)$ the stability map
\begin{multline*}
[E(n,A,I),E(n,A,J)]/E(n,A,IJ+JI)\map\\
[E(n+1,A,I),E(n+1,A,J)]/E(n+1,A,IJ+JI)
\end{multline*}
\noindent
is an isomorphism.
\end{Lem}
\begin{proof}
Clearly,
\begin{multline*}
[E(n,A,I),E(n,A,J)]/E(n,A,IJ+JI)\le \\
\GL(n,A,IJ+JI)/E(n,A,IJ+JI)=K_1(n,A,IJ+JI).
\end{multline*}
\noindent
By Theorem~\ref{gen-2},  $[E(n,A,I),E(n,A,J)]$ is
generated by $[E(2,A,I),E(2,A,J)]$ as a normal subgroup of
$\GL(n,A)$. Since $K_1(n,A,IJ+JI)$ is central in the quotient
$\GL(n,A)/E(n,A,IJ+JI)$, for $n\ge\sr(A)$, the stability map
is surjective and becomes an isomorphism one step further,
when the stability map
$$ K_1(n,A,IJ+JI)\map K_1(n+1,A,IJ+JI) $$
\noindent
becomes an isomorphism by Lemma~\ref{BassBV}.
\end{proof}

\begin{LemA}\label{Lem:Habdank}
Let $A$ be a ring and $I, J$ be two-sided ideals of $A$. Then 
\begin{multline*}
E(n,A, IJ+JI)\le  \big[E(n,I),E(n,J) \big] \le  \big[E(n,A,I),E(n,A,J)\big] \le  \\ \big[E(n,A,I),\GL(n,A,J)\big] \le  \big[\GL(n,A,I),\GL(n,A,J)\big] \le    \GL(n,A,IJ+JI).
\end{multline*}
\end{LemA}
\begin{proof}
We first show 
\begin{equation}\label{shahrivar}
E(n,A, IJ+JI)\le  \big[E(n,I),E(n,J) \big].
\end{equation}
By Lemma~\ref{Engenerator}, let ${}^{e_{i,j}(a)}e_{j,i}(\beta)$ be a generator of $E(n,A, IJ+JI)$, where $a\in A$ and $\beta \in IJ+JI$. It suffices to show that ${}^{e_{i,j}(a)}e_{j,i}(\alpha \beta) \in  \big[E(n,I),E(n,J) \big]$, where $\alpha \in I$ and $\beta \in J$. 
Using (E3), we have 
\begin{multline*}
{}^{e_{i,j}(a)}e_{j,i}(\alpha \beta)={}^{e_{i,j}(a)}\Big[e_{j,k}(\alpha),e_{k,i}(\beta)\Big]=\\\Big[{}^{e_{i,j}(a)}e_{j,k}(\alpha),{}^{e_{i,j}(a)}e_{k,i}(\beta)\Big]=
\Big[[e_{i,j}(a),e_{j,k}(\alpha)]e_{j,k}(\alpha),e_{k,i}(\beta)[e_{k,i}(-\beta),e_{i,j}(a)]\Big]=\\
\Big[e_{i,k}(a\alpha)e_{j,k}(\alpha),e_{k,i}(\beta)e_{k,j}(-\beta a)\Big ]\in \big[E(n,I),E(n,J) \big].
\end{multline*}
This shows~(\ref{shahrivar}). We are left to show that  
\begin{equation}\label{hgtedis}
\big[\GL(n,A,I),\GL(n,A,J)\big] \le   \GL(n,A,IJ+JI).
\end{equation}

Let $x \in \GL(n, A, I)$ and  $y \in \GL(n, A, J)$. Then $x = e + x_1$ and $x^{-1} = e + x_2$ for some 
$x_1, x_2 \in M(n, I)$ such that $x_1 + x_2 + x_1x_2 = 0$. Similarly, $y = e + y_1$ and $y^{-1} = e + y_2$ for some $y_1, y_2 \in M(n, J)$ such that $y_1 + y_2 + y_1y_2 = 0$. Then the following equality holds modulo $IJ+JI$. 
\[[x,y] = (e+x_1)(e+y_1)(e+x_2)(e+y_2)=e+x_1 +x_2 +x_1x_2 +y_1 +y_2 +y_1y_2 = e\] 
which proves~(\ref{hgtedis}).
\end{proof}

A stable version of Lemma~\ref{Lem:Habdank} implies that 
\[[E(R,I),E(R,J)]/E(R,I\circ J)\] lives inside $K_1(R,I\circ J)$.

\begin{proof}[Proof of Theorem~\ref{mat73}]
By Lemma~\ref{MS-1} one has
\begin{multline*}
E(n,A,IJ+JI)\le [E(2n,A,I),E(2n,A,J)]\le\\
[\GL(n,A,I),\GL(n,A,J)]\le[\GL(n,A,I),E(2n,A,J)]\le\\
[\GL(2n,A,I),E(2n,A,J)]=[E(2n,A,I),E(2n,A,J)].
\end{multline*}
\noindent
By Lemma~\ref{Lem:Habdank} one has $[\GL(n,A,I),\GL(n,A,J)]\le\GL(n,R,IJ+JI)$.
On the other hand, by Lemma~\ref{MS-2}
$$ [E(2n,A,I),E(2n,A,J)]\cap\GL(n,R,IJ+JI)\le [E(n,A,I),E(n,A,J)],  $$
\noindent
so that $[\GL(n,A,I),\GL(n,A,J)]=[E(n,A,I),E(n,A,J)]$, as claimed.
\end{proof}

There are (counter)examples that the Mason-Stothers Theorem does not hold for an arbitrary module finite ring~\cite{8}. However, recently Stepanov and Vavilov~\cite{112} proved Bass' Theorem~\ref{hhggff} for any commutative ring and $n\geq 3$. The authors, using Bak's localisation and patching method, extended the theorem to all module finite rings~\cite{46}. Then 
in~\cite{114}, using the Hall-Witt identity, a very short proof for this theorem was found. We include this proof here. We refer to Bass' Theorem in this setting as the {\it generalised commutator formula}.

\begin{TheA}[Generalized commutator formula]\label{gcformula1}
Let $A$ be a module finite $R$-algebra and $I,J$  be two-sided ideals of $A$. Then 
\[ \big [E(n,A,I),\GL(n,A,J) \big ]=\big [E(n,A,I),E(n,A,J) \big ]. \]
\end{TheA}
\begin{proof}
We first prove 
\begin{equation}\label{gdhsja}
\big [E(n,A,I),\GL(n,A,J) \big ] \le  \big [E(n,A,I),E(n,A,J) \big ].
\end{equation}
Writing $E(n,A,I)=\big[E(n,A),E(n,A,I) \big ]$ by Theorem~\ref{standard} and then using the three subgroup lemma,  
i.e., $\big [ [F,H],L\big]\le  \big [ [F,L ],H\big ] \big[F, [H,L ]\big]$ for three normal subgroups $F,H$ and $L$ of a group $G$, we have 
\begin{multline*}
\big [E(n,A,I),\GL(n,A,J) \big ]  =\Big [ \big[E(n,A),E(n,A,I) \big ], \GL(n,A,J) \Big ]  \le \\
 \Big[ \big [E(n,A),\GL(n,A,J)\big ], E(n,A,I) \Big] \Big [E(n,A), \big [E(n,A,I),\GL(n,A,J)\big ]\Big].
\end{multline*}
But using Theorem~\ref{standard}, $\Big[ \big [E(n,A),\GL(n,A,J)\big ], E(n,A,I) \Big]=\big [E(n,A,I),E(n,A,J) \big ]$. On the other hand using Theorem~\ref{Lem:Habdank} twice, along with Theorem~\ref{standard} again, we get 
\begin{multline*}
\Big [E(n,A), \big [E(n,A,I),\GL(n,A,J)\big ]\Big]  \le  \big [E(n,A), \GL(n,A,IJ+JI)\big] \le  \\
  E(n,A,IJ+JI) \le  \big [E(n,A,I),E(n,A,J) \big ]. 
\end{multline*}
The inclusion~(\ref{gdhsja}) now follows. The opposite inclusion is obvious.
\end{proof}

In the similar manner one can establish the generalised commutator formula in the setting of unitary groups and Chevalley  groups. Again, in these setting the calculations are more challenging. We include the proof of the unitary version of Lemma~\ref{Lem:Habdank} as an indication of complexity of calculations. Recall from~\S\ref{sec3} that 
$$ (I,\Gamma)\circ (J,\Delta)=
\big(IJ+JI,\Gamma_{\min}(IJ+JI)+{}^J\Gamma+{}^I\Delta\big). $$

 \begin{LemB}\label{yyqq1}
 Let\/ $(I,\Gamma)$ and\/ $(J,\Delta)$ be two form ideals
of a form ring\/ $(A,\Lambda)$. Then
\begin{multline*} \EU(2n,(I,\Gamma)\circ(J,\Delta))\le
[\FU(2n,I,\Gamma),\FU(2n,J,\Delta)]\le \\
[\EU(2n,I,\Gamma),\EU(2n,J,\Delta)]\le [\GU(2n,I,\Gamma),\GU(2n,J,\Delta)] \\\le  \GU(2n,(I,\Gamma)\circ(J,\Delta)).
\end{multline*} 
\end{LemB}
\begin{proof}
We first show
\begin{equation}\label{yyqq1_1}
\EU(2n,(I,\Gamma)\circ (J,\Delta))\le  [\FU (2n, I, \Gamma), \FU(2n, J, \Delta)].
\end{equation}

It is well known that $\EU(2n,(I,\Gamma)\circ (J,\Delta))$ is generated by $T_{i,j}(\alpha)^{T_{j,i}(\xi)}$ with $\alpha \in I\circ J$, $\xi\in A$ when $i\ne \pm j$ and with $\alpha \in\lambda^{-(\ep(i)+1)/2} \Gamma\circ\Delta$ and $\xi\in \lambda^{(\ep(i)-1)/2}\Lambda$ when $i=-j$. We divide the proof into cases according the length of the elementary element.

Case I.  $T_{i,j}(\alpha) $ is a short root, namely $i\ne\pm j$. Then $\alpha \in I\circ J$. It is sufficient show that $T_{i,j}(a_1b_1+a_2b_2)^{T_{j,i}(\xi)}\in [\FU (2n, I, \Gamma), \FU(2n, J, \Delta)]$ for any $a_1,a_2\in I$ and $ b_1, b_2\in J$.
By (R2), we have
$$
T_{i,j}(a_1b_1+a_2b_2)^{T_{j,i}(\xi)}=T_{i,j}(a_1b_1)^{T_{i,j}(\xi)}T_{i,j}(a_2b_2)^{T_{j,i}(\xi)}.
$$
We will show the first factor of the right hand side of the above equation, and left the second to the reader.
Choose a $k\ne\pm i,\pm j$. Using (R4), the first factor can be rewrite as a commutator
\begin{eqnarray*}
T_{i,j}(a_1b_1)^{T_{j,i}(\xi)}&=&[T_{i,k}(a_1),T_{k,j}(b_1)]^{T_{j,i}(\xi)}\\
&=&[T_{i,k}(a_1)^{T_{j,i}(\xi)},T_{k,j}(b_1)^{T_{j,i}(\xi)}].\\
&=&\Big[[T_{i,k}(a_1),{T_{j,i}(-\xi)}]T_{i,k}(a_1), [T_{k,j}(b_1),{T_{j,i}(-\xi)}]T_{k,j}(b_1)\Big]
\end{eqnarray*}
Again by (R4), we have
\begin{multline*}
\Big[[T_{i,k}(a_1),{T_{j,i}(-\xi)}]T_{i,k}(a_1), [T_{k,j}(b_1),{T_{j,i}(-\xi)}]T_{k,j}(b_1)\Big]=\\\Big[T_{j,k}(a_1\xi)T_{i,k}(a_1), T_{k,i}(-\xi b_1)T_{k,j}(b_1)\Big].
\end{multline*}
Clearly $a_1\xi ,a_1 \in I$  and $-\xi b_1, b_1\in J$, thus 
\begin{multline}\label{eqn:c2}
T_{i,j}(a_1b_1)^{T_{j,i}(\xi)}=\\\Big[T_{j,k}(a_1\xi)T_{i,k}(a_1), T_{k,i}(-\xi b_1)T_{k,j}(b_1)\Big]\in [\FU (2n, I, \Gamma), \FU(2n, J, \Delta)].
\end{multline}
This finishes the proof of Case I.

Case II. $T_{i,j}(\alpha) $ is a long root, namely $i\ne - j$. Therefore we have $\alpha\in \lambda^{-(\ep(i)+1)/2} \Gamma\circ\Delta$. Without loss of generality, we may assume that $i<0$. Hence $\alpha\in  \Gamma\circ\Delta$. By definition, 
$$
\Gamma\circ\Delta=\Gamma^J+\Delta^I+\Gamma_{min}(IJ+JI).
$$
It suffices to show that $T_{i,-i}(\alpha_1+\alpha_2+\alpha_3)^{T_{-i,i}(\xi)}\in [\FU (2n, I, \Gamma), \FU(2n, J, \Delta)]$ with 
$\alpha_1\in \Gamma^J$, $\alpha 2\in \Delta^I$ and $\alpha_2\in \Gamma_{min}(IJ+JI)$. 
By (R2), 
$$
T_{i,-i}(\alpha_1+\alpha_2+\alpha_3)^{T_{-i,i}(\xi)}=T_{i,-i}(\alpha_1)^{T_{-i,i}(\xi)}T_{i,-i}(\alpha_2)^{T_{-i,i}(\xi)}T_{i,-i}(\alpha_3)^{T_{-i,i}(\xi)}.
$$
We prove one by one that each of the factors above belongs to \[[\FU (2n, I, \Gamma), \FU(2n, J, \Delta)].\] 

Since $\alpha_1\in \Gamma^J$, we may rewrite $\alpha_1=  a \gamma\bar a$ with $a\in J$ and $\gamma\in \Gamma$.  Therefore
$$
T_{i,-i}(\alpha_1)^{T_{-i,i}(\xi)}=T_{i,-i}( a\gamma\bar a)^{T_{-i,i}(\xi)}. 
$$
Choose a $j\ne i$ and $j<0$.  Equation (R6) implies that
\begin{eqnarray}
T_{i,-i}( a\gamma\bar a)^{T_{-i,i}(\xi)}&=&\Big(T_{i,-j}(-a\gamma )\big[T_{i,j}(a), T_{j,-j}(\gamma)\big]\Big){}^{T_{-i,i}(\xi)}\notag\\
&=&T_{i,-j}(-a\gamma ){}^{T_{-i,i}(\xi)}\big[T_{i,j}(a), T_{j,-j}(\gamma)\big]{}^{T_{-i,i}(\xi)}\notag\\
&=&\big[T_{i,-j}(-a\gamma ),{T_{-i,i}(\xi)}\big] T_{i,-j}(-a\gamma ) \big[T_{i,j}(a), T_{j,-j}(\gamma)\big]{}^{T_{-i,i}(\xi)}\label{eqn:c1}
\end{eqnarray}
 Again by (R6), the first factor 
 $$
 \big[T_{i,-j}(-a\gamma ),{T_{-i,i}(\xi)}\big]=T_{-j,j}(-\bar\lambda \gamma a\xi\overline{\gamma a})T_{-i,-j}(\bar\lambda\xi\overline{\gamma a} ) .
 $$
 Because $ \gamma a\xi\overline{\gamma a}\in\Gamma_{min}(I\circ J)$ and $\bar\lambda\xi\overline{\gamma a}\in I\circ J$, we have 
 $$
 T_{-j,j}(-\bar\lambda \gamma a\xi\overline{\gamma a})T_{-i,-j}(\bar\lambda\xi\overline{\gamma a} )\in \FU(2n,(I,\Gamma)\circ (J,\Delta))\le 
 [\FU (2n, I, \Gamma), \FU(2n, J, \Delta)].
 $$
 Furthermore, $\gamma a\in I\circ J$ implies the second factor of (\ref{eqn:c1})
 $$
 T_{i,j}(-\gamma a)\in \FU(2n,(I,\Gamma)\circ (J,\Delta))\le  [\FU (2n, I, \Gamma), \FU(2n, J, \Delta)].
 $$
 As for the last factor of (\ref{eqn:c1}),
\begin{eqnarray*}
\big[T_{i,j}(a), T_{j,-j}(\gamma)\big]{}^{T_{-i,i}(\xi)}&=&\big[T_{i,j}(a){}^{T_{-i,i}(\xi)}, T_{j,-j}(\gamma){}^{T_{-i,i}(\xi)}\big]\\
&=&\Big[\big[T_{i,j}(a),{T_{-i,i}(\xi)}\big]T_{i,j}(a), T_{j,-j}(\gamma)\Big]
\end{eqnarray*}
Apply (R6) to the first component of the commutator above shows that
\begin{eqnarray*}
\big[T_{i,j}(a),{T_{-i,i}(\xi)}\big]T_{i,j}(a)=T_{-j,j}(\lambda\bar a\xi a)T_{-i,j}(-\xi a) T_{i,j}(a)\in \FU(2n, J,\Delta).
\end{eqnarray*}
Thus 
$$T_{i,-i}(\alpha_1)^{T_{-i,i}(\xi)}\in [\FU (2n, I, \Gamma), \FU(2n, J, \Delta)].$$

 A similar argument, which is left to the reader, shows that  
$$T_{i,-i}(\alpha_2)^{T_{-i,i}(\xi)}\in [\FU (2n, I, \Gamma), \FU(2n, J, \Delta)].$$

For the third factor of (\ref{eqn:c2}),  we have  
$$
T_{i,-i}(\alpha_3)^{T_{-i,i}(\xi)}\in T_{i,-i}(\Gamma_{min}(IJ+JI))^{T_{-i,i}(\xi)}.
$$
By definition, 
$$
\Gamma_{min}(IJ+JI)=\{a-\lambda \bar a\mid a\in IJ+JI\}+\langle b\gamma \bar b\mid b\in IJ+JI, \gamma\in\Lambda\rangle.
$$
Hence, we shall show that for any given $\alpha_4$ and $\alpha_5$ which belong to the first and second summands of the above equation respectively,  both
$$
T_{i,-i}(\alpha_4)^{T_{-i,i}(\xi)} \text{ and } T_{i,-i}(\alpha_5)^{T_{-i,i}(\xi)} \in [\FU (2n, I, \Gamma), \FU(2n, J, \Delta)].
$$
For the first inclusion, take a typical generator $c_1d_1+d_2c_1-\lambda\overline{c_1d_1+d_2c_2} $ of $\{a-\lambda \bar a\mid a\in IJ+JI\}$ with $c_1,c_2\in I$ and $d_1,d_2\in J$. It suffices to prove that 
\begin{multline*}
T_{i,-i}(c_1d_1+d_2c_2-\lambda\overline{c_1d_1+d_2c_2} )^{T_{-i,i}(\xi)}\\=T_{i,-i}(c_1d_1-\lambda\overline{c_1d_1} )^{T_{-i,i}(\xi)}T_{i,-i}(d_2c_2-\lambda\overline{d_2c_2} )^{T_{-i,i}(\xi)}
\end{multline*}
We shall prove that $T_{i,-i}(c_1d_1-\lambda\overline{c_1d_1} )^{T_{-i,i}(\xi)}$ and the rest follows by the same augments.  

Choose a $j\ne i$ and $j<0$. Using (R5), we get
\begin{eqnarray*}
T_{i,-i}(c_1d_1-\lambda\overline{c_1d_1} )^{T_{-i,i}(\xi)}&=&\big[T_{i,j}(c_1), T_{j,-i}(d_1)\big]^{T_{-i,i}(\xi)}\\
&=&\big[T_{i,j}(c_1)^{T_{-i,i}(\xi)}, T_{j,-i}(d_1)^{T_{-i,i}(\xi)}\big]\\
&=&\Big[\big[T_{i,j}(c_1),{T_{-i,i}(\xi)}\big]T_{i,j}(c_1), \big[T_{j,-i}(d_1),{T_{-i,i}(\xi)}\big]T_{j,-i}(d_1)\Big]
\end{eqnarray*}
By (R6), $\big[T_{i,j}(c_1),{T_{-i,i}(\xi)}\big]$ can be write as a product of elements from $\FU (2n, I, \Gamma)$ and 
$\big[T_{j,-i}(d_1),{T_{-i,i}(\xi)}\big]$ a product of elements from $\FU(2n, J, \Delta)$. Thus 
$$
T_{i,-i}(\alpha_4)^{T_{-i,i}(\xi)} \in [\FU (2n, I, \Gamma), \FU(2n, J, \Delta)].
$$

Finally, as $\alpha_5\in \langle b\gamma \bar b\mid b\in IJ+JI, \gamma\in\Lambda\rangle$, we reduce our proof by (R2) to the case $$\alpha_5
=\Big(\sum_k a_k \Big)\gamma\Big(\overline{\sum_k a_k}\Big), \quad \text{ with} \quad a_k\in IJ+JI.
$$ 
By induction, it can be further reduce to 
\begin{eqnarray*}
\alpha_5&=&(a_1b_1+b_2a_2)\gamma\overline{a_1b_1+b_2a_2}
\end{eqnarray*}
with $a_1,a_2\in I$ and $b_1,b_2\in J$. The above equation can be rewritten as
\begin{eqnarray*}
\alpha_5&=&a_1b_1\gamma\overline{a_1b_1}+b_2a_2\gamma\overline{b_2a_2}+a_1b_1\gamma\overline{a_2b_2} + a_2b_2\gamma\overline{a_1b_1}\\
&=&a_1b_1\gamma\overline{a_1b_1}+b_2a_2\gamma\overline{b_2a_2}+(a_1b_1\gamma\overline{a_2b_2} -\lambda\overline{ a_1b_1\gamma\overline{a_2b_2}}).
\end{eqnarray*}
The last summand is of the same form as $\alpha_4$'s, hence it follows immediately by the proof of $\alpha_4$ that
$$
T_{i,-i}(a_1b_1\gamma\overline{a_2b_2} -\lambda\overline{ a_1b_1\gamma\overline{a_2b_2}})^{T_{-i,i}(\xi)} \in [\FU (2n, I, \Gamma), \FU(2n, J, \Delta)].
$$
Now consider the first two summands. Note that 
$$a_1b_1\gamma\overline{a_1b_1}= a_1(b_1\gamma\bar b_1)\bar a_1$$
and 
$$b_2a_2\gamma\overline{b_2a_2}=b_2(a_2\gamma\bar a_2)\bar b_2.
$$
By the definition of relative form parameter, $a_1(b_1\gamma\bar b_1)\bar a_1$ and $b_2(a_2\gamma\bar a_2)\bar b_2$ belong to 
$\Delta^I $ and $\Gamma^J$ respectively. The proofs for $\alpha_1$ and $\alpha_2$ show that 
\begin{eqnarray*}
T_{i,-i}(a_1(b_1\gamma\bar b_1)\bar a_1)^{T_{-i,i}(xi)}T_{i,-i}(a_2(b_2\gamma\bar b_2)\bar a_2)^{T_{-i,i}(xi)}\in [\FU (2n, I, \Gamma), \FU(2n, J, \Delta)].
\end{eqnarray*}
This proves (\ref{yyqq1_1}). We are left to show that 
\begin{equation}
\label{yyqq1_2}
[\GU(2n,I,\Gamma),\GU(2n,J,\Delta)] \le  \GU(2n,(I,\Gamma)\circ(J,\Delta)).
\end{equation}
We first show that (\ref{yyqq1_2}) holds for the {\it stable\/}
unitary groups, namely that
\begin{equation}\label{le_stable1}
[\GU(I,\Gamma),\GU(J,\Delta)]\le \GU((I,\Gamma)\circ(J,\Delta)).
\end{equation}
\noindent
In the stable level, we have inclusions
\begin{equation}\label{le_stable2}
        \EU((I,\Gamma)\circ(J,\Delta))\le
[\EU(I,\Gamma),\EU(J,\Delta)]\le[\GU(I,\Gamma),\GU(J,\Delta)]
\end{equation}
and
\begin{equation}\label{le_stable3}
[\EU(I,\Gamma),\EU(J,\Delta)]\le \GU((I,\Gamma)\circ(J,\Delta)).
\end{equation}

Since the subgroup $[\GU(I,\Gamma),\GU(J,\Delta)]$ is normalized by
$E(\FormR)$, applying Theorem~\ref{GU-sandwich},
we can conclude that there exists a unique form ideal $(K,\Omega)$ such
that
\begin{equation}\label{le_stable4}
\EU(K,\Omega)\le [\GU(I,\Gamma),\GU(J,\Delta)]\le \GU(K,\Omega).
\end{equation}
By Identity (C7), we get
\begin{multline*}
\big[[\GU(I,\Gamma),\GU(J,\Delta)],\EU(\FormR)\big]\le \\
\big[[\GU(I,\Gamma),\EU(\FormR)],\GU(J,\Delta)\big]\cdot
\big[[\GU(J,\Delta),\EU(\FormR)],\GU(I,\Gamma)\big].
\end{multline*}
\noindent
But the absolute commutator formula implies that
\begin{multline}
\big[[\GU(I,\Gamma),\EU(\FormR)],\GU(J,\Delta)\big]\cdot
\big[[\GU(J,\Delta),\EU(\FormR)],\GU(I,\Gamma)\big]=\\
[\EU(I,\Gamma),\EU(J,\Delta)].
\end{multline}
\noindent
Thus,
\begin{equation}\label{pjhenis}
\big[[\GU(I,\Gamma),\GU(J,\Delta)],\EU(\FormR)\big]
\le [\EU(I,\Gamma),\EU(J,\Delta)].
\end{equation}
\noindent
Again by the general commutator formula and (\ref{le_stable3}), we have
\begin{multline}\label{assme}
\EU((I,\Gamma)\circ(J,\Delta))=
[\EU((I,\Gamma)\circ(J,\Delta)),\EU(\FormR)]\\
\le \big[[\EU(I,\Gamma),\EU(J,\Delta)],\EU(\FormR)\big]\\
\le [\GU((I,\Gamma)\circ(J,\Delta)),\EU(\FormR)]
=\EU((I,\Gamma)\circ(J,\Delta)).
\end{multline}
\noindent
Forming another commutator of~(\ref{pjhenis}) with $\EU(\FormR)$
and applying the inequalities obtained in~(\ref{assme}) we get
$$ \Big[\big[[\GU(I,\Gamma),\GU(J,\Delta)],\EU(\FormR)\big],
\EU(\FormR)\Big]=\EU((I,\Gamma)\circ(J,\Delta)). $$
\noindent
Using inclusions (\ref{le_stable4}), we see that
\begin{multline*}
\EU(K,\Omega)=\big[[\EU(K,\Omega),\EU(\FormR)],\EU(\FormR)\big]\\
\le  \Big[\big[[\GU(I,\Gamma),\GU(J,\Delta)],\EU(\FormR)\big],
\EU(\FormR)\Big]
=\EU((I,\Gamma)\circ(J,\Delta))\\
=\big[[\EU((I,\Gamma)\circ(J,\Delta)),\EU(\FormR)],\EU(\FormR)\big]\\
\le  \Big[\big[[\GU(I,\Gamma),\GU(J,\Delta)],\EU(\FormR)\big],
\EU(\FormR)\Big]\\
\le \big[[\GU(K,\Omega),\EU(\FormR)],\EU(\FormR)\big]
=\EU(K,\Omega).
\end{multline*}
\noindent
Thus, we can conclude that $\EU(K,\Omega)=\EU((I,\Gamma)\circ(J,\Delta))$.
This implies that $(K,\Omega)=(I,\Gamma)\circ(J,\Delta)$, see the
second paragraph of the proof of~\cite[Theorem~5.4.10]{32}.
Substituting this equality in (\ref{le_stable4}), we see that inclusion
(\ref{le_stable1}) holds at the stable level, as claimed.
\par
Let $\varphi$ denote the usual stability embedding
$\varphi:\GU(2n,\FormR)\to\GU(\FormR)$. Then
\begin{multline*}
\varphi\big(\big[\GU(2n,I,\Gamma),\GU(2n,J,\Delta)\big]\big)=
\big[\varphi\big(\GU(2n,I,\Gamma)\big),
\varphi\big(\GU(2n,J,\Delta)\big)\big]<\\
\big[\GU(I,\Gamma),\GU(J,\Delta)\big].
\end{multline*}
\noindent
In particular, the result at the stable level implies that
\[ \varphi\big(\big[\GU(2n,I,\Gamma),\GU(2n,J,\Delta)\big]\big)
\le \varphi\big(\GU(2n,\FormR)\big)\cap
\GU((I,\Gamma)\circ(J,\Delta)). \]
\noindent
On the other hand,
$$ \varphi\big(\GU(2n,\FormR)\big)\cap\GU((I,\Gamma)\circ(J,\Delta))=
\varphi\big(\GU(2n,(I,\Gamma)\circ(J,\Delta))\big). $$
\noindent
Since $\varphi$ is injective, we can conclude that
$$ [\GU(2n,I,\Gamma),\GU(2n,J,\Delta)]\le
\GU(2n,(I,\Gamma)\circ(J,\Delta)). $$
\noindent
This finishes the proof.

\end{proof}

We can show state the unitary version of generalised commutator formula. 

\begin{TheB}
Let\/ $n\ge 3$,\/ $R$ be a commutative ring,\/ $(A,\Lambda)$ be a form
ring such that\/ $A$ is a module finite\/ $R$-algebra. Further, let\/
$(I,\Gamma)$ and\/ $(J,\Delta)$ be two form ideals of the form ring\/
$(A,\Lambda)$. Then
$$ \big[\EU(2n,I,\Gamma),\GU(2n,J,\Delta)\big]=
\big[\EU(2n,I,\Gamma),\EU(2n,J,\Delta)\big]. $$
\end{TheB}

Actually, in the {\it commutative\/} case the principal congruence
subgroup in the left hand side of the equalities can be replaced by
the full congruence subgroup. In other words, when $R$ is commutative,
one has
$$ [E(n,R,I),C(n,R,J)]=[E(n,R,I),E(n,R,J)]. $$
\noindent
Similarly, when $A$ is commutative, one has
$$ \big[\EU(2n,I,\Gamma),\CU(2n,J,\Delta)\big]=
\big[\EU(2n,I,\Gamma),\EU(2n,J,\Delta)\big]. $$
\noindent
On the other hand, it is easy to construct non-commutative
counter-examples to these stronger assertions, see \cite{70}.
\par

Finally, for Chevalley groups the corresponding result was first
officially stated by You Hong~\cite[Theorem~1]{121}, see also ~\cite[Lemmas~17,19]{43}.

\begin{LemC}\label{gbsinea6}
Let\/ $\rk(\Phi)\ge 2$. In the cases\/ $\Phi=\B_2,\G_2$ assume
that $R$ does not have residue fields\/ ${\Bbb F}_{\!2}$ of\/ $2$ elements and
in the case\/ $\Phi=\B_2$ assume additionally that any\/ $c\in R$ is
contained in the ideal\/ $c^2R+2cR$.
\par
Then for any two ideals\/ $I$ and\/ $J$ of the ring\/ $R$ one has
the following inclusion
\begin{multline*}
E(\Phi,R,IJ)\le\big [E(\Phi,R,I),E(\Phi,R,J)\big]\le
\big [E(\Phi,R,I),G(\Phi,R,J)\big]\le \\
\big[G(\Phi,R,I),C(\Phi,R,J)\big]\le G(\Phi,R,IJ).
\end{multline*}
\end{LemC}
For groups of rank 2, these additional assumptions are indeed necessary.
It is classically known that when the ground ring $R$ has residue
fields of 2 elements, the groups of types $\B_2$ and $\G_2$ are
not perfect. Thus, the left-most inclusion fails even at the
absolute level, when $I=J=R$.

The second assumption for $\B_2$ is not visible at the absolute
level. But without that assumption the upper and lower levels of
the relative commutator subgroup $\big [E(\Phi,R,I),E(\Phi,R,J)\big]$
do not coincide, so that the left-most inclusion in the above lemma
should be replaced by
$$ E(\Phi,R,IJ,I^2J+2IJ+IJ^2)\le\big[E(\Phi,R,I),E(\Phi,R,J)\big]. $$
Here, $E(\Phi,R,I,J)$ is the elementary subgroup corresponding to
an admissible pair $(I,J)$ in the sense of Abe, where $I$ is an
ideal of $R$, expressing the short root level (= upper level),
whereas a Jordan ideal $J$, expressing the long root level (=
lower level), plays the role of a form parameter. Not to complicate
things any further, in the sequel we {\it always\/} impose these
additional restrictions on $R$, when $\Phi=\B_2,\G_2$. These two
cases, especially that of the group $\Sp(4,R)$, require separate
analysis anyway, [CK1], [CK2].

Since Chevalley groups of types other than $\A_l$ are only defined
over commutative rings, we can state the next result with the
full congruence subgroup right from the outset. It is (essentially)~\cite[Theorem~3]{43}, with slightly weaker assumptions for Chevalley 
groups of rank $2$.

\begin{TheC}\label{hhggaw1}
Let\/ $\Phi$ be a reduced irreducible root system,\/ $rk(\Phi)\ge 2$.
Further, let\/ $R$ be a commutative ring, and\/ $I,J\trianglelefteq R$
be two ideals of\/ $R$. In the cases\/ $\Phi=\B_2,\G_2$ assume
that\/ $R$ does not have residue fields\/ ${\Bbb F}_{\!2}$ of\/ $2$
elements and in the case\/ $\Phi=\B_2$ assume additionally that
any\/ $c\in R$ is contained in the ideal\/ $c^2R+2cR$. Then
$$ [E(\Phi,R,I),C(\Phi,R,J)]=[E(\Phi,R,I),E(\Phi,R,J)]. $$
\end{TheC}

Actually, relative standard commutator formulas can be proven
by localisation, as in~\cite{46,42,43}, and this is precisely
the proof on which most generalisations are based. Otherwise, they
can be reduced to the {\it absolute\/} standard commutator formulas
by level calculations, as in~\cite{121,114,42,43}. Of course, 
the usual proofs of the absolute commutator formulas themselves 
in this generality involve some forms of localisation, at least 
in the non-commutative case.
\par
Before proceeding to higher generalisations, we dwell a bit more on
the structure and generation of the relative commutator subgroups
$[E(R,I),E(R,J)]$ that appear in these theorems. These results are
essentially elementary, sheer abstract or algebraic group theory,
and do not use localisation. But they are useful and amusing,
and serve to motivate, prove or amplify our main theorems.


\section{Relative commutator subgroups are not elementary}\label{gbdtmu43}

In view of Theorem~\ref{gcformula1}, it is natural to ask, whether the commutators
of relative elementary subgroups are themselves elementary of
the corresponding level, in other words, whether 
\begin{equation}\label{psy62}
[E(A,I),E(A,J)]=E(A,I\circ J)
\end{equation}
holds?

This is known to be the case in many important classical situations,
for instance, at the absolute level, where $I=A$ or $J=A$. In fact,
this equality holds under much weaker assumptions. Specifically,
it is easily verified when the ideals $I$ and $J$ are comaximal,
$I+J=A$. We will reproduce the proof of this fact in the setting of general linear group from~\cite{114}. The proof in the setting of unitary groups and Chevalley groups can be now found in \cite[Theorem~3]{42}, and 
\cite[Theorem~3]{43}, respectively. 

\begin{TheA}\label{yy771}
Let\/ $A$ be a quasi-finite ring,\/ $n\ge 3$. Then for any
two comaximal ideals\/ $I,J\trianglelefteq A$, $I+J=A$, one has
$$ [E(n,A,I),E(n,A,J)]=E(n,A,I\circ J). $$
\end{TheA}
\begin{proof}
First observe that an application of (E1) shows that for any ideals $I$ and $J$ of $A$, we have 
\begin{equation}\label{pagisj1}
E(n,A,I)E(n,A,J) = E(n,A,I+J).
\end{equation}
Since $I$ and $J$ are comaximal, from~(\ref{pagisj1}) it follows $E(n,A,I)E(n,A,J) = E(n,A)$.

Now 
\[E(n, A, I ) = [E(n, A,  I), E(n, A)] = [E(n, A, I), E(n, A, I)E(n, A, J)].\]
Thus using Lemma~\ref{Lem:Habdank}  we can write 
\begin{multline*}
E(n, A, I) \le  [E(n, A, I), E(n, A, I)][E(n, A, I), E(n, A, J)]
\le \\  [E(n, A, I), E(n, A, I)]\GL(n, A, IJ+JI).
\end{multline*}
Commuting this inclusion with $E(n, A,J)$, we see that
\begin{multline*} 
[E(n, A, I), E(n, A, J)] \le  \\
\Big[[E(n, A, I), E(n, A, I)], E(n, A, J)\Big]\Big[\GL(n, A, IJ+JI), E(n, A, J)\Big]. 
\end{multline*}

Applied to the second factor, the standard commutator formula, Theorem~\ref{standard},  shows that 
\begin{multline*} 
[\GL(n,A,IJ+JI),E(n,A,J)] \le \\ [\GL(n,A,IJ+JI),E(n,A)] = E(n,A,IJ+JI).
\end{multline*}

On the other hand, applying Lemma~\ref{Lem:Habdank}  to the first factor, and then invoking the  standard commutator formula again, we have 
\begin{multline*} 
\Big[[E(n, A, I), E(n, A, J)], E(n, A, I)\Big] \le  [\GL(n, A, IJ+JI), E(n, A, I)] \le  \\ [\GL(n,A,IJ+JI),E(n,A)] = E(n,A,IJ+JI).
\end{multline*}
Thus we have 
\[ [E(n,A,I),E(n,A,J)] \le  E(n,A,IJ+JI). \]
Combining this with Lemma~\ref{Lem:Habdank}, the proof is complete. 
\end{proof}

\begin{TheB}
Let\/ $n\ge 3$, and\/ $(A,\Lambda)$ be an arbitrary form
ring for which absolute standard commutator formulae are satisfied.
Then for any two comaximal form ideals\/ $(I,\Gamma)$ and $(J,\Delta)$
of the form ring $(A,\Lambda)$, $I+J=A$, one has the following equality
$$ [\EU(2n,I,\Gamma),\EU(2n,J,\Delta)]=
\EU(2n,IJ+JI,{}^J\Gamma+{}^I\Delta+\Gamma_{\min}(IJ+JI)). $$
\end{TheB}

\begin{TheC}\label{ttrr6d}
Let\/ $\Phi$ be a reduced irreducible root system,\/ $rk(\Phi)\ge 2$.
Further, let\/ $A$ be a commutative ring, and\/ $I,J\trianglelefteq A$
be two ideals of\/ $A$. In the cases\/ $\Phi=\B_2,\G_2$ assume
that\/ $A$ does not have residue fields\/ ${\Bbb F}_{\!2}$ of\/ $2$
elements. Then for any two comaximal ideals $I,J\trianglelefteq A$,
$I+J=A$, one has the following equality
$$ [E(\Phi,A,I),E(\Phi,A,J)]=E(\Phi,A,IJ). $$
\end{TheC}

Observe, that unlike Theorem~\ref{hhggaw1}, in Theorem~\ref{ttrr6d} the extra
assumption on $R$ for type $\B_2$ turned out to be redundant
(due to more accurate level calculations in terms of admissible pairs).

\subsection{}\label{countexam}
Despite Theorem~\ref{yy771}, the relative commutator subgroup $[E(A,I),E(A,J)]$ cannot be
always ele\-mentary of the form (\ref{psy62}). We reproduce from~\cite{72,70} one
such example based on the calculation of relative $K_1$-functors
for Dedekind rings of arithmetic type by Hyman Bass, John Milnor
and Jean-Pierre Serre~\cite{15}. We do not make any attempt to recall
the explicit formula for
$$ \SK_1(n,A,I)=\SL(n,A,I)/E(n,A,I) $$
in the general case. Instead, we cite the explicit answer for the {\it first\/} non-trivial
case of Gaussian integers $A=\Int[i]$. Consider the prime ideal
$\prim=(1+i)A$. Then for any $n\ge 3$ and any ideal
$I\trianglelefteq A$ one has
$$ \SK_1(n,A,I)=\SK_1(n,A,\prim^s),\qquad s=\ord_{\prim}(I). $$
\noindent
On the other hand,
$$ \SK_1(n,A,\prim^s)=
\begin{cases} 1,&s\le 3,\\ 2,&s=4,5,\\ 4,&s\ge 6.\\ \end{cases} $$
\noindent
Now a straightforward calculation shows that
\begin{multline*}
E(n,\Int[i],\prim^6)<
[E(n,\Int[i],\prim^3),E(n,\Int[i],\prim^3)]=\\
[\SL(n,\Int[i],\prim^3),\SL(n,\Int[i],\prim^3)]<\SL(n,\Int[i],\prim^6),
\end{multline*}
where {\it both\/} inclusions are strict. In fact, both indices
are equal to 2.

This, and many further examples of arithmetic and algebra-geometric
nature show that in general the relative commutator subgroup
$[E(n,A,I),E(n,A,J)]$ is {\it strictly larger\/} than the relative
elementary subgroup $E(n,A,I\circ J)$.

In particular, it follows that in general
$$ [E(n,A,I),E(n,A,J)]\neq [E(n,A,K),E(n,A,L)] $$
for two pairs of ideals $(I,J)$ and $(K,L)$, such that
$I\circ J=K\circ L$. In fact, this already follows from the
previous example, for pairs $(I,J)$ and $(K,L)=(I\circ J,A)$,
but it is easy to construct many further examples, much fancier
than that.

Summarising the above, we can conclude that in general the
double relative commutator subgroups do not reduce to relative
elementary subgroups, and reveal some new layers of the internal
structure of $\K_1(A,I)$.

Amazingly, all higher multiple commutator
subgroups reduce to {\it double\/} commutator subgroups. In
other words, forming successive com\-mu\-tators of relative elementary
subgroups {\it never\/} results in anything new inside $K_1(A,K)$,
apart from the groups
$$ [E(A,I),E(A,J)]/E(A,K)\le K_1(A,K), $$
for some other ideals $I$ and $J$, such that $I\circ J=K$. We will discuss this in~\S\ref{ziuso1}.

\section{Generators of relative commutator subgroups}

Here, we describe generators of relative commutator subgroups
$[E(A,I),E(A,J)]$ as normal subgroups of $E(A)$. These results are
elementary algebraic group theory, but they are an essential
complement to Theorem~\ref{gcformula1}, an important tool in the proof of
multiple commutator formula, and the starting point for results
on relative commutator width.

By Lemma~\ref{Lem:Habdank} the relative commutator subgroup $[E(A,I),E(A,J)]$
contains the elementary subgroup $E(A,I\circ J)$.
In particular, it contains the generators of that group.
However, we know that in general $[E(A,I),E(A,J)]$ may be
strictly larger, than $E(A,I\circ J)$ (see~\S\ref{gbdtmu43}). Thus, we have to produce
the missing generators. As in the case of the relative elementary
subgroups $E(A,I)$ themselves, these generators will sit in the
fundamental $\SL_2$'s and are in fact commutators of {\it some\/}
elementary generators of $E(\Phi,A,I)$ and $E(\Phi,A,J)$.

\begin{LemA}\Label{Comgenerator}
Let $A$ be a ring and $I, J$ be two-sided ideals of $A$. Then \[\big[E(n,A,I), E(n,A,J)\big]\] is generated as a group by the elements of the form 
\begin{align}\label{yyttrree}
&^{c}\Big[e_{j,i}(\alpha),{}^{e_{i,j}(a)}e_{j,i}(\beta)\Big], \notag\\
&^{c}\big[e_{j,i}(\alpha), e_{i,j}(\beta)\big],\notag\\
&^{c}e_{i,j}(\alpha\beta),\notag\\ 
&^{c}e_{i,j}(\beta\alpha),
\end{align}
where $1\leq i\ne j\leq n$, $\alpha\in I$, $\beta\in J$, $a \in A$ and $c\in E(n,A)$.
\end{LemA}
\begin{proof}
A typical generator of $\big[E(n,A,I), E(n,A,J)\big]$ is of the form 
$[e,f]$, where $e \in E(n,A,I)$ and $f \in E(n,A,J)$. Thanks to Lemma~\ref{Engenerator}, we may assume that $e$ and $f$ are products of elements of the form 
$$
e_{i}=^{e_{p',q'}(a)} e_{q',p'}(\alpha)  \quad \text{and} \quad 
f_{j}=^{e_{p,q}(b)} e_{q,p}(\beta),
$$
where $a,b\in A$, $\alpha\in I $ and $\beta\in J$, respectively. Applying (C$1^+$) and then (C$2^+$), it follows that
$\big[E(n,A,I), E(n,A,J)\big]$ is generated by  elements of the form 
$$^{c}\big[^{e_{i',j'}(a)}e_{j',i'}(\alpha),{}^{e_{i,j}(b)}e_{j,i}(\beta)\big],$$
where $c\in E(n,A)$. Furthermore, 
$$
^c\big[^{e_{i',j'}(a)}e_{j',i'}(\alpha),{}^{e_{i,j}(b)}e_{j,i}(\beta)\big]=^{ce_{i',j'}(a)}\big[e_{j',i'}(\alpha),{}^{e_{i',j'}(-a)e_{i,j}(b)}e_{j,i}(\beta)\big].
$$
The normality of $E(n,A,J)$ implies that ${}^{e_{i',j'}(-a)e_{i,j}(b)}e_{j,i}(\beta)\in E(n,A,J)$, which is a product of  $^{e_{p,q}(a)}e_{q,p}(\beta)$, $a\in A$ and $\beta \in J$ by Lemma~\ref{Engenerator}. Again by (C$1^+$), one reduces the proof to the case 
of showing that
$$
\big[e_{i',j'}(\alpha),{}^{e_{i,j}(a)}e_{j,i}(\beta)\big]
$$
is a product of the generators listed in~(\ref{yyttrree}). We need to consider the following cases:
\begin{itemize}
\item If $i'=j,j'=i$: Then there is nothing to prove.

\smallskip 

\item if $i'=j, j'\ne i$:
\begin{eqnarray*}
\big[e_{j,j'}(\alpha),{}^{e_{i,j}(a)}e_{j,i}(\beta)\big]&=&{}^{e_{i,j}(a)}\big[{}^{e_{i,j}(-a)}e_{j,j'}(\alpha ), e_{j,i}(\beta)\big]\\
&=&{}^{e_{i,j}(a)}\big[[{e_{i,j}(-a)},e_{j,j'}(\alpha )]e_{j,j'}(\alpha ), e_{j,i}(\beta)\big]\\
&=&{}^{e_{i,j}(a)}\big[e_{i,j'}(-a\alpha )e_{j,j'}(\alpha ), e_{j,i}(\beta )\big].
\end{eqnarray*}
Applying now (C2),
\begin{eqnarray*}
[e_{i,j'}(-a\alpha )e_{j,j'}(\alpha ), e_{j,i}(\beta )]&=&
\big({}^{e_{i,j'}(-a\alpha )}[e_{j,j'}(\alpha ), e_{j,i}(\beta )]\big)[e_{i,j'}(-a\alpha ), e_{j,i}(\beta )]\\
&=&[e_{i,j'}(-a\alpha ), e_{j,i}(\beta )]\\
&=&[ e_{j,i}(\beta ),e_{i,j'}(-a\alpha )]^{-1}\\
&=&e_{j,j'}(-\beta a \alpha)^{-1}\\
&=&e_{j,j'}(\beta a \alpha)
\end{eqnarray*}
Thus 
$$
\big[e_{j,j'}(\alpha),{}^{e_{i,j}(a)}e_{j,i}(\beta)\big]={}^{e_{i,j}(a)}e_{j,j'}(\beta a \alpha)
$$
which satisfies the lemma.
\item if $i'\ne j, j'= i$: The argument is similar to the previous case.

\smallskip 

\item if $i'\not =j, j' \not = i$: We consider four cases: 
\begin{itemize} 

\smallskip 

\item if $i'=i,j'=j$: 
$$
\big[e_{i,j}(\alpha),{}^{e_{i,j}(a)}e_{j,i}(\beta)\big]={}^{e_{i,j}(a)}\big[e_{i,j}(\alpha),e_{j,i}(\beta)\big].
$$
\item if $i'=i,j'\ne j$:
\begin{eqnarray*}
\big[e_{i,j'}(\alpha),{}^{e_{i,j}(a)}e_{j,i}(\beta)\big]&=&{}^{e_{i,j}(a)}\big[e_{i,j'}(\alpha),e_{j,i}(\beta)\big]\\
&=&{}^{e_{i,j}(a)}e_{j,j'}(-\beta\alpha).
\end{eqnarray*}

\item if $i'\ne i,j'=j$:
 \begin{eqnarray*}
\big[e_{i',j}(\alpha),{}^{e_{i,j}(a)}e_{j,i}(\beta)\big]&=&{}^{e_{i,j}(a)}\big[e_{i',j}(\alpha),e_{j,i}(\beta)\big]\\
&=&{}^{e_{i,j}(a)}e_{i,i'}(\alpha\beta).
\end{eqnarray*}

\item if $i'\ne i, j'\ne j$: 
$$
\big[e_{i',j'}(\alpha),{}^{e_{i,j}(a)}e_{j,i}(\beta)\big]=1.
$$

\end{itemize}
\end{itemize}
This finishes the proof.
\end{proof}


\begin{TheA}\label{gen-2}
Let $A$ be a quasi-finite algebra with $1$, let $n\ge 3$, and let 
$I$, $J$ be two-sided ideals of $A$. Then the mixed commutator
subgroup $\big[E(n,A,I),E(n,A,J)\big]$
is generated as a group by the elements of the form
\begin{align}\label{yyttrree2}
&\big[e_{ji}(\alpha),{}^{e_{ij}(a)}e_{ji}(\beta)\big],\notag\\
&\big[e_{ji}(\alpha),e_{ij}(\beta)\big],\notag\\
&z_{ij}(\alpha\beta,a),\notag \\  
&z_{ij}(\beta\alpha,a),
\end{align}
where $1\le i\neq j\le n$, $\alpha\in I$, $\beta\in J$, $a\in A$.
\end{TheA}
\begin{proof}
By Lemma~\ref{Comgenerator}, the current generating set~(\ref{yyttrree2}) generates $\big[E(n,A,I),E(n,A,J)\big]$ as a normal subgroup. Therefore, it suffices to show that any conjugates of the generators~(\ref{yyttrree2}) is a product of these generators. Let $g$ be a generator listed in~(\ref{yyttrree2}), and $c\in E(n,A)$. Lemma~\ref{Lem:Habdank}  shows that 
$g\in \GL(n,A, I\circ J)$. Now applying the general commutator formula (see Theorem~\ref{standard}), one obtains
$$[c,g]\in [\GL(n,A,I\circ J), E(n,A)]= E(n, A, I\circ J).$$
Therefore by Lemma~\ref{Engenerator}, $[c,g]$ is a product of  $z_{ij}(\alpha\beta,a)$ and $z_{ij}(\beta\alpha,a)$ with $\alpha\in I$, $\beta\in J$, $a\in A$. It follows immediately that $cgc^{-1}$ is a product of the generators listed in~\ref{yyttrree2}. This completes the proof. 
\end{proof}

A closer look at the generating set in Theorem~\ref{gen-2} reveals an interesting fact that  all the generators are taken from $\big[E(n,I),E(n,A,J)\big]$. This implies the following corollary. 

\begin{CorA}  Let $A$ be a module finite ring and $I$ and $J$ two sided ideals of $A$. Then 
\[\big[E(n,I),E(n,A,J)\big]=\big[E(n,A,I),E(n,J)\big]=\big[E(n,A,I),E(n,A,J)\big].\] 
\end{CorA}

\begin{Cor}\label{cor_gen-2}
Let $A$ be an quasi-finite algebra with identity, $n\ge 3$, and let 
$I$, $J$ be two-sided ideals of $A$. Then the absolute mixed commutator
subgroup $\big[E(n,I),E(n,J)\big]$ is a normal subgroup of $E(n,A)$. 
\end{Cor}
\begin{proof}
Let $g$ be a typical element in $\big[E(n,I),E(n,J)\big]$ and let $c\in E(n,A)$. As in the proof of Theorem~\ref{gen-2}, we have
$[c,g]\in E(n, A, I\circ J)\le  \big[E(n,I),E(n,J)\big] .$
It follows immediately that $cgc^{-1}\in\big[E(n,I),E(n,J)\big]$. Thus $\big[E(n,I),E(n,J)\big]$ is a normal subgroup of $E(n,A)$.
\end{proof}

A similar result for unitary groups is \cite[Theorem~9]{44}, which is 
more techni\-cal. To somewhat shorten the next statement, we describe
conditions on the generators in the form $T_{ji}(\xi)\in\EU(2n,I,\Gamma)$.
Recall that (as in~\cite{13,33,41}) that this means that $\xi\in I$, for 
$i\neq\pm j$, and $\xi\in\Gamma$, for $i=-j$.
\begin{LemB}
Let\/ $(\Form)$ be a form ring and\/ $(I,\Gamma)$, $(J,\Delta)$
be two form ideals of\/ $(\Form)$. Then as a normal subgroup
of\/ $\EU(2n,R,\Lambda)$, $n\ge 3$, the mixed commutator
subgroup\/ $\big[\EU(2n,I,\Gamma),\EU(2n,J,\Delta)\big]$
is generated by the elements of the form
\begin{itemize}
\item[$\bullet$] $\big[T_{ji}(\xi),{}^{T_{ij}(\eta)}T_{ji}(\zeta)\big]$,
\item[$\bullet$] $\big[T_{ji}(\xi),T_{ij}(\zeta)\big]$,
\item[$\bullet$] $T_{ij}(\xi\zeta)$ and $T_{ij}(\zeta\xi)$,
\end{itemize}
where\/ $T_{ji}(\xi)\in\EU(2n,I,\Gamma)$,
$T_{ji}(\zeta)\in\EU(2n,J,\Delta)$, $T_{ij}(\eta)\in\EU(2n,\Form)$,
and\/ $T_{ij}(\theta)\in\EU\big(2n,(I,\Gamma)\circ(J,\Delta)\big)$.
\end{LemB}
The proof for Chevalley groups is similar, with some additional
complications in the rank 2 case. The following result is
 of~\cite[Theorem~2]{45}.
\begin{LemC}\label{gytd83}
Let\/ $\rk(\Phi)\ge 2$ and let\/ $I$, $J$ be two ideals
of a commutative ring\/ $R$. In the cases\/ $\Phi=\B_2,\G_2$ assume
that $R$ does not have residue fields\/ ${\Bbb F}_{\!2}$ of\/ $2$
elements and in the case\/ $\Phi=\B_2$ assume additionally that
any\/ $c\in R$ is contained in the ideal\/ $c^2R+2cR$.

Then as a normal subgroup of the elementary Chevalley
group\/ $E(\Phi,R)$ the mixed commutator subgroup\/
$\big[E(\Phi,R,I),E(\Phi,R,J)\big]$ is generated by the elements
of the form
\begin{itemize}
\item[$\bullet$] $\big[x_{\alpha}(\xi),
{}^{x_{-\alpha}(\eta)}x_{\alpha}(\zeta)\big]$,
\item[$\bullet$] $\big[x_{\alpha}(\xi),x_{-\alpha}(\zeta)\big]$,
\item[$\bullet$] $x_{\alpha}(\xi\zeta)$,
\end{itemize}
where\/ $\alpha\in\Phi$, $\xi\in I$, $\zeta\in J$, $\eta\in R$.
\end{LemC}

Actually, the proof of this result in~\cite{45} replaces most of explicit
fiddling with the Chevalley commutator formula and commutator
identities, by a reference to some obvious properties of parabolic
subgroups, which makes it {\it considerably\/} less computational,
than the proofs of Lemma~\ref{Lem:Habdank} and Lemma~\ref{yyqq1} in~\cite{47,43}.

We sketch the proof of Lemma~\ref{gytd83} as well.  First of all, observe that these elements indeed belong
to the relative commutator subgroups $[E(R,I),E(R,J)]$ by Lemma~\ref{gbsinea6}.
Next, recall that the elementary generators of the elementary
groups $E(R,I)$ themselves are classical\-ly known, and look as
follows:
\begin{itemize}
\item[$\bullet$] $z_{ji}(\xi,\eta)=e_{ij}(\eta)e_{ji}(\xi)e_{ij}(-\eta)$,
for $\GL_n$, (see Lemma~\ref{Engenerator}). 
\item[$\bullet$]$Z_{ji}(\xi,\eta)=T_{ij}(\eta)T_{ji}(\xi)T_{ij}(-\eta)$,
for unitary groups, (see~\cite{13}).
\item[$\bullet$] $z_{\alpha}(\xi,\eta)=
x_{-\alpha}(\eta)x_{\alpha}(\xi)x_{-\alpha}(-\eta)$, for Chevalley
groups (see~\cite{85,96,100,3}).
\end{itemize}

Observe, that these generators are precisely the second factors of
the first type of generators in the above Lemma~\ref{gytd83}, and we use this
shorthand notation in the sequel.
 The usual commutator identities imply that
as a normal subgroup \[\big[E(\Phi,R,I),E(\Phi,R,J)\big]\] is generated
by the commutators of the form
$\big[z_{\a}(\xi,\eta),z_{\b}(\zeta,\theta)\big]$. Since we are working
up to elementary conjugation, we can replace these generators by
\[\big[x_{\a}(\xi),{}^{x_{-\a}(-\eta)}z_{\b}(\zeta,\theta)].\]
Since the groups $E(\Phi,R,J)$ are normal in $E(\Phi,R)$, the
conjugates ${}^{x_{-\a}(-\eta)}z_{\b}(\zeta,\theta)$ can be again
expressed as products of elementary generators. Once more applying
commutator identities, we see that as a normal subgroup
$\big[E(\Phi,R,I),E(\Phi,R,J)\big]$ is generated by the commutators
$\big[x_{\a}(\xi),z_{\b}(\zeta,\theta)]$. At this point, we are left
with three options:
\begin{itemize}
\item[$\bullet$]$\a=\b$, and we get the first type of generators,
\item[$\bullet$] $\a=-\b$, and we get the second type of generators,
up to conjugation,
\item[$\bullet$] $\a\neq\pm\b$. If $\a$ and $\b$ are strictly orthogonal,
then $\big[x_{\a}(\xi),z_{\b}(\zeta,\theta)]=e$. Thus, we can assume
that $\a$ and $\b$ generate an irreducible root system of rank 2, and
fiddle with the Chevalley commutator formula therein. Alternatively,
we can choose an order such that $\b$ is fundamental, whereas $\a$
is positive. Then $\big[x_{\a}(\xi),z_{\b}(\zeta,\theta)]$ sits
inside the unipotent radical $U_{\b}$ of the minimal (=rank 1)
standard  parabolic subgroup $P_{\b}$. On the other hand, by
Lemma~\ref{gbsinea6} it sits inside $G(\Phi,R,IJ)$. Clearly,
$U_{\b}\cap G(\Phi,R,IJ)\le E(\Phi,IJ)$. Thus, in this last case
$\big[x_{\a}(\xi),z_{\b}(\zeta,\theta)]$ is a product of generators of
the third type.
\end{itemize}

\begin{TheB}\label{5bbb}
Let\/ $n\ge 3$,\/ $R$ be a commutative ring,\/ $(A,\Lambda)$ be a form
ring such that\/ $A$ is a quasi-finite\/ $R$-algebra. Further, let\/
$(I,\Gamma)$ and\/ $(J,\Delta)$ be two form ideals of the form ring\/
$(A,\Lambda)$.
\par
Then the mixed commutator subgroup
$\big[\EU(2n,I,\Gamma),\EU(2n,J,\Delta)\big]$ is generated as a group
by the elements of the form
\begin{itemize}
\item[$\bullet$] $\big[T_{ji}(\xi),Z_{ji}(\zeta,\eta)\big]$,
\item[$\bullet$] $\big[T_{ji}(\xi),T_{ij}(\zeta)\big]$,
\item[$\bullet$]
 $Z_{ij}(\theta,\eta)$,
\end{itemize}
where\/ $T_{ji}(\xi)\in\EU(2n,I,\Gamma)$, while\/
$T_{ij}(\zeta),Z_{ji}(\zeta,\eta)\in\EU(2n,J,\Delta)$,
and\/ $Z_{ij}(\theta,\eta)\in\EU\big(2n,(I,\Gamma)\circ(J,\Delta)\big)$.
\end{TheB}
\begin{TheC}\label{tuboon1}
Let\/ $\rk(\Phi)\ge 2$ and let\/ $I$, $J$ be two ideals
of a commutative ring\/ $A$. In the cases\/ $\Phi=\B_2,\G_2$ assume
that $A$ does not have residue fields\/ ${\Bbb F}_{\!2}$ of\/ $2$
elements and in the case\/ $\Phi=\B_2$ assume additionally that
any\/ $c\in A$ is contained in the ideal\/ $c^2A+2cA$.
\par
Then the mixed commutator subgroup\/
$\big[E(\Phi,A,I),E(\Phi,A,J)\big]$ is generated as a group by
the elements of the form
\begin{itemize}
\item[$\bullet$] $\big[x_{\a}(\xi),z_{\a}(\zeta,\eta)\big]$,
\item[$\bullet$] $\big[z_{\a}(\xi),z_{-\a}(\zeta)\big]$,
\item[$\bullet$] $z_{\a}(\xi\zeta,\eta)$,
\end{itemize}
where\/ $\alpha\in\Phi$, $\xi\in I$, $\zeta\in J$, $\eta,\in A$.
\end{TheC}

Let us sketch the proof of Theorem~\ref{tuboon1}. From this proof, it will be
clear, why a similar slick argument does not prove Theorem~\ref{gen-2} and
Theorem~\ref{5bbb} for arbitrary associative rings or arbitrary form rings.
\par
The set described in this theorem contains the set described in
Lemma~\ref{gytd83}, which already generates $\big[E(\Phi,A,I),E(\Phi,A,J)\big]$
as a normal subgroup of $E(\Phi,A)$. Therefore, it suffices to show
that elementary conjugates of the above generators are themselves
products of such generators. Let $g$ be one of these generators
and let $h\in E(\Phi,A)$. By Lemma~\ref{gytd83}, one has $g\in G(\Phi,A,IJ)$.
Now the [absolute] standard commutator formula implies that
$$ [h,g]\in[G(\Phi,A,IJ),E(\Phi,A)]=E(\Phi,A,IJ). $$
\noindent
Being an element $E(\Phi,A,IJ)$, the commutator $[h,g]$ is a product
of some elementa\-ry generators $z_{\a}(\xi\zeta,\eta)$, where
$\a\in\Phi$, $\xi\in I$, $\zeta\in J$, $\eta\in A$. Thus, any conjugate
$hgh^{-1}=[h,g]g$ is a product of some generators of the third
type and the generator $g$ itself.
\par
In fact, mostly this argument relied on {\it elementary\/} calculations,
such as the one needed to prove Lemma~\ref{gytd83} and Theorem 3C. But at one
instance we had to invoke a special case of Theorem 1C, the [absolute]
standard commutator formula. This last result is not elementary, and
certainly it does not hold over arbitrary associative rings. There are
explicit counter-examples to the standard commutator formula in this
generality, the first of them by Victor Gerasimov~\cite{28}.
\par
It seems incongruous that [what appears to be] a pure group theoretic
result should depend on commutativity conditions. This poses the
following problem.
\begin{Prob}
Find elementary proofs of Theorems~\ref{gen-2} and~\ref{5bbb}
that work over arbitrary associative rings/form rings.
\end{Prob}
By juggling with commutator identities, we succeeded in proving
a slightly weaker version of Theorem~\ref{gen-2}, with a somewhat larger
set of generators, all of them still sitting inside fundamental
$\GL_2$'s. However, a straightforward calculation, based on
induction on the length of the conjugating element, is so long
and appalling, that it strongly discouraged us from any attempt to
prove the technically much fancier Theorem~\ref{5bbb} for arbitrary form
rings along these lines.
\par
A closer look at the generators in Theorems~\ref{gen-2}--\ref{tuboon1} shows that all
of them in fact belong already to $\big[E(\Phi,I),E(\Phi,A,J)\big]$.
By symmetry, we may switch the role of factors. In particular, this
means that Theorems~\ref{gen-2}--\ref{tuboon1} imply the following curious corollaries.

\begin{CorB}
Let\/ $n\ge 3$,\/ $R$ be a commutative
ring,\/ $(A,\Lambda)$ be a form ring such that\/ $A$ is a
quasi-finite\/ $R$-algebra. Further, let\/ $(I,\Gamma)$ and\/
$(J,\Delta)$ be two form ideals of the form ring\/ $(A,\Lambda)$.
Then one has
 \begin{multline*}
 [\FU(2n,I,\Gamma),\EU(n,J,\Delta)]=
[\EU(2n,I,\Gamma),\FU(n,J,\Delta)]=\\
[\EU(2n,I,\Gamma),\EU(n,J,\Delta)].
\end{multline*}
\end{CorB}
\begin{CorC}
Let\/ $\rk(\Phi)\ge 2$ and let\/ $I$, $J$ be two ideals
of a commutative ring\/ $A$. In the cases\/ $\Phi=\B_2,\G_2$ assume
that $A$ does not have residue fields\/ ${\Bbb F}_{\!2}$ of\/ $2$
elements and in the case\/ $\Phi=\B_2$ assume additionally that
any\/ $c\in A$ is contained in the ideal\/ $c^2A+2cA$. Then one has
$$ [E(\Phi,I),E(\Phi,A,J)]=[E(\Phi,A,I),E(\Phi,J)]=
[E(\Phi,A,I),E(\Phi,A,J)]. $$
\end{CorC}

\section{Higher commutators}

Once we understand double commutators, it is natural to consider
higher com\-mu\-tators of relative elementary subgroups and congruence
subgroups. Let $G$ be a group and $H_1,\ldots,H_m\le G$ be its
subgroups. There are many ways to form a higher commutator of these
groups, depending on where we put the brackets. Thus, for three
subgroups $F,H,K\le G$ one can form two triple commutators
$[[F,H],K]$ and $[F,[H,K]]$. For four subgroups $F,H,K,L\le G$
one can form 5 such commutators
\begin{multline}\label{uyuyu}
[[[F,H],K],L],\qquad [[F,[H,K]],L],\qquad [[F,H],[K,L]],\\
[F,[H,[K,L]]],\qquad [F,[[H,K],L]].
\end{multline}
To be exact, there are as many as the Catalan number \[c_m=\frac{1}{(m+1)}\binom{2m}{m}\] ways to arrange the brackets involving $m+1$ subgroups in a meaningful way. 

Usually, we write $[H_0,H_1,\ldots,H_m]$ for the {\it left-normed\/}
commutator, defined inductively by
$$ [H_0,\ldots,H_{m-1},H_m]=[[H_0,\ldots,H_{m-1}],H_m]. $$
\noindent
To stress that we consider any commutator of these subgroups, with an
arbitrary placement of brackets, we write $[\![H_0,H_2,\ldots,H_m]\!]$.
Thus, for instance, $[\![F,H,K,L]\!]$ refers to any of the five
arrangements in (\ref{uyuyu}).
\par
Actually, a specific arrangement of brackets usually does not play
major role in our results -- and in fact any role whatsoever over
commutative rings! -- apart from one important attribute. Namely,
what will matter a lot is the position of the outermost pairs of
inner brackets. Namely, every higher commutator subgroup
$[\![H_0,H_2,\ldots,H_m]\!]$ can be uniquely written as
$$ [\![H_0,H_2,\ldots,H_m]\!]=
\Big[[\![H_0,\ldots,H_h]\!],[\![H_{h+1},\ldots,H_m]\!]\Big], $$
\noindent
for some $h=1,\ldots,m-1$. This $h$ will be called the {\it cut point}
of our multiple commutator. Thus, among the quadruple commutators
$[\![F,H,K,L]\!]$, two arrange\-ments, $[[[F,H],K],L]$
and $[[F,[H,K]],L]$, cut at 3; one, $[[F,H],[K,L]]$, cuts
at 2; and the remaining two, $[F,[H,[K,L]]]$ $[F,[[H,K],L]]$,
cut at 1.
\par
Now, let $I_i$, $i=0,1,\ldots,m$, be ideals of the ring $A$. Our
ultimate objective is to compute the commutator subgroups of
congruence subgroups
$$ [\![G(A,I_0),G(A,I_1),\ldots,G(A,I_m)]\!], $$
\noindent
but that is a {\it highly\/} strenuous enterprise. So far, we
have done it only for the case $G=\GL_n$, provided that
$m>\delta(A)$.
\par
In~\S\ref{nervsue} we embark on the [somewhat easier]
calculation of higher commutators of relative elementary subgroups
$$ [\![E(A,I_0),E(A,I_1),\ldots,E(A,I_m)]\!]. $$
\noindent
Even this turns out to be a rather non-trivial task. In fact,
we do not see any other way to do that, but to prove a
higher analogue of the standard commutator formula, viz.
$$ [\![E(A,I_0),G(A,I_1),\ldots,G(A,I_m)]\!]=
[\![E(A,I_0),E(A,I_1),\ldots,E(A,I_m)]\!]. $$
\noindent
This multiple commutator formula will be discussed in \S\ref{nervsue} and \S\ref{ziuso1}. Unlike the {\it general\/} multiple commutator
formula in which we are ultimately interested, and which only
works for finite-dimensional rings, this weaker formula holds
over arbitrary quasi-finite/commutative rings.
\par
Amazingly, the resulting {\it multiple\/} commutator subgroups
will always coincide with some {\it double\/} relative commutator
subgroups, depending not on the ideals $I_i$ themselves, but
only on two symmetrised products of these ideals. Since the
symmetrised product of ideals is not associative, some traces
of the initial arrang\-ment will still be visible in these
symmetrised products. However, for commu\-tative rings the
symmetrised product becomes the usual product of ideals, which
is associative, so that the result will not depend on the arrangement
itself either, but only on its cut point. We discuss these results
in \S\ref{ziuso1}.


\section{Multiple commutator formula}\label{nervsue}

The following theorem is the main result of the paper~\cite{47}. Initially, it was
conceived as {\it part\/} of the answer to a problem
proposed in~\cite{112,114}. As a matter
of fact, it turned out to be of significant independent
interest. The proof of the following result in~\cite{47} is based
on a further enhancement of relative localisation which
we outline in~\S\ref{loci32}.

\begin{The}\label{comain}
Let $A$ be  a quasi-finite $R$-algebra with identity and  $I_i $, $i=0,...,m$, be two-sided ideals of $A$. Then 
\begin{multline}\label{corll11}
\Big [E(n,A,I_0),\GL(n,A,I_1), \GL(n,A, I_2),\ldots, \GL(n,A, I_m) \Big]=\\
 \Big[E(n,A,I_0),E(n,A,I_1),E(n,A, I_2),\ldots, E(n,A, I_m) \Big].
\end{multline}
\end{The}
\begin{proof}
We prove the statement by induction. For $i=1$ this is the generalised commutator formula Theorem~\ref{gcformula1}
\[
\big[E(n,A,I_0),\GL(n,A,I_1)\big]=\big[E(n,A,I_0),E(n,A,I_1)\big].
\]
 For $i=2$, this will be proved in Theorem~\ref{the43} which will be the first step of induction.
Suppose the statement is valid for $i=m-1$ (i.e., there are $m$ ideals in the commutator formula). To prove ~(\ref{corll11}), using Theorem~\ref{the43}, we have
\begin{multline*}
 \bigg [\Big[\big[E(n,A,I_0),\GL(n,A,I_1)\big], \GL(n,A, I_2)\Big],\GL(n,A, I_3),\ldots, \GL(n,A, I_m) \bigg ]=\\
 \bigg [\Big[\big[E(n,A,I_0),E(n,A,I_1)\big ],E(n,A, I_2)\Big],\GL(n,A, I_3),\ldots, \GL(n,A, I_m) \bigg].
\end{multline*}
 By Lemma~\ref{Lem:Habdank}, $[E(n,A,I_0),E(n,A,I_1)]\le  \GL(n,A,I_0I_1+I_1I_0)$. Thus 
\begin{multline*}
 \bigg [\Big[\big[E(n,A,I_0),E(n,A,I_1)\big ],E(n,A, I_2)\Big],\GL(n,A, I_3),\ldots, \GL(n,A, I_m) \bigg] \le  \\
  \bigg [\Big[\GL(n,A,I_0I_1+I_1I_0),E(n,A, I_2)\Big],\GL(n,A, I_3),\ldots, \GL(n,A, I_m) \bigg].
 \end{multline*}
Since there are $m$ ideals involved in the commutator subgroups in the right hand side, by induction we get 
\begin{multline*}
 \bigg [\Big[\GL(n,A,I_0I_1+I_1I_0),E(n,A, I_2)\Big],\GL(n,A, I_3),\ldots, \GL(n,A, I_m) \bigg]=\\
 \bigg [\Big[E(n,A,I_0I_1+I_1I_0),E(n,A, I_2)\Big],E(n,A, I_3),\ldots, E(n,A, I_m) \bigg].
\end{multline*}

Finally again by Lemma~\ref{Lem:Habdank}, \[E(n,A,I_0I_1+I_1I_0) \le  \big[E(n,A,I_0),E(n,A,I_1)\big].\] Replacing this in the above equation we obtain that the left hand side of~(\ref{corll11}) is contained in the right hand side. The opposite inclusion is obvious. This completes the proof. 
\end{proof}

\begin{TheA}\label{nemorh3}
Let\/ $n\ge 3$, let $A$ be a quasi-finite ring
with\/ $1$ and let\/ $I_i\unlhd A$, $i=0,\ldots,m$, be
ideals of\/ $A$. Then one has
\begin{multline*}
[\![E(n,A,I_0),\GL(n,A,I_2),\ldots,\GL(n,A,I_m)]\!]=\\
[\![E(n,A,I_0),E(n,A,I_2),\ldots,E(n,A,I_m)]\!].
\end{multline*}
\end{TheA}

In this theorem the arrangement of brackets on the left hand
side may be arbitrary. But it is essential that the placement of
brackets on the right hand side coincides with that on the left hand
side. Without this assumption the equality may fail dramatically,
even if all factors are elementary, as we shall see in~\S\ref{jalg7}.
Of course, the same observation applies to the theorems below.
\par
For unitary groups, similar result is established in~\cite{44}, by
essentially the same method. However, as one could expect,
the necessary calculations are tangibly more complicated and
require a completely different level of technical strain.

\begin{TheB}\label{nemorh4}
Let\/ $n\ge 3$ and let\/ $(\Form)$ be a form ring
such that\/ $A$ is a quasi-finite\/ $R$-algebra over a commutative
ring\/ $R$. Further, let\/ $(I_i,\Gamma_i)$, $i=0,\ldots,m$,
be form ideals of\/ $(\Form)$. Then
\begin{multline*}
[\![\EU(2n,I_0,\Gamma_0),\GU(2n,I_1,\Gamma_1),\ldots,\GU(2n,I_m,\Gamma_m)]\!]=\\
[\![\EU(2n,I_0,\Gamma_0),\EU(2n,I_1,\Gamma_1),\ldots,\EU(2n,I_m,\Gamma_m)]\!].
\end{multline*}
\end{TheB}

Finally, let us pass to Chevalley groups. We believe that at this
point we possess two independent proofs of the following result.
One of them, by the authors, is
conventional, and involves an further elaboration of the relative
commutator calculus in the style of~\cite{43}. Another one, by A. Stepanov, is somewhat shorter, and employs his method of universal
localisation~\cite{86}. But the definitive expositions are still missing.

\begin{TheC}
Let\/ $\rk(\Phi)\ge 2$ and let\/ $I_i\unlhd A$, $i=0,\ldots,m$, be
ideals of a commutative ring\/ $A$. In the cases\/ $\Phi=\B_2,\G_2$
assume that $A$ does not have residue fields\/ ${\Bbb F}_{\!2}$
of\/ $2$ elements and in the case\/ $\Phi=\B_2$ assume additionally
that any\/ $c\in A$ is contained in the ideal\/ $c^2A+2cA$.
\par
Then one has
\begin{multline*}
[\![E(\Phi,A,I_0),G(\Phi,A,I_),\ldots,G(\Phi,A,I_m)]\!]=\\
[\![E(\Phi,A,I_0),E(\Phi,A,I_1),\ldots,E(\Phi,A,I_m)]\!].
\end{multline*}
\end{TheC}

These theorems are broad generalisations of the double commutator
formulas. Let us explain, why they do not reduce to the double
formula. Consider {\it three\/} ideals $I,J,K$ of $A$ and form the
commutator $[[E(A,I),G(A,J)],G(A,K)]$. The double commutator
formula implies that
$$ [[E(A,I),G(A,J)],G(A,K)]=[[E(A,I),E(A,J)],G(A,K)]. $$
\noindent
But as we know, the relative commutator subgroup $[E(A,I),E(A,J)]$
may be strictly larger, than $E(A,I\circ J)$ (see~\S\ref{countexam}), so it is not at all
clear, why the equality 
$$ [[E(A,I),E(A,J)],G(A,K)]=[[E(A,I),E(A,J)],E(A,K)]$$
should hold. 
\par
This is indeed the key new leap in the proof of Theorem~\ref{comain}, and the
commutator calculus developed in~\cite{46,42,43} is not powerful 
enough here. This step requires a new layer of the relative commutator 
calculus, which we discuss in~\S\ref{loci32}.

\section{Multiple $\rightsquigarrow$ double}\label{ziuso1}

\subsection{}\label{jalg7}
In connection with Theorems 6 and 7 it is natural to ask, whether
the equality
\begin{equation} \label{ogwt1}
[[E(A,I),E(A,J)],E(A,K)]=[E(A,I),[E(A,J),E(A,K)]] 
\end{equation}
holds for any three ideals $I$, $J$ and $K$ of $A$. If this were
the case, one could drop the requirement that the arrangement of
brackets on the left hand side and the right hand side of these
theorems should coincide.
\par
However, in general this equality fails, as can be shown by
easy examples. Let us retreat to the case of $\GL_n$. In fact,
setting here $K=A$ we see that
\begin{multline*}
E(n,A,I\circ J)=[E(n,A,I\circ J),E(n,A)]\le\\
[[E(n,A,I),E(n,A,J)],E(n,A)]\le[\GL(n,A,I\circ J),E(n,A)]=\\
[E(n,A,I\circ J),E(n,A)]=E(n,A,I\circ J).
\end{multline*}
\noindent
This shows that 
\[[[E(A,I),E(A,J)],E(A,K)]=E(n,A,I\circ J).\]
 On the other
hand, for $K=A$, we have
\[[E(A,I),[E(A,J),E(A,K)]]=[E(A,I),[E(A,J)].\]
Thus, in this case if the associativity of commutators (\ref{ogwt1}) holds, we obtain
$$ [E(n,A,I),E(n,A,J)]=E(n,A,I\circ J). $$
\noindent
However, as we know from the example provided in~\S\ref{countexam}, this equality does not hold, in general.
\par
\subsection{} To motivate the next theorem, let us calculate these triple
commutators. Combining Lemma~\ref{Lem:Habdank} and Theorem~\ref{gcformula1}, we see that
\begin{multline*}
[E(n,A,I\circ J),E(n,A,K)]\le  \Big[[E(n,A,I),E(n,A,J)],E(n,A,K)\Big]\le \\
[\GL(n,A,I\circ J),E(n,A,K)]=[E(n,A,I\circ J),E(n,A,K)].
\end{multline*}
\noindent
In other words,
$$ \Big[[E(n,A,I),E(n,A,J)],E(n,A,K)\Big]=[E(n,A,I\circ J),E(n,A,K)]. $$
\noindent
Similarly, one can verify that
$$ \Big[E(n,A,I),[E(n,A,J),E(n,A,K)]\Big]=[E(n,A,I),E(n,A,J\circ K)]. $$
\par
Plugging in the above calculation Theorem~\ref{nemorh3} instead of Theorem~\ref{gcformula1},
we get the following amazing corollary. It asserts that multiple
commutators of relative elementary subgroups can always be
expressed as {\it double\/} commutators of such subgroups,
corresponding to some symmetrised product ideals. The following is
observed in~\cite{38}.
\begin{TheA}
Let\/ $A$ be a quasi-finite ring with $1$ and let\/
$I_i\unlhd A$, $i=0,\ldots,m$, be ideals of\/ $A$.
Consider an arbitrary configuration of brackets\/ $[\![\ldots]\!]$
and assume that the outermost pairs of brackets between
positions\/ $h$ and\/ $h+1$. Then one has
\begin{multline*}
[\![E(n,A,I_0),E(n,A,I_1),\ldots,E(n,A,I_m)]\!]=\\
[E(n,A,I_0\circ\ldots\circ I_h),E(n,A,I_{h+1}\circ\ldots\circ I_m)],
\end{multline*}
where the bracketing of symmetrised products on the right hand side
coincides with the bracketing of the commutators on the left hand side.
\end{TheA}
\begin{proof}
Alternated application of Lemma~\ref{Lem:Habdank} and Theorem~\ref{gcformula1}
shows that
\begin{multline*}
\Big[\big\llbracket E(n,A,I_0),E(n,A,I_1),\ldots,E(n,A,I_{k})\big\rrbracket,
\big\llbracket E(n,A,I_{k+1}),\ldots,E(n,A,I_m)\big\rrbracket\Big]\le  \\
\Big[\GL(n,A,I_0\circ\ldots\circ I_k\big),
\llbracket E(n,A,I_{k+1}),\ldots,E(n,A,I_m)\big\rrbracket\Big]=\\
\Big[E(n,A,I_0\circ\ldots\circ I_k\big),
\llbracket E(n,A,I_{k+1}),\ldots,E(n,A,I_m)\big\rrbracket\Big]\le  \\
\big[E(n,A,I_0\circ\ldots\circ I_k\big),\GL(n,A,I_{k+1}\circ\ldots\circ I_m)\big]=\\
\big[E(n,A,I_0\circ\ldots\circ I_k\big),E(n,A,I_{k+1}\circ\ldots\circ I_m)\big]\le \\
\Big[\big\llbracket E(n,A,I_0),E(n,A,I_1),\ldots,E(n,A,I_{k})\big\rrbracket, \big
\llbracket E(n,A,I_{k+1}),\ldots,E(n,A,I_m)\big\rrbracket\Big],
\end{multline*}
as claimed.
\end{proof}

For the unitary case it is \cite[Theorem 7]{44}.
\begin{TheB}
 Let\/ $(\Form)$ be a quasi-finite ring with $1$
and let\/ $(I_i,\Gamma_i)$, $i=0,\ldots,m$, be form ideals
of the form ring\/ $(\Form)$. Consider an arbitrary configuration
of brackets\/ $[\![\ldots]\!]$ and assume that the outermost
pairs of brackets between positions\/ $h$ and\/ $h+1$. Then
one has
\begin{multline*}
[\![\EU(2n,I_0,\Gamma_0),\EU(2n,I_1,\Gamma_1),\ldots,\EU(2n,I_m,\Gamma_m)]\!]=\\
[\EU(2n,(I_0,\Gamma_0)\circ\ldots\circ(I_h,\Gamma_h)),
\EU(n,(I_{h+1},\Gamma_{h+1})\circ\ldots\circ(I_m,\Gamma_m))].
\end{multline*}
\end{TheB}
Of course, similar result also holds in the context of Chevalley
groups, once we have Theorem~6C.
\refstepcounter{TheC}
\begin{TheC}
Let\/ $A$ be a commutative ring with\/ $1$ and let\/
$I_i\unlhd A$, $i=0,\ldots,m$, be ideals of\/ $A$.
Consider an arbitrary configuration of brackets\/ $[\![\ldots]\!]$
and assume that the outermost pairs of brackets between
positions\/ $h$ and\/ $h+1$. Then one has
\begin{multline*}
[\![E(\Phi,A,I_0),E(\Phi,A,I_1),\ldots,E(\Phi,A,I_m)]\!]=\\
[E(\Phi,A,I_0\ldots I_h),E(\Phi,A,I_{h+1}\ldots I_m)].
\end{multline*}
\end{TheC}

\section{Localisation}

\subsection{} \label{fdye7}
In this  paper we only use central localisation. Namely, for an $R$-algebra $A$, we consider the localisation with respect to a multiplicative closed subset of $R$. 

First,
we fix some notation. Let $R$ be a commutative ring with
1, $S$ be a multiplicative closed subset in $R$ and $A$ be an $R$-algebra. Then $S^{-1}R$ and $S^{-1}A$ are the
corresponding localisation. We mostly use localisation
with respect to the two following types of multiplicative
systems.
\par\smallskip
$\bullet$ {\it Principal localisation\/}: $S$ coincides with
$\langle s\rangle=\{1,s,s^2,\ldots\}$, for some non-nilpotent
$s\in R$, in this case we usually write $\langle s\rangle^{-1}R=R_s$ and   $\langle s\rangle^{-1}A=A_s$.
\par\smallskip
$\bullet$ {\it Localisation at a maximal ideal}: $S=R\setminus\mm$,
for some maximal ideal $\mm\in\Max(R)$ in $R$, in this case
we usually write $(R\setminus\mm)^{-1}R=R_\mm$ and $(A\setminus\mm)^{-1}A=A_\mm$
\par\smallskip
We denote by $F_S:A\map S^{-1}A$ the canonical ring homomorphism
called the {\it localisation homomorphism\/}. For the two special
cases above, we write $F_s:A\map A_s$ and $F_\mm:A\map A_\mm$,
respectively.
\par
When we write an element as a fraction, like $a/s$ or
$\displaystyle{\frac{a}{s}}$ we {\it always\/} think of it
as an element of some localisation $S^{-1}A$, where $s\in S$.
If $s$ were actually invertible in $R$, we would have written
$as^{-1}$ instead.
\par
Ideologically, all proofs using localisations are based on the
interplay of the three following observations:
\par\smallskip
$\bullet$ Functors of points $A\rightsquigarrow G(A)$ are
compatible with localisation,
$$ g\in G(A)\qquad\aeq\qquad F_\mm(g)\in G(A_\mm),
\quad\text{for all\ } \mm\in\Max(A). $$
\par$\bullet$ Elementary subfunctors $A\rightsquigarrow E(A)$
are compatible with factorisation, for any $I\trianglelefteq A$
the reduction homomorphism $\rho_I:E(A)\map E(A/I)$ is surjective.
\par\smallskip
$\bullet$ On a [semi-]local ring $A$ the values of semi-simple
groups and their elementary subfunctors coincide, $G(A)=E(A)$.
\par\smallskip
The following property of the functors $G$ and $E$
will be crucial for what follows: they are continuous functors, i.e., they {\it commute with direct
limits\/}. In other words, if $A=\varinjlim A_i$, where
$\{A_i\}_{i\in I}$ is an inductive system of rings, then
$$ G(\varinjlim A_i)=\varinjlim G(A_i),  \qquad E(\varinjlim A_i)=\varinjlim E(A_i) . $$
\noindent
We use this property in the two following situations.
\par\smallskip
$\bullet$ {\it Noetherian reduction\/}: let $A_i$ be the inductive
system of all finitely generated subrings of $A$ with respect to
inclusion. Then
$$ G(A)=\varinjlim G(A_i), \qquad E(A)=\varinjlim E(A_i). $$
\noindent
This allows to reduce most of the proofs to the case of
Noetherian rings.
\par\smallskip
$\bullet$ {\it Reduction to principal localisations\/}:
let $S$ be a multiplicative closed set in $R$ and let $A_s$, $s\in S$,
be the corresponding inductive system with respect to the
principal localisation homomorphisms: $F_{t}:A_s\map A_{st}$.
Then
$$ G(S^{-1}A)=\varinjlim G(A_s), \qquad E(S^{-1}A)=\varinjlim E(A_s).$$
\noindent
This reduces localisation in any multiplicative system to
the principal localisation.


\subsection{Injectivity of localisation homomorphism}

Most localisation proofs rely on the injectivity of localisation
homomorphism $F_S$. As observed in~\S\ref{fdye7}, we can
only consider {\it principal\/} localisation homomorphisms $F_s$.
Of course, $F_s$ is injective when $s$ is regular. Thus,
localisation proofs are particularly easy for integral domains.
A large part of what follows are various devices to fight with
the presence of zero-divisors.
\par
When $s$ is a zero-divisor, $F_s$ is not injective on the group
$G(A)$ itself. But its restrictions to appropriate congruence
subgroups often are. Here is an important typical case, i.e., 
Noetherian ring.

\begin{Lem}\Label{Lem:03}
Let $A$ be a module finite $R$-algebra, where $R$  is a commutative Noetherian ring. Then for any $s\in R$,  there exists a positive integer $l$ such that the homomorphism 
$F_s: \GL(n,A, s^l A)\longrightarrow \GL(n,A_s)$ is injective.
\end{Lem}
\begin{proof}
The homomorphism $F_s: \GL(n,A, s^l A)\longrightarrow \GL(n,A_s)$ is
injective whenever $F_s:s^kA\map A_s$ is injective.
Let $\ma_i=\Ann_R(s^i)$ be the annihilator of $s^i$ in $A$.
Since $R$ is Noetherian, and $A$ is finite over $R$, $A$ is Noetherian and so there exists $k$ such that
$\ma_k=\ma_{k+1}=\cdots$. If $s^ka$ vanishes in $A_s$, then
$s^is^ka=0$ for some $i$. But since $\ma_{k+i}=\ma_k$, already
$s^ka=0$ and thus $s^kA$ injects in $A_s$.
\end{proof}

Another important trick to override the presence of zero-divisors
consists in throwing in polynomial variables. Namely, instead
of the ring $R$ itself we consider the polynomial ring $R[t]$
in the variable $t$. In that ring $t$ is not a zero-divisor,
so that the localisation homomorphism $F_t$ is injective.
We can use that, and then specialise $t$ to any $s\in R$.
\par
Actually, throwing in polynomial variables has more than one
use. The elementary subfunctors $R\rightsquigarrow E(R)$ are
compatible with localisation, i.e., 
$$ g\in E(R)\qquad\seq\qquad F_\mm(g)\in E(R_\mm),
\quad\text{for all\ } \mm\in\Max(R), $$
\noindent
but the converse implication does not hold, for otherwise
$E(R)$ would coincide with the [semi-simple part of] $G(R)$
for all commutative rings.
\par
The following remarkable observation was due to Daniel Quillen
at the level of $\K_0$, and was first applied by Andrei Suslin
at the level of $\K_1$, in the context of solving Serre's
conjecture, and its higher analogues~\cite{90}. See~\cite{58} for a
description of Quillen--Suslin's idea in its historical
development. We refer to the following result as Quillen--Suslin's
lemma.
\begin{The}
Let $g\in G(R[t],tR[t])$. Then $g\in E(R[t])$ if and only if  $F_\mm(g)\in E(R_\mm[t])$,
for all  $\mm\in\Max(R).$
\end{The}

\subsection{}\label{localization}
Let $(A,\Lambda)$ be a form algebra over a commutative ring $R$ with $1$,
and let $S$ be a multiplicative subset of $R_0$, (see \S\ref{quasi-finite}).
For any $R_0$-module $M$ one can consider its localisation $S^{-1}M$
and the corresponding localisation homomorphism $F_S:M\map S^{-1}M$.
By definition of the ring $R_0$ both $A$ and $\Lambda$ are $R_0$-modules,
and thus can be localised in $S$.

\subsection{Localisation of form rings} \label{hgbdg43a}
In the setting of form rings, we need to adjust the ground field of the localisation. For a form ring $(A,\Lambda)$, where $A$ is an $R$-algebras, the form $\Lambda$ is not necessarily an $R$-module (see~\S\ref{quasi-finite}). This forces us to
replace $R$ by its subring $R_0$, generated by all $\alpha\bar\alpha$
with $\alpha\in R$. Clearly, all elements in $R_0$ are invariant with
respect to the involution, i.~e.\ $\bar r=r$, for $r\in R_0$. Furthermore, $\Lambda$ is an $R_0$-module.

As in the setting of general linear group (\S\ref{fdye7}), we mostly use localisation in the unitary setting with respect to the
following two types of multiplication closed subsets of $R_0$.
\par\smallskip
$\bullet$ {\it Principal localisation\/}: for any $s\in R_0$ with $\bar s=s$,
the multiplicative closed subset generated by $s$ is defined as
$\langle s\rangle=\{1,s,s^2,\ldots\}$. The localisation of the form algebra
$(\FormR)$ with respect to multiplicative system $\langle s\rangle$ is usually
denoted by $(A_s,\Lambda_s)$, where as usual $A_s=\langle s\rangle^{-1}A$ and
$\Lambda_s=\langle s\rangle^{-1}\Lambda$ are the usual principal localisations
of the ring $A$ and the form parameter $\Lambda$.
Notice that, for each $\alpha\in A_s$, there exists an integer $n$ and an
element $a\in A$ such that $\displaystyle\alpha=\frac a{s^n}$, and for
each $\xi\in\Lambda_s$, there exists an integer $m$ and an element
$\zeta\in\Lambda$ such that $\displaystyle\xi=\frac\zeta{s^m}$.
\par\smallskip
$\bullet$ {\it Maximal localisation\/}: consider a maximal ideal $\gm\in\Max(R_0)$
of $R_0$ and the multiplicative closed set $S_{\gm}=R_0\backslash\gm$. We
denote the localisation of the form algebra $(\FormR)$ with respect to $S_{\gm}$
by $(A_\gm,\Lambda_\gm)$, where $A_\gm=S_{\gm}^{-1}A$ and
$\Lambda_\gm=S_{\gm}^{-1}\Lambda$ are the usual maximal localisations of the
ring $A$ and the form parameter, respectively.
\par\smallskip
In these cases the corresponding localisation homomorphisms will be
denoted by $F_s$ and by $F_{\gm}$, respectively.
\par
The following fact is verified by a straightforward computation.
\begin{Lem}
For any\/ $s\in R_0$ and for any\/ $\gm\in\Max(R_0)$ the
pairs\/ $(A_s,\Lambda_s)$ and\/ $(A_\gm,\Lambda_\gm)$ are form rings.
\end{Lem}





\section{Triple Commutators/Base of induction}\label{loci32}

We prove Theorem~\ref{comain}  by induction on $m$. The case of $m=2$ is precisely the relative commutator
formula, Theorem~\ref{gcformula1}. However, the base of induction for Theorem~\ref{comain} is $i=2$, and it is the most demanding part of the induction step.
 In fact, the proof of the following
special case constitutes bulk of the paper~\cite{47}.

\begin{TheA}\label{the43}
Let\/ $n\ge 3$, and let\/ $A$ be a quasi-finite ring.
Further, let\/ $I$, $J$ and\/ $K$ be three two-sided ideals
of\/ $A$. Then
$$ \Big[[E(n,A,I),\GL(n,A,J)],\GL(n,A,K)\Big]=
\Big[[E(n,A,I),E(n,A,J)],E(n,A,K)\Big]. $$
\end{TheA}
As we have just observed, the standard commutator formula, Theorem~\ref{gcformula1},
implies that
$$ \Big[[E(n,A,I),\GL(n,A,J)],\GL(n,A,K)\Big]=
\Big[[E(n,A,I),E(n,A,J)],\GL(n,A,K)\Big]. $$
\noindent
Thus, to prove Theorem~\ref{the43} it remains to establish the
following equality
\begin{equation}\label{ytyht}
 \Big[[E(n,A,I),E(n,A,J)],\GL(n,A,K)\Big]=
\Big[[E(n,A,I),E(n,A,J)],E(n,A,K)\Big]. 
\end{equation}

\noindent
However, this last equality does not follow from the standard
commutator formula. To establish this, we shall use the general ``yoga of commutators'' which is developed in~\cite{46}  based on the work of Bak on  localisation and patching in general linear groups (see~\cite{8,36} and~\cite[\S13]{41}).  In order to make use of this method, one needs to overcome two problems: firstly to devise an appropriate conjugation calculus to approach the identity~(\ref{ytyht}) and secondly to perform the actual calculations. Both of these problems are equally challenging as the nature of the conjugation calculus depends on the problem in hand. In fact the term  yoga of commutators is chosen
to stress the overwhelming feeling of technical strain and exertion. 

In this section we prove Theorem~\ref{the43}, following~\cite{47}. 
We need the following elementary conjugation calculus, which are Lemmas~7, 8 and 11  from~\cite{46}, respectively. 
Note that in Equations~\ref{lem5}, \ref{lem8} and~\ref{lem11} the calculations take place in the group $E(n,A_t)$.

\begin{Lem}[cf. ~\cite{46}]\Label{LemHZ}
Let $A$ be a module finite $R$-algebra, $I,J$ two-sided ideals of $A$, $a,b,c\in A$ and  $t\in R$. If $m,l$ are given, there is an integer $p$ such that 
\begin{equation}\label{lem5}
  {}^{ E^1(n,\frac{c}{t^m})}E(n, t^pA,t^{ p}\langle  a\rangle)\le  E(n, t^lA,t^{ l}\langle  a\rangle),
\end{equation}
there is an integer $p$ such that 
\begin{equation}\label{lem8}
  {}^{ E^1(n,\frac{c}{t^m})}\big[E(n, t^pA,t^{ p}\langle  a\rangle), E(n,t^pA,t^{ p}\langle  b\rangle)\big]\le  \big[E(n, t^lA,t^{ l}\langle  a\rangle), E(n,t^lA,t^{ l}\langle  b\rangle)\big],
\end{equation}
and there is an integer $p$ such that 
\begin{equation}\label{lem11}
\Big[E(n,t^p A, t^{ p}  I), E^1\big(n,\frac{J}{t^m}\big)\Big]\le  \big[E(n,t^l A, t^{ l}  I), E(n,t^lA,t^{ l} J)\big].
\end{equation}
\end{Lem}

By   Lemma~\ref{LemHZ}, one easily obtains the following result. The proof is left to the reader.
\begin{Lem}\Label{LemHZ1}
Let $A$ be a module finite $R$-algebra, $I,J$ two-sided ideals of $A$ and  $t\in R$. If $m,l, L$ are given, there is an integer $p$ such that 
\begin{equation}\label{lemnew}
\Big[E(n,t^p A, t^{ p}  I), {}^{E^L\big(n,\frac{A}{t^m}\big)}E^1\big(n,\frac{J}{t^m}\big)\Big]\le  \big[E(n,t^l A, t^{ l}  I), E(n,t^lA,t^{ l} J)\big].
\end{equation}
\end{Lem}

Denote by $E^L\big(n,\frac{A}{t^m}, \frac{K}{t^m}\big)$ the product of $L$ elements (or fewer) of the form \[^{E^1(n,\frac{A}{t^m})}E^1\big(n,\frac{K}{t^m}\big).\] In the following two Lemmas, as in Lemma~\ref{LemHZ}, all the calculations take place in the fraction ring $A_t$. All the subgroups of $\GL(n,A_t)$ used in the Lemmas, such as the ones denoted by $E(n,A,I)$ or $\GL(n,A,J)$, are in fact the homomorphic images of these subgroups in $\GL(n,A)$ under the natural homomorphism $A\rightarrow A_t$. Since  lemmas such as Lemma~\ref{Lem:Habdank} and the generalised commutator formula (Theorem~\ref{gcformula1}) hold for these subgroups in $\GL(n,A)$, they also hold for their corresponding homomorphic images in $\GL(n,A_t)$.

\begin{Lem}\Label{Lem:New1}
Let $A$ be an $R$-algebra, $I,J$ two-sided ideals of $A$, and  $t\in R$.  For any given $e\in \GL(n,A_t,J_t)$ and an integer $l$, there is an integer $p$ such that for any $g\in \GL(n,A,t^p I )$ 
$$
[e,g]\in \GL\big(n,{A, t^l (IJ+JI)}\big).
$$
\end{Lem}
\begin{proof}
Note that all the entries of $g-1$ and  $g^{-1}-1$ are in $t^pI$ (to emphasize our convention, they are in the image of $t^pI$ under the homomorphism $\theta:A\rightarrow A_t$) and all the entries of $e-1$ and $e^{-1}-1$ are in $J_t$. Choose $k\in \mathbb N$  such that one can write all the entries of $e-1$ and $e^{-1}-1$ in the form $j/t^k$, $j \in J$.
Let 
\begin{eqnarray*}
g=1+\ep& \quad \text{and}\quad &g^{-1}=1+\ep'\\
e=1+\delta& \quad \text{and}\quad &e^{-1}=1+\delta'.
\end{eqnarray*}
A straightforward computation shows that 
\begin{align*}
& \ep+\ep'+\ep\ep'=\ep+\ep'+\ep'\ep=0\\
& \delta +\delta'+\delta\delta'=\delta +\delta'+\delta'\delta=0.
\end{align*}
By the equalities above, one has
$$[e,g]=[1+\delta,1+\ep]=1+\delta'\ep'+\ep\delta'+\ep\delta'\ep'+\delta\delta'\ep'+\delta\ep\delta'+\delta\ep\delta'\ep'.
$$
So  the entries of $[e,g]-1$  belong to  $t^{p-2k}(IJ +JI)$. We finish the proof by choosing $p\ge l+2k$. 
\end{proof}

\begin{Lem}\Label{Lem:08}
Let $A$ be a module finite $R$-algebra, $I,J,K$  two-sided ideals of  $A$ and $t\in R$.  For any given $e_2\in E(n,A_t,K_t)$ and an integer $l$, there is a  sufficiently large integer $p$, such that 
\begin{equation}\label{eqn:l1}
[e_1,e_2]\in \Big[\big[E(n,A,t^lI), E(n,A,t^lJ)\big], E(n,A,t^lK)\Big],
\end{equation}
where  $e_1\in [E(n,t^pI), E(n,A,J)]$.
\end{Lem}
\begin{proof}
%
For any given $e_2\in E(n,A_t, K_t)$, one may find some positive integers $m$ and $L$ such that
$$
e_2\in E^L\big(n,\frac{A}{t^m},\frac{K}{t^m}\big).
$$
Applying the identity (C1$^+$) and repeated application of (\ref{lem5}) in Lemma~\ref{LemHZ}, we reduce the problem to show that 
$$
\Big[[E(n,t^pI), E(n,A,J)], {}^ce_{i',j'}(\frac{\gamma}{t^m})\Big] \le  \Big[\big[E(n,A,t^lI), E(n,A,t^lJ)\big], E(n,A,t^lK)\Big],
$$
where $c\in E^1(n,\frac{A}{t^m})$ and  $\gamma\in K$.
We further decompose   \[e_{i',j'}(\frac{\gamma}{t^m})=[e_{i',k}(t^{p'}),e_{k,j'}(\frac{\gamma}{t^{m+p'}})]\] for some integer $p'$. Then
\begin{equation*}
 \Big [e_1, {}^ce_{i',j'}(\frac{\gamma}{t^m})\Big]=
\Big [e_1, \big [{}^ce_{i',k}(t^{p'}),{}^c e_{k,j'}(\frac{\gamma}{t^{m+p'}})\big]\Big].
\end{equation*}
We use a variant of the Hall-Witt identity (see (C3)) 
$$ \big[x,[y^{-1},z]\big]=  {}^{y^{-1}x}\big[[x^{-1},y],z\big]\, \, {}^{y^{-1}z}\big[[z^{-1},x],y\big],$$ to obtain 
\begin{multline}\label{eqn:99}
\Big [e_1, \big [{}^ce_{i',k}(t^{p'}),{}^c e_{k,j'}(\frac{\gamma}{t^{m+p'}})\big]\Big]=\\
{}^{y^{-1}x}\bigg[\Big[e_1^{-1}, {}^ce_{i',k}(-t^{p'})\Big],{}^c e_{k,j'}(\frac{\gamma}{t^{m+p'}})\bigg] \times \\
{}^{y^{-1}z}\bigg[\Big[{}^c e_{k,j'}(\frac{-\gamma}{t^{m+p'}}), e_1\Big],{}^ce_{i',k}(-t^{p'})\bigg],
\end{multline}
where $x=e_1$, $y={}^ce_{i',k}(-t^{p'})$, $z={}^c e_{k,j'}(\frac{\gamma}{t^{m+p'}})$ and as before $c\in E^1(n,\frac{A}{t^{m}})\le  E^1(n,\frac{A}{t^{m+p'}})$. 
We will look at each of the two factors of~(\ref{eqn:99}) separately.  

By (\ref{lem5}) in Lemma~\ref{LemHZ}, for any given $p''$, one may find a sufficiently large $p'$ such that 
\begin{equation}\Label{eqn:100}
y={}^ce_{i',k}(-t^{p'})\in E(n,t^{p''}A, t^{p''}A) \le  E(n,A).
\end{equation}
Then 
\begin{eqnarray*}
\Big[e_1^{-1}, {}^ce_{i',k}(-t^{p'})\Big] &\in&\big[[E(n,t^p I), E(n,A,J)], E(n,A)\big]\\
&\le  &\big [\GL(n,A, t^p(IJ+JI)), E(n,A)\big]\\
&\le  & E(n,A, t^p(IJ+JI)).
\end{eqnarray*}
Set $p_1=p$. Thanks to Lemma~\ref{Lem:Habdank}, 
\begin{multline}E\big(n,A, t^{p_1}(IJ+JI)\big)\le  \Big[E(n,t^{\LF {\frac{p_1}{2}\RF}}A), E\big(n,t^{\LF {\frac{p_1}{2}\RF}}(IJ+JI)\big)\Big]\le  \\
E\big(n,t^{\LF {\frac{p_1}{2}\RF}}A, t^{\LF {\frac{p_1}{2}\RF}}(IJ+JI)\big).
\end{multline}
Hence  we obtain that
\begin{multline*}
{}^{y^{-1}x}\bigg[\Big[e_1^{-1}, {}^ce_{i',k}(-t^{p'})\Big],{}^c e_{k,j'}(\frac{\gamma}{t^{m+p'}})\bigg]
\in \\ {}^{y^{-1}x}\bigg[E(n,t^{\LF {\frac{p_1}{2}\RF}}A, t^{\LF {\frac{p_1}{2}\RF}}(IJ+JI)),{}^c e_{k,j'}(\frac{\gamma}{t^{m+p'}})\bigg], 
\end{multline*}
where $x\in [E(n,t^{p_1} I), E(n,A,J)]$, $y\in  E(n,t^{p''}A, t^{p''}A)$. By Lemma~\ref{LemHZ1}, for any given integer $l'$ we may find a sufficiently large $p_1$ such that
\begin{align*}
{}^{y^{-1}x}\bigg[E(n,t^{\LF {\frac{p_1}{2}\RF}}&A, t^{\LF {\frac{p_1}{2}\RF}}(IJ+JI)),{}^c e_{k,j'}(\frac{\gamma}{t^{m+p'}})\bigg] \\
& \in \,  {}^{y^{-1}x}\Big[E\big (n,t^{2l'}A, t^{2l'}(IJ+JI)\big), E(n,t^{2l'}A, t^{2l'}K)\Big]\\
&\le  {}^{y^{-1}x}\Big[\big[E(n,t^{l'}A, t^{l'} I),E(n,t^{l'}A, t^{l'}J)\big], E(n,t^{2l'}A, t^{2l'}K)\Big]\\
&\le  {}^{y^{-1}x}\Big[\big[E(n,t^{l'}A, t^{l'} I),E(n,t^{l'}A, t^{l'}J)\big], E(n,t^{l'}A, t^{l'}K)\Big]\\
&=\Big[\big[{}^{y^{-1}x}E(n,t^{l'}A, t^{l'} I),{}^{y^{-1}x}E(n,t^{l'}A, t^{l'}J)\big], {}^{y^{-1}x}E(n,t^{l'}A, t^{l'}K)\Big],
\end{align*}
where by definition $y^{-1}x\in E(n,\frac{A}{t^{0}},\frac{A}{t^{0}})$. By (\ref{lem5}) in Lemma~\ref{LemHZ}, for any given integer $l$, we may find a sufficiently large $l'$ such that
\begin{multline*}
{}^{y^{-1}x}\Big[\big[E(n,t^{l'}A, t^{l'} I),E(n,t^{l'}A, t^{l'}J)\big], E(n,t^{l'}A, t^{l'}K)\Big]\le \\ \Big[\big[E(n,t^{l}A, t^{l} I),E(n,t^{l}A, t^{l}J)\big], E(n,t^{l}A, t^{l}K)\Big].
\end{multline*}
This shows that for any given $l$, one may find a sufficiently large $p_1$ such that the first factor of (\ref{eqn:99})
\begin{multline*}
{}^{y^{-1}x}\bigg[\Big[e_1^{-1}, {}^ce_{i',k}(-t^{p'})\Big],{}^c e_{k,j'}(\frac{\gamma}{t^{m+p'}})\bigg] \in \\
\Big[\big[E(n,t^{l}A, t^{l} I),E(n,t^{l}A, t^{l}J)\big], E(n,t^{l}A, t^{l}K)\Big].
\end{multline*}

Next we consider the second factor of (\ref{eqn:99}),
$$
{}^{y^{-1}z}\bigg[\Big[{}^c e_{k,j'}(\frac{-\gamma}{t^{m+p'}}), e_1\Big],{}^ce_{i',k}(-t^{p'})\bigg].$$
Set $p_2=p$. Note that 
$$e_1\in \big[E(n, t^{p_2} I),E(n,A,  J)\big]\le  \GL\big(n,A, t^{p_2}(IJ+JI)\big)$$
and
$$
{}^c e_{k,j'}(\frac{\gamma}{t^{m+p'}})\in  {}^{E^1(n,\frac A {t^{m+p'}})}E^1(n,\frac K {t^{m+p'}}),
$$ where $p'$ is given by~(\ref{eqn:100}) from the first part of the proof. 
We may apply Lemma~\ref{Lem:New1} to find a sufficiently large $p_2$ such that 
\begin{equation}\label{jhafyw}
\Big[{}^c e_{k,j'}(\frac{-\gamma}{t^{m+p'}}), e_1\Big]\in  \GL\Big(n,A, t^{p''} (K(IJ+JI)+(IJ+JI)K)\Big)
\end{equation}
for any given $p''$. Using the commutator formula together with (\ref{eqn:100}), one gets
\begin{multline*}
{}^{y^{-1}z}\bigg[\Big[{}^c e_{k,j'}(\frac{-\gamma}{t^{m+p'}}), e_1\Big],{}^ce_{i',k}(-t^{p'})\bigg] \in \\
{}^{y^{-1}z} E\Big(n,A, t^{p''} \big(K(IJ+JI)+(IJ+JI)K\big)\Big).
\end{multline*}
Applying  Lemma~\ref{Lem:Habdank} twice, one gets
\begin{multline*}
E\Big(n,A, t^{p''} \big(K(IJ+JI)+(IJ+JI)K\big)\Big) \le \\  \bigg[E\Big(n,t^{\LF\frac {2p''}{3}\RF} \big((IJ+JI)+(IJ+JI)\big)\Big),E\Big(n,t^{\LF\frac {p''}{3}\RF}K\Big)\bigg]
\le  \\ \bigg[\Big[E(n,t^{\LF\frac {p''}{3}\RF}I),E(n,t^{\LF\frac {p''}{3}\RF}J)\Big],E(n,t^{\LF\frac {p''}{3}\RF}K)\bigg].
\end{multline*}
Hence, we have
\begin{multline*}
{}^{y^{-1}z}\bigg[\Big[{}^c e_{k,j'}(\frac{-\gamma}{t^{m+p'}}), e_1\Big],{}^ce_{i',k}(-t^{p'})\bigg]\le  \\
{}^{y^{-1}z}\bigg[\Big[E(n,t^{\LF\frac {p''}{3}\RF}I),E(n,t^{\LF\frac {p''}{3}\RF}J)\Big],E(n,t^{\LF\frac {p''}{3}\RF}K)\bigg]\\
=\bigg[\Big[{}^{y^{-1}z}E(n,t^{\LF\frac {p''}{3}\RF}I),{}^{y^{-1}z}E(n,t^{\LF\frac {p''}{3}\RF}J)\Big],{}^{y^{-1}z}E(n,t^{\LF\frac {p''}{3}\RF}K)\bigg].
\end{multline*}
Now applying (\ref{lem5}) in Lemma~\ref{LemHZ} to every component of the commutator above, we may find a sufficiently large $p''$ such that for any given $l$,
\begin{multline*}
\bigg[\Big[{}^{y^{-1}z}E(n,t^{\LF\frac {p''}{3}\RF}I),{}^{y^{-1}z}E(n,t^{\LF\frac {p''}{3}\RF}J)\Big],{}^{y^{-1}z}E(n,t^{\LF\frac {p''}{3}\RF}K)\bigg]\le  \\  \Big[\big[E(n,t^{l}A, t^{l} I),E(n,t^{l}A, t^{l}J)\big], E(n,t^{l}A, t^{l}K)\Big].
\end{multline*}
Choose $p_2$ in (\ref{jhafyw}) according to this $p''$ and then consider $p$ to be the larger of  $p_1$ and $p_2$. This finishes the Lemma.
\end{proof}

The proof of this result, as also the proofs of similar results
for other groups, are mostly prestidigita\-tion and tightrope
walking, and similar in spirit to the relative commutator calculus
in~\cite{46}. However, this piece of commutator calculus operates at a
different level of technical sophistication. For instance, now
we have to plug in not just the elementary generators, or their
conjugates, as in~\cite{46,42,43}, but also the other two types
of generators constructed in Theorem 3.

\begin{proof}[Proof of Theorem 7A]
The functors $E_n$ and $\GL_n$ commute with direct limits. By Proposition~\ref{Prop:01} and~\S\ref{fdye7}, one reduces the proof  to the case where $A$ is finite over $R$ and $R$ is Noetherian. 

First by the generalized commutator formula (Theorem~\ref{gcformula1}), we have
\begin{equation}\label{fin:1}
\big[E(n,A,I),\GL(n,A,J)\big]=\big[E(n,A,I),E(n,A,J)\big].
\end{equation}
Thus it suffices  to prove the following equation
$$
\Big[\big[E(n,A,I),E(n,A,J)\big],\GL(n,A,K)\Big]=
\Big[\big[E(n,A,I),E(n,A,J)\big],E(n,A,K)\Big].
$$
By Lemma~\ref{Comgenerator}, $\big[E(n,A,I),E(n,A,J)\big]$ is generated by the conjugates in $E(n,A)$ of the following four  types of elements
\begin{eqnarray}\label{ggsswwi}
e&=&\Big[e_{j,i}(\alpha),{}^{e_{i,j}(a)}e_{j,i}(\beta)\Big], \notag\\
e&=&\big[e_{j,i}(\alpha), e_{i,j}(\beta)\big],\notag\\
e&=&e_{i,j}(\alpha\beta),\notag\\ 
e&=&e_{i,j}(\beta\alpha),
\end{eqnarray}
where $i\ne j$, $\alpha\in I$, $\beta\in J$ and $a\in A$. We claim that for any $g\in \GL(n,A,K)$, 
\begin{equation}\label{jhgfsap}
\big[e,g] \in \Big[\big[E(n,A,I),E(n,A,J)\big],E(n,A,K)\Big].
\end{equation}

Let  $g\in \GL(n,A, K)$. For any maximal ideal $\gm$ of $R$, the ring $A_\gm$ contains $K_\gm$ as an ideal ($K$ being an ideal of $A$).   
Consider the natural homomorphism $\theta_\gm:A\rightarrow A_\gm$ which induces a homomorphism (call it $\theta_\gm$ again) on the level of general linear groups, $\theta_\gm:\GL(n,A)\rightarrow \GL(n,A_\gm)$. 
 Therefore, $\theta_\gm(g)\in \GL(n,A_\gm,K_\gm)$. 
 Since $A_\gm$ is module finite over the local ring $R_\gm$, $A_\gm$ is semilocal~\cite[III(2.5), (2.11)]{Bass1}, therefore  its stable rank is $1$. It follows that  $\GL(n,A_\gm,K_\gm)=E(n,A_\gm,K_\gm)\GL(1,A_\gm,K_\gm)$ (see \cite[Th.~4.2.5]{32}).  So $\theta_\gm(g)$ can be decomposed as $\theta_\gm(g)=\ep h$, where $\ep \in E(n,A_\gm, K_\gm)$ and $h$ is a diagonal matrix all of whose diagonal coefficients are $1$, except possibly the $k$-th diagonal coefficient, and $k$ can be chosen arbitrarily.
By (\S\ref{fdye7}), there is a $t_\gm\in R\backslash \gm$ such that 
\begin{equation}\label{ppooii}
\theta_{t_\gm}(g)=\ep h,
\end{equation}
where $\ep\in E(n,A_{t_\gm}, K_{t_\gm})$, and $h$ is a diagonal matrix with only one non-trivial diagonal entry which lies in $A_{t_\gm}$.

For any maximal ideal $\gm\lhd R$, choose  $t_{\gm} \in R\backslash \gm$ as above and an arbitrary positive integer $p_{\gm}$. (We will later choose $p_{\gm}$ according to Lemma~\ref{Lem:08}.)  
Since the collection of all $\{t_\gm^{p_\gm} \mid \gm \in \max(R) \}$ is not contained in any maximal ideal, we may find a finite number of $t_{\gm_s}^{p_s}\in R\backslash \gm_s$ and $x_s\in R$, $s=1,\dots,k$, such that
$$
\sum_{s=1}^k t_{\gm_s}^{p_s}x_s=1.
$$

In order to prove~(\ref{jhgfsap}), first we consider the generators of the first kind in~(\ref{ggsswwi}), namely  $e=\Big[e_{j,i}(\alpha),{}^{e_{i,j}(a)}e_{j,i}(\beta)\Big]$. Consider
\begin{multline*}
e=\Big[e_{j,i}(\alpha),{}^{e_{i,j}(a)}e_{j,i}(\beta)\Big]=\Big[e_{j,i}\Big((\sum_{s=1}^kt_{\gm_s}^{p_s}x_s)\alpha\Big),{}^{e_{i,j}(a)}e_{j,i}(\beta)\Big]=\\ \Big[\prod_{s=1}^ke_{j,i}(t_{\gm_s}^{p_s}x_s\alpha),{}^{e_{i,j}(a)}e_{j,i}(\beta)\Big].
\end{multline*}
By (C$2^+$) identity, $e=\displaystyle \Big[\prod_{s=1}^ke_{j,i}(t_{\gm_s}^{p_s}x_s\alpha),{}^{e_{i,j}(a)}e_{j,i}(\beta)\Big]$ can be written as a product of the following form:
\begin{multline}\label{hfsis}
e=\Big({}^{e_k}\Big[e_{j,i}(t_{\gm_k}^{p_k}x_k\alpha),{}^{e_{i,j}(a)}e_{j,i}(\beta)\Big]\Big)\Big({}^{e_{k-1}}\Big[e_{j,i}(t_{\gm_{k-1}}^{p_{k-1}}x_{k-1}\alpha),{}^{e_{i,j}(a)}e_{j,i}(\beta)\Big]\Big)\times\\ \cdots \times  \Big({}^{e_1}\Big[e_{j,i}(t_{\gm_1}^{p_1}x_1\alpha),{}^{e_{i,j}(a)}e_{j,i}(\beta)\Big]\Big),
\end{multline}
where $e_1,e_2,\ldots, e_k\in E(n,A)$. Note that from (C$2^+$) it is clear that all $e_s$, $s=1,\dots,k$, are products of elementary matrices of the form $e_{j,i}(A)$. Thus $e_s=e_{j,i}(a_s)$, where $a_s \in A$ and $s=1,\dots,k$, which clearly commutes with $e_{j,i}(x)$ for any $x\in A$. So the commutator~(\ref{hfsis}) is equal to 
\begin{multline}\label{hfsis1}
e=\Big(\Big[e_{j,i}(t_{\gm_k}^{p_k}x_k\alpha),{}^{e_k}{}^{e_{i,j}(a)}e_{j,i}(\beta)\Big]\Big)\Big(\Big[e_{j,i}(t_{\gm_{k-1}}^{p_{k-1}}x_{k-1}\alpha),{}^{e_{k-1}}{}^{e_{i,j}(a)}e_{j,i}(\beta)\Big]\Big)\times\\ \cdots\times \Big(\Big[e_{j,i}(t_{\gm_1}^{p_1}x_1\alpha),{}^{e_1}{}^{e_{i,j}(a)}e_{j,i}(\beta)\Big]\Big).
\end{multline}


Using (C$2^+)$ and in view of~(\ref{hfsis1}) we obtain that $[e,g]$ is a  product of the conjugates in $E(n,A)$ of 
\[w_s=\bigg[\Big[e_{j,i}(t_{\gm_s}^{p_{s}}x_{s}\alpha),{}^{e_{j,i}(a_s)e_{i,j}(a)}e_{j,i}(\beta)\Big],g\bigg],\] where $a_s \in A$ and $s=1,\ldots,k$.

 For each $s=1,\dots,k$, consider $\theta_{t_{\gm_s}}(w_s)$ which we still write as $w_s$ but keep in mind that this image 
 is in $\GL(n,A_{t_{\gm_s}})$. 

 Note that all $\Big[e_{j,i}(t_{\gm_s}^{p_{s}}x_{s}\alpha),{}^{e_{j,i}(a_s)e_{i,j}(a)}e_{j,i}(\beta)\Big]$, $s=1,\ldots, k$,  differ from the identity matrix at only the $i,j$ rows and the $i,j$ columns. Since $n>2$, we can choose $h$ in the decomposition~(\ref{ppooii}) so that it commutes with 
\[\Big[e_{j,i}(t_{\gm_s}^{p_{s}}x_{s}\alpha),{}^{e_{j,i}(a_s)e_{i,j}(a)}e_{j,i}(\beta)\Big].\] 
This allows us to reduce $\theta_{t_{\gm_s}}(w_s)$  to 
$$
\bigg[\Big[e_{j,i}(t_{\gm_s}^{p_{s}}x_s\alpha),{}^{e_{j,i}(a_s)e_{i,j}(a)}e_{j,i}(\beta)\Big], \ep\bigg],
$$
where $\ep \in E(n,A_{t_{\gm_s}}, K_{t_{\gm_s}})$.
By Lemma~\ref{Lem:08},  for any given $l_{s}$, there is a sufficiently large $p_{s}$, $s=1,\dots k$, such that
\begin{multline*}
\bigg[\Big[e_{j,i}(t_{\gm_s}^{p_{s}}x_s\alpha),{}^{e_{j,i}(a_s)e_{i,j}(a)}e_{j,i}(\beta)\Big], \ep\bigg]
\in \\ \Big[\big[E(n,A,t^{l_{s}}I), E(n,A,t^{l_{s}}J)\big], E(n,A,t^{l_{s}}K)\Big].
\end{multline*}
Let us choose $l_{s}$ to be large enough so that by Lemma~\ref{Lem:03} the restriction of \[\theta_{t_{\gm_s}}:  \GL(n,A,t_{\gm_s}^{l_{s}}A)\to \GL(n,A_{t_{\gm_s}})\] be injective. Then it is easy to see that for any $s$, we have
$$\bigg[\Big[e_{j,i}(t_{\gm_s}^{p_{s}}x_s\alpha),{}^{e_{j,i}(a_s)e_{i,j}(a)}e_{j,i}(\beta)\Big], g\bigg]\in 
\Big[\big[E(n,A,I),E(n,A,J)\big],E(n,A,K)\Big].$$
Since relative elementary subgroups $E_n$ are normal in $\GL(n,A)$ (Theorem~\ref{standard}), it follows that $[e,g]\in \Big[\big[E(n,A,I),E(n,A,J)\big],E(n,A,K)\Big].$

When the generator is of the second kind, $e=[e_{i,j}(\alpha),e_{j,i}(\beta)]$, a similar argument goes through, which is left to the reader.

Now consider the generators of the 3rd and 4th kind, namely, the conjugates of the following two types of elements,
$$e=e_{i,j}(\alpha\beta),\quad \text{or    } e=e_{i,j}(\beta\alpha).$$
By the normality of $E(n,A,IJ+JI)$,  the conjugates of $e$ are in $E(n,A, IJ+JI)$. Then
$$
[e,g]\in \big[E(n,A, IJ+JI), \GL(n,A,K)\big].
$$
By the generalized commutator formula (Theorem~\ref{gcformula1}), one obtains 
$$
\big [E(n,A, IJ+JI), \GL(n,A,K)\big]=\big[E(n,A, IJ+JI), E(n,A,K)\big].
$$
Now applying Lemma~\ref{Lem:Habdank}, we finally get 
$$
[E(n,A, IJ+JI), E(n,A,K)]\le  \Big[\big[E(n,A,I),E(n,A,J)\big],E(n,A,K)\Big].
$$
Therefore  $[e,g]\in \Big[\big[E(n,A,I),E(n,A,J)\big],E(n,A,K)\Big].$
This proves our claim. Thus we established~(\ref{jhgfsap}) for all type of generators $e$ of~(\ref{ggsswwi}).

To finish the proof, let $e \in \big[E(n,A,I),\GL(n,A,J)\big]=\big[E(n,A,I),E(n,A,J)\big]$, and $g\in \GL(n,A,K)$. Then by Theorem~\ref{gen-2},
$$e=e_{1}\times e_{2}\times\cdots\times e_{k}$$
with $e_{i}$ takes any of the forms in~(\ref{yyttrree2}). Thanks to (C$2^+$) identity and the normality of relative elementary subgroups $E_n$, it suffices to show that
$$[e_i,g]\in\Big[\big[E(n,A,I),E(n,A,J)\big],E(n,A,K)\Big],\quad i=1,\dots,k.$$
But this is exactly what has been shown above. This completes the proof. 
\end{proof}

Similar result for unitary groups is \cite[Theorem 7]{44}.
\begin{TheB}\label{yth6}
Let $n\ge 3$, $R$ be a commutative ring, $(\Form)$ be a form ring such
that $A$ is a quasi-finite $R$-algebra. Further, let $(I,\Gamma)$,
$(J,\Delta)$ and $(K,\Omega)$ be three form ideals of a form ring
$(\Form)$. Then
\begin{multline*}
\Big[\big[\EU(2n,I,\Gamma),\GU(2n,J,\Delta)\big],\GU(2n,K,\Omega)\Big]=\\
\Big[\big[\EU(2n,I,\Gamma),\EU(2n,J,\Delta)\big],\EU(2n,K,\Omega)\Big].
\end{multline*}
\end{TheB}

The proof of Theorem~\ref{yth6} is even more toilsome, than that of Theorem~\ref{the43}. In fact,
just the proof of the unitary analogue of the above triple commutator
lemma,~\cite[Lemma~13]{44}, consists of some six solid pages of calculations.
\par
After Theorem~\ref{yth6} is established, Theorem~\ref{nemorh3}, \ref{nemorh4} follows by 2--3 pages of
artless formal juggling with level calculations and commutator
identities, the details of calculations
can be found in~\cite{46,44}. For Chevalley groups they are still
unpublished.




\forget

\section{Commutators of congruence subgroups}

The multiple commutator formula we stated in \S~8 depends only on
some commu\-tativity conditions. The {\it general\/} multiple
commutator formula formula we are going to discuss now, only works
over {\it finite-dimensional\/} rings.
\par
Let, as before, $I_i$, $i=1,\ldots,m$, be ideals of the ring $R$.
Once we have Theorem 6, it is natural to ask, and this was stated
as [114], Problem 2, whether the presence of an elementary factor
is essential there? In other words, does the following stronger
commutator formula
$$ [\![G(R,I_1),G(R,I_2),\ldots,G(R,I_m)]\!]=
[\![E(R,I_1),E(R,I_2),\ldots,E(R,I_m)]\!]. $$
\noindent
hold, at least under some assumptions on $m$ and $R$?
\par
Some special cases of this formula were indeed known before.
Let us look at the two first instances, 0-dimensional rings, and
1-dimensional rings. In fact, these results are stated in terms
of the stable rank $\sr(A)$ of the ring $A$, see~\cite{14}.
\par\smallskip
$\bullet$ When $\sr(A)=1$, one has $\SL(n,A,I)=E(n,A,I)$.
For the commutative case, this follows from a classical result
by Hyman Bass~\cite{14}. In general, this is essentially
Bak's {\it definition\/} of $\SL(n,A,I)$, we discuss below.
However, Bak proves that in the commutative case his definition
does agree with the usual one in terms of determinants, see~\cite[Lemma 3.7]{8}.
\par\smallskip
$\bullet$ When $\sr(A)=2$ a classical result by Alec Mason and
Wilson Stothers asserts that for any two two-sided ideals
$I$ and $J$ of $A$ one has
$$ [\GL(n,A,I),\GL(n,A,J)]=[E(n,A,I),E(n,A,J)]. $$
\par
In fact, this last result is the first non-trivial case of the
following theorem.
\begin{The}
Let\/ $A$ be any associative ring with 1, let\/ $n\ge\sr(A)+1,3$,
and let\/ $I,J\unlhd A$ be two-sided ideals of\/ $A$. Then one
has
$$ [\GL(n,A,I),\GL(n,A,J)]=[E(n,A,I),E(n,A,J)]. $$
\end{The}
For commutative rings, under somewhat stronger assumptions, this
is [72], Corol\-lary 3.9. For non-commutative rings this
theorem is stated in [70], Theo\-rem 1.2, with an indication
that the proof follows the same lines. In [38] we give a short
proof of this result, based on Theorem 4.
\par
The main result we wish to report in this talk is a generalisation
of this theorem to {\it multiple\/} commutators over an
{\it arbitrary\/} finite-dimensional ring. We cannot do this in
terms of the stable rank $\sr(A)$ itself, since that would be far
too much. There is no obvious way to implement induction on
$\sr(A)$.
\par
Following Anthony Bak [8] we use Bass--Serre dimension
$\delta(R)$ of $R$ instead. One of the [many!] deep external
results on which our proof depends, is the following {\it induction
lemma\/} by Bak.
\begin{Lem}
Let\/ $R$ be a commutative ring of finite Bass--Serre dimension\/
$\delta(R)$. Let\/ $X_1\cup\ldots\cup X_r$ be a decomposition of\/
$\Max(R)$ into irreducible Noetherian subspaces of dimension\/
$\le\delta(R)$. If\/ $s\in R$ is an element such that for each\/
$I_k$, $1\le k\le r$, there exists an ideal\/ ${\frak m}_k\in X_k$
such that\/ $s\notin {\frak m}_k$. Then\/
$\delta(\widetilde R_{(s)})<\delta(R)$.
\end{Lem}

To state the next result, we have to recall the definition of
super special linear groups, introduced by Bak [8]. These groups
are another major tool to implement induction in our proofs.
Let $A$ be a module finite algebra over a commutative ring
$R$ and let $I$ be an ideal in $A$. Define
$$ \SSL{m}(n,A,I)=\bigcap_{\phi}\Ker\Big(
\GL(n,A,I)\map\GL(n,A',I')/E(n,A',I')\Big), $$
\noindent
where the intersection is taken over all homomorphisms
$A\map A'$ of rings such that $A'$ is module finite over a
commutative ring $R'$ of Bass--Serre dimension $\delta(R')\le m$,
$I'$ is the ideal of $A'$ generated by $\phi(I)$.
\par
This definition can be extended to all quasi-finite rings
by passage to limits. Namely, let $A=\varinjlim A_i$,
where $A_i$ is module finite over a commutative ring $R_i$.
Consider the ideal $I_i=I\cap A_i$. Then
$$ \SSL{m}(n,A,I)=\varinjlim \SSL{m}(n,A_i,I_i). $$
\noindent
By definition $\SSL{m}(n,A,I)$ is functorial. The group
$\SSL{0}(n,A,I)$ will be denoted simply by $\SL(n,A,I)$.
When $A=R$ is itself commutative, it coincides with the
usual special congruence subgroup.
\par
The main result of Bak [8] is the existence of nilpotent filtration on
relative $K_1$, which becomes finite for finite dimensional rings.

\begin{The} 
Let\/ $A$ be a quasi-finite ring over a commutative
ring\/ $R$, let\/ $n\ge 3$, and let\/ $I$ be an ideal of\/ $A$. Then
\begin{itemize}
\item[$\bullet$] Each\/ $\SSL{m}(n,A,I)$ is a normal subgroup of\/ $\GL(n,A)$.
\item[$\bullet$] The sequence
$$ \SSL{0}(n,A,I)\ge\SSL{1}(n,A,I)\ge\SSL{2}(n,A,I)\ge\cdots $$
\noindent
is a descending\/ $\SSL{0}(n,A,I)$-central series.
\item[$\bullet$] The conjugation action of\/ $\GL(n,A)$ on\/
$\GL(n,A,I)/\SSL{0}(n,A,I)$ is trivial.
\item[$\bullet$] If Bass--Serre dimension of\/ $R$ is finite,\/
$\delta(R)<\infty$, then
$$ \SSL{m}(n,A,I)=E(\Phi,R,I), $$
whenever\/ $m\ge\delta(R)$.
\end{itemize}
\end{The}

Actually, in the first part of this paper we have already discussed
generalisations of such nilpotent filtrations to relative $K_1$,
for unitary groups, and for Chevalley groups, obtained in [33], 
[40], [10].
\par
This theorem implies the following result, where we cite a
slightly better bound for $m$, which follows from Bass' result
(a special case of the Mason--Stothers theorem above).

\begin{The}
Let\/ $A$ be a quasi-finite algebra with\/ $1$ over a commutative
ring\/ $R$ of finite Bass--Serre dimension\/ $\delta(R)$,
let\/ $n\ge 3$, and further let\/ $I$ be a two-sided ideal
of\/ $A$. Assume that\/ $m\ge\max(\delta(R)+3-n,1)$. Then
$$ [\SL(n,A,I),\SL(n,A),\ldots,\SL(n,A)]=E(n,A,I), $$
\noindent
where the number of\/ $\SL(n,A)$ equals\/ $m$.
\end{The}


\section{General multiple commutator formula}

Here we state the {\it general\/} multiple relative commutator
formula for $\GL_n$, which is the main result of our paper [38].
It is a very powerful result, which, when valid, simultaneously
generalises all previously known commutator formulas and
nilpotent filtrations of relative $\K_1$.
\par
However, it requires some finiteness conditions, and relies
on a host of deep external results. At this time, we only
have a complete proof for the case of the general linear group.
In this section we limit ourselves with the case $G=\GL_n$.
It is not that we lack the machinery to generalise these
results to groups, as we shall see in the next sections, all
requisite localisation techniques itself is there. What is not
there, especially for Chevalley groups, are definitive analogues
of such classical results as Whitehead lemma, stability of
$\K_1$, etc.

\begin{TheA}
Let\/ $A$ be a quasi-finite algebra with\/ $1$ over a commutative
ring\/ $R$ of finite Bass--Serre dimension\/ $\delta(R)$,
let\/ $n\ge 3$, and further let\/ $I_i\unlhd R$, $i=1,\ldots,m$,
be two-sided ideals of\/ $A$. Assume that\/ $m\ge\max(\delta(R)+3-n,1)$.
Then
\begin{multline*}
[\![\SL(n,A,I_1),\SL(n,A,I_2),\ldots,\SL(n,A,I_m)]\!]=\\
[\![E(n,A,I_1),E(n,A,I_2),\ldots,E(n,A,I_m]\!].
\end{multline*}
\end{TheA}

The strategy of the proof is induction on $\delta(R)$, using Bak's
induction lemma. Eventually, by induction we show that
\begin{multline*}
[\![\SL(n,A,I_1),\SL(n,A,I_2),\ldots,\SL(n,A,I_m)]\!]=\\
[E(n,R,I_1\circ\ldots\circ I_h),E(n,R,I_{h+1}\circ\ldots\circ I_m)],
\end{multline*}
\noindent
the right hand side of this equality being, as we know from Theorem 8A,
just another expression for
$[\![E(n,A,I_1),E(n,A,I_2),\ldots,E(n,A,I_m]\!]$.
\par
The following result serves as the base of induction. It is
essentially a combination of known results.
\begin{TheA}
Let\/ $A$ be a quasi-finite algebra with\/ $1$ over a commutative
ring\/ $R$ of finite Bass--Serre dimension\/ $\delta(R)$,
let\/ $n\ge 3$, and let\/ $I,J\unlhd A$ be two-sided ideals
of\/ $A$. Then one has
$$ [\GL(n,R,A),\GL(n,R,B)]=\SSL{\delta(R)-n+3}(n,R,A\circ B). $$
\end{TheA}

In fact, this is an immediate corollary of the above Mason--Stothers
theorem, Bass' estimate of the stable rank of $A$ in terms of Bass--Serre
dimension $\delta(R)$, and Bass--Vaserstein theo\-rem on injective
stability of relative $\K_1$.

On the other hand, induction step now looks as follows.

\begin{TheA}
Let\/ $A$ be a quasi-finite algebra with\/ $1$ over a commutative
ring\/ $R$, $n\ge 3$, and let\/ $I,J,K\unlhd A$ be three two-sided
ideals of\/ $A$. Then for any\/ $m$ one has
\begin{multline*}
 [[\SSL{m}(n,A,I),\SSL{m}(n,A,J)],\SSL{0}(n,A,K)]\le\\
[\SSL{m+1}(n,A,I\circ J),\SSL{m+1}(n,A,K)]. \end{multline*}
\end{TheA}

{\it Ideologically\/}, the proof of this result -- as the proofs
in [33], [40] and [10] -- is still modeled on Bak's paper [8]. However,
in most important technical aspects the proof is {\it completely\/}
new, as we explain in the next sections. The most important
innovation can be described as follows. In Bak's paper [8], as also
in [33], [40], the whole interplay between localisation and completion
was implemented at the {\it global\/} level. After that relative
results were derived from the corresponding absolute results
by a version of relativisation, which affords some form of splitting.
\par
This is how the proof of relative results was carried through in
Bak's paper [8], in the case of $\GL_n$. Our joint paper with Bak [10],
where similar relative results were obtained for unitary groups, and
for Chevalley groups, followed the same general strategy. Of course,
considerable additional technical strain was due to the fact that
relativisation with respect to {\it form\/} ideals was harder
to implement.
\par
However, in the above papers Bak and ourselves always considered
{\it one\/} ideal. Now, we have to relativise with respect to
{\it several\/} ideals -- well, at least with respect to two ideals
-- and use some form of {\it relative\/} splitting principle, rather
than just the absolute one. There is no need to persuade anyone,
who has herself ever tried to fool around with relativisation with
{\it two\/} parameters, that this is a gruesome task. An attempt to
prove Theorem 10 along these lines, by splitting several ideals
simultaneously, immediately lead to rather awkward technical
impediments.
\par
Our idea was then to prove {\it multirelative\/} versions of Bak's
localisation completion theorem itself. Recall that localisation
completion theorem is another {\sc local glo\-bal principle}.
Essentially, it asserts that the commutator of something that
becomes elemen\-tary under principal $s$-localisation with something
else that becomes elementary under $s$-adic completion is indeed
{\it globally\/} elementary. The proof is an exemplary manifestation
of the combined force of continuity and density, in $s$-adic topology.
\par
Morally, the advantage of our new approach is that it consists in
moving all relativisation to the {\it local\/} level, where
congruence subgroups coincide with relative elementary subgroups, so
that splitting is not an issue at all. This would be possible if
we could rely on the the full force of {\it relative\/} commutator
calculus.
\par
In fact, as reported in [36] and in \S~9, we worked out {\it some\/}
form of relative commutator calculus in [46], [42] and [43], and then a
slightly fancier one in [47] and [44]. Regretfully, in all these papers
except [43] we implemented only {\it first localisation\/}, whereas
now we need {\it much\/} stronger versions of Theorem 1, with
denominators. To get that, we have to implement what is called
{\it second localisation\/}.
\par
This means that we had to turn the crank again, to redo all relative
commutator calculus {\it from scratch\/}, allowing {\it two\/}
denominators. The target results of this version of commutator
calculus look like a blend of Theorem 1 or Theorem 7 with the
commutator calculus lemmas used in their proofs, and will be
reproduced in the next section.


\section{Relative commutator calculus, revisited}

To state the target results of the relative commutator calculus,
we need to introduce some notation, similar to the notation
we used in the first part of this work, [36] \S~5. Namely, for
two additive subgroups $B$, $C$ of a ring $R$ we denote by
$E^L(C,B)$ the {\it set\/} of all products of $\le L$ relative
elementary generators $z(\xi,\eta)$ of the group $E(R,RBR)$,
such that $\xi\in B$, $\eta\in C$.
\par
Thus, for instance, $E^L\big(\Phi,\frac{1}{s^k}R,\frac{1}{s^k}I\big)$
is the set of products of $\le L$ elements of the form
$$ x_{\a}\Big(\frac{\xi}{s^k},\frac{\eta}{s^k}\Big)=
{}^{x_{-\a}\big(\frac{\eta}{s^k}\big)}
x_{\a}\Big(\frac{\xi}{s^k}\Big),
\qquad \xi\in I,\ \eta\in R. $$
\noindent
In other words,
$$ E^L\Big(\Phi,\frac{1}{s^k}R,\frac{1}{s^k}I\Big)=
{\left({}^{E^1\left(\Phi,\frac{t^l}{s^k}R\right)}
E^1\Big(\Phi,\frac{t^l}{s^k}I\Big)\right)}^L. $$
\noindent
Clearly, for any $x\in E(\Phi,R_s,I_s)$ there exist positive
integers $k$ and $L$ such that
$x\in E^L\big(\Phi,\frac{1}{s^k}R,\frac{1}{s^k}I\big)$.
\par
The following result is [43], Theorem 2. We only established
this result under somewhat stronger assumption in rank 2,
than those in theorem 1C, to spare some 2--3 further pages of
calculations. Anyway, for groups of types $\C_2$ and $\G_2$
relativisation has to be considered separately, in the more
general setting of admissible pairs, or even radices.
\setcounter{TheC}{11}
\begin{TheC}Let\/ $\Phi$ be a reduced irreducible root
system,\/ $\rk(\Phi)\ge 2$. In the cases\/ $\Phi=\C_2,\G_2$ assume
additionally that\/ $2\in R^*$. Then for any\/ $s\in R$, $s\neq 0$,
any\/ $p,k$ and\/ $L$, there exists an\/ $r$ such that for any two
ideals\/ $I$ and\/ $J$ of a commutative ring\/ $R$, one has
\begin{multline*}
\bigg[E^L\Big(\Phi,\frac{1}{s^k}R,\frac{1}{s^k}I\Big),
F_s\big(G(\Phi,R,s^rJ)\big)\bigg]\le\\
\Big[E\big(\Phi,F_s(s^pR),F_s(s^pI)\big),
E\big(\Phi,F_s(s^pR),F_s(s^pJ)\big)\Big].
\end{multline*}
\end{TheC}
{\it Regrettably\/}, in [46] and [42] we only implemented the 
{\it first\/} localisation and patching, whereas theorems of that 
type require also the {\it second\/} localisation. To prove analogues 
of the above theorem for $\GL_n$ and for unitary groups, we were 
forced to replay the whole relative commutator calculus from the 
very start once again, allowing {\it two\/} denomina\-tors. The 
following results are established in [38] and in [39], respectively.

\begin{TheA}
Let\/ $n\ge 3$,\/ $R$ be a commutative ring, and let\/ $A$ be a
quasi-finite\/ $R$-algebra. Then for any\/ $s\in R$, $s\neq 0$,
any\/ $p,k$ and\/ $L$, there exists an\/ $r$ such that for any two
ideals\/ $I$ and\/ $J$ of\/ $A$, one has
\begin{multline*}
\bigg[E^L\Big(n,\frac{1}{s^k}A,\frac{1}{s^k}I\Big),
F_s\big(\GL(n,A,s^rJ)\big)\bigg]\le\\
\Big[E\big(n,F_s(s^pA),F_s(s^pI)\big),
E\big(n,F_s(s^pA),F_s(s^pJ)\big)\Big].
\end{multline*}
\end{TheA}

\setcounter{TheB}{11}
\begin{TheB}
Let\/ $n\ge 3$,\/ $R$ be a commutative ring,\/ $(A,\Lambda)$ be a form
ring such that\/ $A$ is a quasi-finite\/ $R$-algebra. Then for
any\/ $s\in R_0$, $s\neq 0$, any\/ $p,k$ and\/ $L$, there exists
an\/ $r$ such that for any two form ideals\/ $(I,\Gamma)$
and\/ $(J,\Delta)$ of\/ $(A,\Lambda)$, one has
\begin{multline*}
\bigg[\EU^L\Big(n,\frac{1}{s^k}I,\frac{1}{s^k}\Gamma\Big),
F_s\big(\GU(n,J,s^r\Delta)\big)\bigg]\le\\
\Big[\EU\big(n,F_s(s^pI),F_s(s^p\Gamma)\big),
\EU\big(n,F_s(s^pJ),F_s(s^p\Delta)\big)\Big].
\end{multline*}
\end{TheB}


\section{Relative localisation completion}

In the first part of the present paper we have already
discussed Bak's localisation completion theorem [8, and its
analogues for unitary groups, due to the first author [33], [34],
and for Chevalley groups, due to the first and the third
authors [40], see [36], Theorem 8.

\par
We start with explaining this idea in the simplest case of {\it one\/}
ideal. This is not yet sufficient to prove Theorem 10, but it allows
to give easier proofs of nilpotency of $\K_1$.
\par
With this end, recall the necessary notation concerning completion,
see [8] or [36], \S~11. Let $s\in A$. Usually, the $s$-completion
$\widehat A_{s}$ of the ring $A$ is defined as the following inverse
limit:
$$ \widehat R_{s}=\varprojlim R/s^nR,\quad n\in{\Bbb N}. $$
\noindent
However, for our purposes one has to modify this definition,
by interchanging the order of taking limits. Namely, we set
$$ \widetilde A_{s}=\varinjlim (\widehat A_i)_{s}, $$
\noindent
where the limit is taken over all finitely generated subrings
$A_i$ of $A$ which contain $s$. Let us denote by $\widetilde
F_s$ the canonical map $A\map\widetilde A_{s}$. For the
case, where $R$ is Noetherian, $\widetilde F_{s}=\widehat F_{s}$
coincides with the inverse limit of reduction homomorphisms
$\pi_{s^n}:A\map A/s^nA$
\par
In [36], we discussed definition of the groups $G(R,s^{-1})$
and $G(R,\widehat s)$. First of all, we have to introduce
their relative analogues.
\par
First, let $A$ be an algebra over a commutative ring $R$,
$n\ge 3$, $s\in R$, and further let $I$ be an ideal of $A$.
In the case of $\GL_n$ the relative analogues of the above groups
are defined as follows
$$ \aligned
&\GL(n,A,I,s^{-1})=\Ker\big(\GL(n,A,I)\map\GL(n,A_s,I_s)/E(n,A_s,I_s)\big),\\
\noalign{\vskip 5truept}
&\GL(n,A,I,\widehat s)=\Ker\big(\GL(n,R,I)\map
\GL(n,\widetilde A_{(s)},\widetilde I_{(s)})/
E(n,\widetilde A_{(s)},\widetilde I_{(s)})\big).\\
\endaligned $$
\noindent
Now, let $A_i$ be the inductive system of all finite $R_i$-subalgebras,
where $R_i$ ranges over all finitely generated subrings of $R$
containing $s$. We set $I_i=I\cap A_i$. Then
$$ \GL(n,A,I,s^{-1})=\varinjlim\GL(n,A_i,I_i,s^{-1}),\qquad
\GL(n,A,I,\widehat s)=\varinjlim\GL(n,A_i,I_i,\widehat s), $$
\noindent
which reduces most problems about these groups to the case, where $A$
is finite over a Noetherian ring $R$. The same argument works for
other groups and below we usually assume that the ground ring is
Noetherian.
\par
The following result is [38], Theorem 9.
\begin{TheA}
Let\/ $A$ be a quasi-finite algebra over a commutative ring\/ $R$,
$n\ge 3$, let\/ $I$ be an ideal of\/ $R$, and let\/ $s\in R$. Then
$$ [\GL(n,A,I,s^{-1}),\GL(n,A,\widehat s)]\le E(n,A,I). $$
\end{TheA}

Similar result holds also in the unitary setting. Let
$(A,\Lambda)$ be a form ring, which is module finite over a
commutative ring $R$. Further, let $(I,\Gamma)$ be a form ideal
of $(A,\Lambda)$. Take $s\in R_0$ and define
$$ \aligned
&U(2n,I,\Gamma,s^{-1})=\Ker\Big(U(2n,I,\Gamma)\map
U(2n,I_s,\Lambda_s)/\EU(2n,I_s,\Gamma_s)\Big),\\
\noalign{\vskip 5truept}
&U(2n,I,\Gamma,\hat s)=\Ker\Big(U(2n,I,\Gamma)\map
U\big(2n,\tilde{(I,\Gamma)}_{(s)}\big)/
E\big(2n,\tilde{(I,\Gamma)}_{(s)}\big)\Big),\\
\endaligned $$
\noindent
A similar result for unitary groups can be stated as follows.
It is a generalisation of one of the main results of the Thesis
by the first named author [33], [34].
\setcounter{TheB}{12}
\begin{TheB}
Let\/ $(A,\Lambda)$ be a module finite form ring
over a commutative ring\/ $R$, let\/ $(I,\Gamma)$ be a form ideal
of\/ $(A,\Lambda)$, and let\/ $s\in R_0$. Then
$$ [U(2n,I,\Gamma,s^{-1}),U(2n,A,\Lambda,\hat s)]\le\EU(2n,I,\Gamma). $$
\end{TheB}

Finally, let $\Phi$ be an irreducible root system of rank $\geq 2$,
let $R$ be a commutative ring, let $I$ be an ideal of $R$, and let
$s\in R$. Define
$$ \aligned
&G(\Phi,R,I,s^{-1})=\Ker\big(G(\Phi,R,I)\map
G(\Phi,R_s,I_s)/E(\Phi,R_s,I_s)\big),\\
\noalign{\vskip 5truept}
&G(\Phi,R,I,\widehat s)=\Ker\big(G(\Phi,R,I)\map
G(\Phi,\tilde R_{(s)},\tilde I_{(s)})/
E(\Phi,\tilde R_{(s)},\tilde I_{(s)})\big).\\
\endaligned $$
\par
The following result for Chevalley groups is unpublished, and
below we sketch its proof.
\setcounter{TheC}{12}
\begin{TheC}
Let\/ $A$ be a quasi-finite algebra over a commutative ring\/ $R$,
$n\ge 3$, let\/ $I$ be an ideal of\/ $R$, and let\/ $s\in R$. Then
$$ [G(\Phi,R,I,s^{-1}),G(\Phi,R,\widehat s)]\le E(\Phi,R,I). $$
\end{TheC}

The following proof is essentially an enhancement of the proof
of [40], Theorem 6.1, where we simply plug in a more powerful
version of the commutator calculus, {\it relative\/}, instead
of absolute. Of course, the proof in [40] was itself just a 
streamlined adaptation of the original Bak's argument [8].
\par
Denote by
$E^K(\Phi,{1\over s^k}R,{1\over s^k}I)$ the subset of
$E(\Phi,R_s,I_s)$, consisting of products of $\le K$ elements of
the form $z_{\a}(\xi,\zeta)$, where $\a\in\Phi$,
$\xi\le {1\over s^k}I$, $\zeta\in{1\over s^k}R$.
\par
The usual argument allows us to reduce the proof to the case,
where $R$ is Noetherian. Let $x\in G(\Phi,R,I,s^{-1})$ and
$y\in G(\Phi,R,\widehat s)$. By definition, the condition on
$x$ means that $F_s(x)\in E^K(\Phi,{1\over s^k}R,{1\over s^k}I)$
for some $k$ and some $K$. On the other hand, the condition on $y$
means that $\pi_{s^m}(y)\in E(\Phi,R/s^mR)$ for {\it all\/} $m$,
or, what is the same, that $y=uz$, where $u\in\in E(\Phi,R)$ and
$z\in\GL(n,R,s^mR)$.
\par
Thus, $[x,y]=[x,uz]=[x,u]\cdot{}^u[x,z]$. The first factor
belongs to $E(\Phi,R)$ simply because it is normal in $G(\Phi,R)$.
As for the second factor, a typical target result of the relative
commutator calculus -- in this situation [a special case of]
Theorem 2 of [] -- can be stated as follows: for any $q$ there
exists a sufficiently large $m$ such that
$F_{s}([x,z])\in E(\Phi,F_s(s^qR),F_s(s^qI))$. On the other hand,
since $G(\Phi,R,s^qR)$ is normal in $G(\Phi,R)$, one has
$[x,z]\in G(\Phi,R,s^qR)$. Now, since $R$ is assumed to be
Noetherian, the usual argument based on injectivity of localisation
homomorphisms on small neighbourhoods of $e$ in $s$-adic topology
convinces us that $[x,z]\in E(\Phi,R,s^qI)$. This shows that both
$[x,u]$ and ${}^u[x,z]$, and thus also $[x,y]$ are elementary, as
claimed.
\par
Observe that Theorems 3A--3C easily imply by induction nilpotency
of {\it relative\/} $\K_1$, in other words, main results of [8]
and [10]. From the very start, this proof works at the relative level,
without any need to relativise results on nilpotent filtrations of
absolute $\K_1$. We believe that this proof is both better
conceptually and (once the target results of relative commutator
calculus are established!) technically easier than the original proof.


\section{Birelative localisation completion}

However, to prove the general commutator formula we need stronger
results.

\begin{TheA}
Let $A$ be a quasi-finite algebra over a commutative ring $R$,
$n\ge 3$, $s\in R$, and let $I,J\unlhd A$ be two-sided ideals
of $A$. Then
$$ [\GL(n,A,I,s^{-1}),\GL(n,A,J,\widehat s)]\le
[E(n,A,I),E(n,A,J)]. $$
\end{TheA}

\begin{TheB}
Let $(A,\Lambda)$ be a module finite form ring
over a commutative ring $R$, let\/ $(I,\Gamma)$ and\/ $(J,\Delta)$
be two form ideals of\/ $(A,\Lambda)$, and let\/ $s\in R_0$. Then
$$ [U(2n,I,\Gamma,s^{-1}),U(2n,J,\Delta,\hat s)]\le
[\EU(2n,I,\Gamma),\EU(2n,J,\Delta). $$
\end{TheB}

\begin{TheC}
Let\/ $A$ be a quasi-finite algebra over a commutative ring\/ $R$,
$n\ge 3$, let\/ $I$ be an ideal of\/ $R$, and let\/ $s\in R$. Then
$$ [G(\Phi,R,I,s^{-1}),G(\Phi,R,J,\widehat s)]\le
[E(\Phi,R,I),E(\Phi,R,J)]. $$
\end{TheC}

\begin{Lem}
Let $A$ be a module finite algebra over Noetherian ring $R$,
$I$ a two-sided ideal of $A$. Then for any given $l$ and $s\in R$
there is a sufficiently large $m$ such that $I\cap s^mA\le  s^lI$.
\end{Lem}
Indeed, consider the following sequence of ideals
$$ J_m=(I:s^m)=\{a\in A\mid s^mA\in I\},\qquad m=0,1,2,\ldots $$
\noindent
Clearly, $J_0\le J_1\le J_2\le\ldots$. Since $R$ is Noetherian and
$A$ is module finite over $R$, there exists such an $k$ that
$J_k=J_{k+1}=J_{k+2}=\ldots$. The equality $J_m=J_k$ for any $m>k$,
amounts to the following: for any $a\in R$ inclusion $s^{m}a\in I$
implies that already $s^ka\in I$. In particular, taking $m=k+l$,
and an $a\in R$ such that $s^mb\in I\cap s^m A$, we can conclude
that $s^kb\in I$, and thus $s^ma=s^{k+l}a=s^l(s^ka)\in s^lI$.

\begin{TheA}\label{thewayh13}
Let $A$ be a quasi-finite algebra over a commutative ring $R$,
$n\ge 3$, $s\in R$, and let $I,J,K\unlhd A$ be two-sided ideals
of $A$. Then
$$ [[\GL(n,A,I,\widehat s),\GL(n,A,J,\widehat s)],\GL(n,A,K,s^{-1})]\le
[E(n,A,I\circ J),E(n,A,K)]. $$
\end{TheA}

\par
For unitary groups, part of the requisite relative commutator
calculus with two denominators was developed in our paper [39].
We are convinced that in this case we could prove an analogue of
Theorem 11 by the same strategy. In this case, the external
equipment, on which our prove depends, such as appropriate versions
of Whitehead lemma, stability theorems, etc., can be mostly found
in the existing literature.
\par
For Chevalley groups, the situation is different. On the one hand,
in [43] we have already developed a stronger version of relative
commutator calculus, and [43], Theorem 2, is precisely what we need
to prove the birelative localisation completion theorem. On the other
hand, many of the background results are not available in
\par
The second author is positive that he can prove an analogue
of Theorem 8A for Chevalley groups with the use of his
universal localisation method. With this approach all problems
related to splitting are shifted the affine ring of $G$ and other
universal rings.


\section{Relative commutator length}

In the first part of this paper we have already discussed
results on the commutator width in the absolute case, what
we called the ``anti-Ore'' behavious of groups of points
of algebraic groups and their elementary subgroups over
rings of dimension $\ge 2$. Namely, in [80] and [88] we proved
that commutators have bounded length in elementary generators
of $\GL_n$ and Chevalley groups respectively, in the absolute
case.
\par
Here, we state the relative versions of these results, recently
obtained by the second author. We limit ourselves with the
results themselves and a minimum of comments. One should
consult our Porto Cesareo conference paper [37] for a much
broader picture, including background, history, motivation,
ideas of proof, analogues, and unsolved problems.
\par
First, we replace one occurence of $R$ by an ideal. Of course,
in this case the elementary generators, but rather their conjugates
$z_{ij}(\xi,\eta)$, $Z_{ij}(\xi,\eta)$, and $z_{\a}(\xi,\eta)$, we
introduced in \S~5. These results are new even in the case of the
general linear group.

\begin{TheA}Let\/ $R$ be a commutative ring and let\/ $I\unlhd R$
be an ideal of\/ $R$. Then there exists an\/ $L$ such that any
commutator\/ $[x,y]$, where
$$ x\in\SL(n,R,I),\quad y\in E(n,R)\qquad\text{or}\qquad
x\in\SL(n,R),\quad y\in E(n,R,I) $$
\noindent
is a product of not more than\/ $L$ elementary generators\/
$z_{ij}(\xi,\zeta)$, where\/ $1\le i\neq j\le n$, $\xi\in I$,
$\zeta\in R$.
\end{TheA}

\setcounter{TheC}{15}

\begin{TheC}
Let\/ $R$ be a commutative ring and let\/ $I\unlhd R$,
be an ideal of\/ $R$. Then there exists an\/ $L$ such that any
commutator\/ $[x,y]$, where
$$ x\in G(\Phi,R,I),\quad y\in E(\Phi,R)\qquad\text{or}\qquad
x\in G(\Phi,R),\quad y\in E(\Phi,R,I) $$
\noindent
is a product of not more than\/ $L$ elementary generators\/
$z_{\alpha}(\xi,\zeta)$, where\/ $\alpha\in\Phi$, $\xi\in I$,
$\zeta\in R$.
\end{TheC}

Now we state the ultimate {\it birelative\/} versions of the
results on commutator length. These results use the elementary
generating systems of relative commutator subgroups
$[E(R,I),E(R,J)]$, constructed in \S~6, and thus could not even
be stated this way before [47], [44], [45]. In these results 
{\it both\/} occurences of $R$ are replaced by its ideals.
\begin{TheA}
Let\/ $R$ be a commutative ring and
let\/ $I,J\unlhd R$ be two ideals of\/ $R$. Then there exists
an\/ $L$ such that any commutator
$$ [x,y],\qquad x\in\SL(n,R,I),\quad y\in E(n,R,J) $$
\noindent
is a product of not more than\/ $L$ elementary generators
listed in Theorem~\ref{thewayh}.
\end{TheA}


\begin{TheC}
Let\/ $R$ be a commutative ring and let\/
$I,J\unlhd R$ be two ideals of\/ $R$. Then there exists
an\/ $L$ such that any commutator
$$ [x,y],\qquad x\in G(\Phi,R,I),\quad y\in E(\Phi,R,J) $$
\noindent
is a product of not more than\/ $L$ elementary generators
listed in Theorem~\ref{tuboon1}.
\end{TheC}

Quite remarkably, the bound $L$ in these theorems does not depend
either on the ring $R$, or on the choice of the ideals $I,J$.
The proofs of these theorems are paradigmatic applications of
the method of {\sc universal localisation\/}, introduced by the
second author [86], specifically to eliminate any dependence upon
dimension of $R$. Actually, we sketched the main ideas of this
method in [36], \S~10.
\par
These proofs are not particularly long, but rely on a whole bunch
of universal constructions and will be published elsewhere. From
the proofs, it becomes apparent that similar results hold also
in other such situations: for any other functorial generating set,
for multiple relative commutators [86], [37], etc., etc.

It is natural to ask, whether similar results hold for unitary
groups? Well, this is an open question. The point is that Bak's
unitary groups are not always algebraic, and it is not immediate,
how to generalise the universal constructions of [86] from groups
defined by equations to groups defined by congruences.
\par
As a result, at this time we only have {\it weaker\/} versions
of the above results, with bounds depending on the Jacobson
dimension of the centre. The proofs of the following results use
localisation in the same style as [42], [44], they will be published
in [39].
\setcounter{TheB}{15}

\begin{TheB}
Let\/ $n\ge 3$, and let\/ $(A,\Lambda)$ be a quasi
finite form algebra over a commutative ring\/ $R$
with\/ $\dim\Max(R)=d<\infty$. Further, let\/ $(I,\Gamma)$ be a
form ideal of\/ $(\Form)$. Then there exists an\/ $L$ such that
any commutator\/ $[x,y]$, where
$$ x\in\SU(2n,I,\Gamma),\ \ y\in\EU(2n,R,\Lambda)
\quad\text{or}\quad
x\in\SU(2n,R,\Lambda),\ \ y\in\EU(2n,I,\Gamma) $$
\noindent
is a product of not more than\/ $L$ elementary generators\/
$Z_{ij}(\xi,\zeta)\in\EU(2n,I,\Gamma)$.
\end{TheB}

\begin{TheB}
Let\/ $n\ge 3$, and let\/ $(A,\Lambda)$ be a quasi
finite form algebra over a commutative ring\/ $R$
with\/ $\dim\Max(R)=d<\infty$. Further, let\/ $(I,\Gamma)$
and\/ $(J,\Delta)$ be two form ideals of\/ $(\Form)$. Then there
exists an\/ $L$ such that any commutator
$$ [x,y],\qquad x\in\SU(n,I,\Gamma),\quad y\in\EU(n,J,\Delta) $$
\noindent
is a product of not more than\/ $L$ elementary generators
listed in Theorem~\ref{5bbb}.
\end{TheB}

Thus, the following problem naturally suggests itself.
\setcounter{Prob}{2}
\begin{Prob}
Develop versions of universal localisation in the
non-algebraic setting, in particular, for unitary groups.
\end{Prob}
Since the publication of [36] the second author had some progress
in this direction, but a definitive solution is still missing.



\section{Relative splitting}

In this section we describe the relative splitting principle,
which establishes a remarkable interpretion of the relative
commutator subgroups in the case of splitting ideals. This
connection was discovered in the paper by Chattopadhya and Alexei
Stepanov \cite{CS}.
\par
A two-sided ideal $I\unlhd A$ is called a splitting ideal, if
$A=I\oplus B$ as an additive group, and at that $B\le A$ is
a subring of $A$. Clearly, in this case $B\cong A/I$ and the
projection $A\map B\le A$ is a retraction.

The following [absolute] splitting principle is classically known
and widely used

\begin{Lem}
Let $I$ be a splitting ideal of an associative ring $A$. Then
$$ \GL(n,A,I)\cap E(n,A)=E(n,A,I). $$
\end{Lem}

This principle immediately implies the following result.

\begin{Lem}
Let $I$ be a splitting ideal of an associative ring $A$ and
$J\unlhd A$ be any two-sided ideal. Then
$$ \GL(n,A,I)\cap E(n,A,J)=E(n,A,I)\cap E(n,R,J). $$
\end{Lem}

The following result is a proper relative version of the
splitting principle. It establishes a remarkable connection
of splitting with the relative commutator sub\-groups.

\begin{Lem}
Let $I$ be a splitting ideal of an associative ring $A$,
$A=I\oplus B$, and let $J\unlhd A$ be a two-sided ideal of
$A$ generated by a two-sided ideal $K\unlhd B$. Then
$$ \GL(n,A,I)\cap E(n,A,J)=[E(n,A,I),E(n,R,J)]. $$
\end{Lem}
\begin{proof}
By Lemma it suffices to show that
$$ E(n,A,J)=[E(n,A,I),E(n,R,J)]E(n,B,K). $$
\noindent
By definition of a splitting ideal one has $I=K\oplus IJ$.
\end{proof}


\section{Relative Quillen--Suslin lemma}

The following result is the main result of the paper
by Himanee Apte with the second author.

\begin{The}
Let $\Phi$ be a root system of rank $\ge 2$, $I\unlhd R$ be an
ideal of a commutative ring $R$. Further, let
$g\in G(\Phi,R[x],xI[x])$ be such that $F_{\mm}\in
E(\Phi,R_{\mm}[x],xI_{\mm}[x])$. Then $g\in E(\Phi,R[x],xI[x])$.
\end{The}


\section{Where next?}

\begin{center}
{\parbox{.6\textwidth}{\small Let us hear the suspicions. I will look after the proofs.\\
\smallskip
Arthur Conan Doyle}}
\end{center}
\bigskip

In conclusion, we attach an updated list of unsolved problems
related to the results of [36] and the present paper. The results
of the present paper solve some of the problems stated in [36],
whereas some others need corrections. We keep working on these
problems and hope to be able to address them in subsequent
publications.
\setcounter{Prob}{0}
\begin{Prob}
Obtain explicit length estimates in the relative
conjugation calculus and commutator calculus.
\end{Prob}
\begin{Prob}
Obtain explicit length estimates in the universal
localisation.
\end{Prob}
\setcounter{Prob}{3}
Notice that in the vast majority of results on bounded width
of commutators we had to assume that both factors have
determinant 1. So far we were unable to fight the toral factor.
The following problem is not fully solved even in the absolute
case for the general linear group. In [80] there is a partial
result for rings of geometric origin.
\begin{Prob}
Prove the bounded width of commutators of the form
$[x,y]$, where $x\in\GL(n,R)$, $y\in E(n,R)$.
\end{Prob}
After this is done, one could certainly improve the
corresponding results also at the relative level.
\begin{Prob}
Replace in Theorems $7A$, $7B$, $9A$, $9B$
commutativity of the ground ring by a weaker commutativity
assumption.
\end{Prob}
Proposition 2.7 of [80] is a stable version of the main result
on the bounded width of commutators. It asserts that under
condition $n\ge\sr(R)+1$ the length of any commutator $[x,y]$,
where $x,y\in\GL(n,R)$ in elementary generators is bounded.
It would be natural to generalise this result to unitary
groups, and to Chevalley groups. However, it would require
very precise forms of injective stability for $K_1$, with
efficient proofs, if one wishes to obtain explicit bounds.
\begin{Prob}
Prove that under condition $n\ge\Lambda\sr(R)+1$
the width of commutators of the form $[x,y]$, where
$x,y\in\GU(2n,R,\Lambda)$, in the elementary generators of
the group $\EU(2n,R,\Lambda)$ is bounded.
\end{Prob}
It seems that for Chevalley groups one still have to do some
work to find appropriate stability conditions.
\begin{Prob}
Find stability conditions, under which the width
of commutators of the form $[x,y]$, where $x,y\in G(\Phi,R)$,
in the elementary generators of the group $E(\Phi,R)$ is bounded.
\end{Prob}
Let us mention yet another variation of the finite width of
commutators in elementary generators. It seems that such a
result would require centrality of the extension
$\St(n,R)\map E(n,R)$, hence the
restriction $n\ge 5$. Thus, with our present state of knowledge,
there is little hope to generalise it to other groups. Even
for the linear case it is only proven for rings of geometric
origin, [80], Theorem 2.1.
\begin{Prob}
Let $R$ be a commutative ring. Prove the bounded width
of commutators of the form $[x,y]$, where $x,y\in\St(n,R)$,
$n\ge 5$, with respect to the elementary generators.
\end{Prob}

Another important challenge is to improve rank bounds in
the commutator calculus for the unitary groups. In the absolute
case it is the usual condition stated in [12], [13].
\begin{Prob}
Develop conjugation calculus and commutator calculus
in the group\/ $\GU(4,R,\Lambda)$, pro\-vided\/
$\Lambda R+R\Lambda=\Lambda$.
\end{Prob}
In [36] without much thinking we conjectured that the same
condition would suffice to develop relative commutator
calculus. However, a closer look shows that in the relative case
the above condition is far too weak. The correct conditions
should be stated in terms of the {\it relative\/} form
parameters, and seem to be extremely restrictive, so that we
are not sure we would be inclined to work out all details
under such ridiculous confinements.
\begin{Prob}
Prove relative commutator formulae for the group
$\GU(4,R,\Lambda)$, pro\-vided $\Gamma J+J\Gamma=\Gamma$,
$\Delta I+I\Delta=\Delta$.
\end{Prob}
Another important problem is the description of {\it subnormal\/}
subgroups of $G(R)$. For the case of $\GL(n,R)$ this problem has
a fully satisfactory answer, due to the works by John Wilson,
Leonid Vaserstein, and others, see in particular [Wi], [7], [99],
[63], [106], [103], [123].
\par
For unitary groups, there are works by G\"unter Habdank, the fourth
author, and You Hong, see, in particular, [30], [31], [122] -- [125].
Presently You Hong, the third and the fourth authors are working
towards an improvement of bounds. But there are still a number of
loose ends so that an alternative approach based on
the methods of [33], [34], [10], [42], [44], [39]
would be very much desirable.
\begin{Prob} Give localisation proofs for the description of
subgroups of the unita\-ry group $\GU(2n,R,\Lambda)$, normalised
by the relative elementary subgroup $\EU(\Phi,I,\Gamma)$,
for a form ideal $(I,\Gamma)$.
\end{Prob}
Similar problem for Chevalley groups is still largely unsolved.
Again, a localisa\-tion approach based on the methods [40], [88],
[43] would be most welcome.
\begin{Prob}
Using relative localisation, describe subgroups of
a Chevalley group $G(\Phi,R)$, normalised by the relative
elementary subgroup $E(\Phi,R,I)$, for an ideal $I\unlhd R$.
\end{Prob}
Actually, the third and the fourth author classified such subgroups
in Chevalley groups of types $\E_6$ and $\E_7$, but we used a
{\it completely different\/} geometric method, the proof from
the Book.
\par
Now that we have relativised results on nilpotent filtration for
the general linear group, it is natural to obtain similar results
for unitary groups. This will require improvement and relative
versions of many known results, for instance, those related to
stabilisation.
\begin{Prob}
Obtain the analogues of the general relative multiple
commutator formula for unitary groups.
\end{Prob}
In [36] and the present paper we work at the level of $K_1$. It
is natural to ask, to which extent
\begin{Prob}
Develop localisation methods at the level of $K_2$.
In particular, devise a localisation proof of the centrality of
$K_2$.
\end{Prob}
The fact that higher mixed commutators of relative elementary
subgroups can always be expressed as double commutators, seems
to be very surprising. It strongly suggests that there is
close connection with some finer aspects of the structure of $K_1$,
such as homotopy stability, exact sequences and excision.
Especially striking is the analogy with the works of Susan Geller
and Charles Weibel [24] -- [27] on doubly relatives $K_1$ and
excision.
\begin{Prob}
Establish connection of $[E(\Phi,R,A),E(\Phi,R,B)]$
with the excision kernel.
\end{Prob}
Morally, the results of [80], [88], [36], [37], [86] assert that 
$G(R)$ has {\it very\/} few commutators
when $\dim(R)\ge 2$. Positive results on commutator width
only hold under some very strong finiteness conditions on $R$,
and require completely different techniques. See [83] for
some specific conjectures.
\par
It is very challenging to understand, to which extent such
behaviour is typical for more general classes of group words.
There are a lot of recent results showing that the verbal length
of the finite simple groups is strikingly small [81], [82], [59], 
[60], [29]. In fact, under some natural assumptions for large finite
simple groups this verbal length is 2. We do not expect
similar results to hold for rings other than the zero-dimensional
ones, and some arithmetic rings of dimension 1.
\par
Powers are a class of words in a certain sense opposite to
commutators. Alireza Abdollahi suggested that before passing
to more general words, we should first look at powers. An
answer -- in fact, {\it any\/} answer! -- to the following
problem would be amazing. However, we would be less surprised
if for rings of dimension $\ge 2$ the verbal maps in $G(R)$
would have very small images.
\begin{Prob}
Establish finite width of powers in elementary
generators, or lack thereof.
\end{Prob}
In this connection let us mention the following purely group
theoretical problem. It is well known that a commutator is
a product of [not more than] three squares $[x,y]=x^2(x^{-1}y)^2y^2$.
Similarly, a commutator of commutators $[[x,y],[u,v]]$
is a product of not more than 60 cubes [4], whereas the
Engelian commutator $[x,y,y]$ is a product of 3 cubes [5].
\begin{Prob}
When a higher commutator of commutators can be
expressed as a bounded product of $m$-th powers?
\end{Prob}
The following two problems are in fact not individual clear cut
problems, but rather huge research projects.

\begin{Prob}
Generalise results of $[36]$ and the present paper
to odd unitary groups.
\end{Prob}
\begin{Prob}
Obtain results similar to those of $[36]$ and the
present paper for {\rm[}groups of points of{\rm]} isotropic
reductive groups.
\end{Prob}
In the first one of these settings there are foundational works by
Victor Petrov [75] -- [77], while in the second one there are papers
by Victor Petrov, Anastasia Stavrova, and Alexander Luzgarev [78],
[68], [KS], [84] with versions of Quillen--Suslin lem\-ma. But 
that's about it. Most of
the conjugation calculus and the commutator calculus, including
relative results, explicit estimates, etc., have to be developed
from scratch.
\par
In particular, this applies to the relative results. If one is not
willing to stipulate invertibility of 2 and 3, these results should
be stated not in terms of ideals of the ground ring, but rather in
terms of much fancier non-associative structures.
\par
Let us mention two much less ambitious subprojects. First, all
isotropic reductive groups are {\it inner\/} forms of quasi-split
groups = twisted Chevalley groups. The following problem seems to
be much more immediate than Problem 16. It can be easily approached
by the methods developed in [40], [88], [43], [86]. On the other
hand, it is not clear, how much helpful it might be to treat it
separately, outside of the general context of isotropic reductive
groups.
\begin{Prob}
Obtain results similar to those of $[36]$ and the
present paper for twisted Chevalley groups.
\end{Prob}
Another much tamer common piece of Problems 15 and 16 can be stated
as follows.
\begin{Prob}
Generalise results of $[36]$ and the present paper
to orthogonal and symplectic groups of Witt index $\ge 3$.
\end{Prob}
It seems that most necessary tools are contained already in
the works by Victor Petrov, Tony, Bak, Rabeya Basu, Ravi Rao and
Reema Khanna [75] -- [77], [9], [16] -- [17].

\forgotten 



\begin{thebibliography}{100}
\bibitem{1} E.~Abe, {\it
Whitehead groups of Chevalley groups over polynomial rings,}
{ Comm. Algebra}  {\bf 11} (1983), no. 12, 1271--1308.

\bibitem{2} E.~Abe, {\it Chevalley groups over
commutative rings} {Proc.\ Conf.\ Radical Theory, Sendai,}
-- 1988, pp. 1--23.

\bibitem{3} E.~Abe, {\it Normal subgroups of Chevalley
groups over commutative rings,}  Contemp.\ Math. {\bf83}
(1989), 1--17.

\bibitem{4} M.~Akhavan-Malayeri, {\it Writing
certain commutators as products of cubes in free groups,}
{ J.\ Pure Appl. Algebra} {\bf177} (2003), no. 1,
1--4.

\bibitem{5} M.~Akhavan-Malayeri, {\it Writing commutators
of commutators as products of cubes in groups,}
 Comm.\ Algebra {\bf37}  (2009),  2142--2144.


\bibitem{6} H.~Apte, A.~Stepanov,
{\it Local-global principle for congruence subgroups of Chevalley groups}
{\tt arXiv:1211.3575} (2012). 

\bibitem{B1} A.~Bak, The stable structure of quadratic modules.
Thesis, Columbia University, 1969.

\bibitem{B2} A.~Bak, K-Theory of Forms. Annals of
Mathematics Studies {\bf98}, Princeton University Press. Princeton, 1981.


\bibitem{7} A.~Bak, {\it Subgroups of the general linear group
normalized by relative elementary groups,}
Lecture Notes in Math., vol. 967, Springer, Berlin, (1982), 1--22.


\bibitem{8} A.~Bak, {\it Non-abelian $\K$-theory: The
nilpotent class of $\K_1$ and general stability}
K--Theory {\bf4} (1991), 363--397.

\bibitem{9} A.~Bak, R.~Basu, R.~A.~Rao,
{\it Local-Global Principle for Transvection Groups,}
Proc. Amer. Math. Soc.  (2010) (to appear).


\bibitem{10} A.~Bak, R.~Hazrat, N.~A.~Vavilov,
{\it Localization-completion strikes again: relative\/ $\text{\rm K}_1$
is nilpotent by abelian}, J. Pure Appl. Algebra {\bf 213} (2009),
1075--1085.


\bibitem{11} A.~Bak, A.~Stepanov,
{\it Dimension theory and nonstable\/ $\K$-theory for net groups,}
Rend. Sem. Mat. Univ. Padova {\bf106}
(2001),   207--253.

\bibitem{12} A.~Bak, N.~A.~Vavilov, {\it Normality for
elementary subgroup functors,}  Math. Proc.
Cambridge Phil. Soc.  {\bf 118}  (1995), no. 1, 35--47.


\bibitem{13} A.~Bak, N.~A.~Vavilov, {\it Structure of
hyperbolic unitary groups. {\rm I}. Elementary subgroups,} 
Algebra Colloq. {\bf7} (2000), no. 2,  159--196.

\bibitem{14} H.~Bass, {\it $\K$-theory and stable algebra,}
Inst. Hautes \'Etudes Sci. Publ. Math.  (1964), no. 22, 5--60.


\bibitem{15} H.~Bass, J.~Milnor, J.-P.~Serre,
{\it Solution of the congruence subgroup problem for
$\SL_n$ $(n\ge3)$ and $\Sp_{2n}$ $(n\ge2)$,}
Inst. Hautes \'Etudes Sci. Publ. Math. (1967),
no. 33,  59--133.


\bibitem{bass73} H. Bass, {\it Unitary algebraic $K$-theory}, 
Lecture Notes Math. {\bf 343} (1973), 57--265.

\bibitem{Bass1}H. Bass, \textit{Algebraic {K}-theory}. Benjamin, New York, 1968.


\bibitem{16} R.~Basu,
{\it Topics in Classical Algebraic K-Theory,}
 Ph.~D. Thesis,  TIFR, 2006, pp. 1--70.

\bibitem{17} R.~Basu,
{\it On general quadratic group}
 (2012), 1--8.


\bibitem{18} R.~Basu,
{\it Local-global principle for quadratic and hermitian groups
and the nilpotence of $\K_1$,}
(2012), 1--18.


\bibitem{19} R.~Basu, R.~A.~Rao, R.~Khanna,
{\it On Quillen's local global principle,}
Commutative algebra and algebraic geometry,
Contemp. Math. {\bf390} (2005), Amer. Math. Soc.  Providence, RI, 17--30.

\bibitem{borvav} Z. Borevich, N.A. Vavilov,  The distribution of
subgroups in the full linear group over a commutative ring, \textit{Proc.
Steklov Institute Math} {\bf 3} (1985), 27--46.


\bibitem{20} P.~Chattopadhya, R.~A.~Rao,
{\it Excision and elementary symplectic action,}
(2012),  1--14.


\bibitem{21} R.~K.~Dennis, L.~N.~Vaserstein, {\it On a
question of M.~Newman on the number of commutators} 
J.~Algebra {\bf118} (1988), 150--161.


\bibitem{22} R.~K.~Dennis, L.~N.~Vaserstein, {\it
Commutators in linear groups,} $\K$-Theory {\bf2} (1989),
761--767.


\bibitem{23} E.~Ellers, N.~Gordeev,
{\it On the conjectures of J.~Thompson and O.~Ore,}
Trans. Amer. Math. Soc. {\bf350} (1998), 3657--3671.

\bibitem{estesohm} D. Estes, J. Ohm, {\it Stable range in
commutative rings}, J.~Algebra, {\bf 7} (1967), no. 3, 343--362.


\bibitem{24} S.~C.~Geller, C.~A.~Weibel,
{\it $\K_2$ measures excision for $\K_1$,}
 Proc. Amer. Math. Soc. {\bf80}, (1980), no. 1, 1--9.


\bibitem{25} S.~C.~Geller, C.~A.~Weibel, {\it
$\K_1(A,B,I)$,}  J. Reine Angew. Math. {\bf342}
(1983), 12--34.


\bibitem{26} S.~C.~Geller, C.~A.~Weibel, {\it
Subroups of elementary and Steinberg groups of congruence level
$I^2$,}  J. Pure Appl. Algebra {\bf35} (1985),
123--132.


\bibitem{27} S.~C.~Geller, C.~A.~Weibel, {\it
$\K_1(A,B,I)$, {\rm II},} $K$-Theory {\bf2} 
(1989), no. 6, 753--760.


\bibitem{28} V.~N.~Gerasimov,
{\it The group of units of the free product of rings,}
Math. Sbornik {\bf134} (1987), 42--65.


\bibitem{29} R.~M.~Guralnick, G.~Malle,
{\it Products of conjugacy classes and fixed point spaces}
J. Amer. Math. Soc. {\bf25} (2012), no. 1, 77--121.


\bibitem{30} G.~Habdank, {\it A classification of subgroups of
$\Lambda$-quadratic groups normalized by relative elemen\-tary
subgroups,}
 Dissertation Universit\"at Bielefeld, 1987, pp. 1--71.

\bibitem{31} G.~Habdank, {\it A classification of subgroups
of $\Lambda$-quadratic groups normalized by relative elemen\-tary
subgroups,}
Adv. Math. {\bf110} (1995), no. 2, 191--233.


\bibitem{32} A.~J.~Hahn, O.~T.~O'Meara,  The
classical groups and K-theory,  Springer, Berlin
et al.,  1989.


\bibitem{33} R.~Hazrat, {\it Dimension theory and
non-stable\/ $\K_1$ of quadratic module,}
$\K$-Theory {\bf27} (2002), 293--327.




\bibitem{35} R.~Hazrat, V.~Petrov, N.~Vavilov,
{\it Relative subgroups in Chevalley groups,}
J. $\K$-Theory {\bf5} (2010), 603--618.

\bibitem{36} R.~Hazrat, A.~Stepanov, N.~Vavilov,
Z.~Zhang,  {\it The yoga of commutators,}
J. Math. Sci. (N.~Y.)  {\bf179} (2011), no. 6, 662--678.


\bibitem{37} R.~Hazrat, A.~Stepanov, N.~Vavilov,
Z.~Zhang,
{\it Commutators width in Chevalley groups,}
Note di Matematica 33 (2013) no. 1, 139--170.


\bibitem{38} R.~Hazrat, A.~Stepanov, N.~Vavilov,
Z.~Zhang,
{\it Multiple commutator formula. {\rm II},}
(2012). 




\bibitem{40} R.~Hazrat, N.~Vavilov, {\it $\K_1$ of
Chevalley groups are nilpotent,}  J. Pure Appl. Algebra
{\bf 179} (2003), 99--116.


\bibitem{41} R.~Hazrat, N.~Vavilov, {\it Bak's
work on\/ $\K$-theory of rings\/ {\rm(}\!with an appendix by
Max Karoubi\,{\rm)},} J. K-Theory {\bf4} (2009), no. 1, 1--65.


\bibitem{42} R.~Hazrat, N.~Vavilov, Z.~Zhang,
{\it Relative commutator calculus in unitary groups,
and applications,}  J. Algebra {\bf343} (2011),
107--137.


\bibitem{43} R.~Hazrat, N.~Vavilov, Z.~Zhang,
{\it Relative commutator calculus in Chevalley groups,
and applications} (2012), 1--37.
{\tt arXiv:1107.3009v1 [math.RA]} 15 Jul 2011.


\bibitem{44} R.~Hazrat, N.~Vavilov, Z.~Zhang,
{\it Multiple commutator formulas for unitary groups,}
(2012), 1--33.
{\tt arXiv:1205.6866v1 [math.RA]} 31 May 2012.


\bibitem{45} R.~Hazrat, N.~Vavilov, Z.~Zhang,
{\it Generation of relative commutator
subgroups in Chevalley groups,}
(2012), 1--15. Proc. Edinburgh Math. Soc., to appear. 


\bibitem{46} R.~Hazrat, Z.~Zhang,
{\it Generalized commutator formula,} Comm. Algebra {\bf39}
(2011), no. 4, 1441--1454.


\bibitem{47} R.~Hazrat, Z.~Zhang,
{\it Multiple commutator formula,} Israel J. Math.,
195 (2013), 481--505.

\bibitem{48} D.~A.~Jackson,
{\it Basic commutator in weights six and seven as relators,}
 Comm. Algebra {\bf36} (2008), 2905--2909.


\bibitem{49} D.~A.~Jackson, A.~M.~Gaglione, D.~Spellman,
{\it Basic commutator as relators,}
J. Group Theory {\bf 5} (2001), 351--363.


\bibitem{50} D.~A.~Jackson, A.~M.~Gaglione, D.~Spellman,
{\it Weight five basic commutator as relators,}
Contemp. Math. {\bf511} (2010), 39--81.


\bibitem{51} W.~van der Kallen,
{\it Another presentation for Steinberg groups,}
Indag. Math. {\bf39} (1977), no. 4, 304--312.


\bibitem{52} W.~van der Kallen, {\it $\SL_3(\Co[x])$
does not have bounded word length,}  Springer Lecture Notes
Math., vol. 966, 1982, pp. 357--361.


\bibitem{53} W.~van der Kallen, {\it A module structure on
certain orbit sets of unimodular rows,}  J. Pure Appl.
Algebra {\bf57} (1989), no. 3, 281--316.


\bibitem{kallenmagurn} W. van der Kallen,  B. Magurn, 
L. Vaserstein, {\it Absolute stable rank and Witt cancellation for
non-commutative rings}, Invent. Math., {\bf 91} (1988) 525--542.

\bibitem{54} L.-C.~Kappe, R.~F.~Morse,
{\it On commutators in groups,}
Groups St.~Andrews 2005, vol. 2,  C.U.P.,  2007,
pp. 531--558.


\bibitem{khleb} S.~Khlebutin, {\it Elementary subgroups of linear
groups over rings\/}, Ph.~D.\ Thesis, Moscow State Univ., 1987 (in
Russian).

\bibitem{55} M.-A.~Knus,
{\it Quadratic and hermitian forms over rings,}
Springer Verlag, Berlin et al., 1991. 


\bibitem{56} V.~I.~Kopeiko,
{\it The stabilization of symplectic groups
over a polynomial ring,} Math.\ U.S.S.R. Sbornik
{\bf 34}  (1978), 655--669.

\bibitem{57} N.~Kumar, R.~A.~Rao,
{\it Quillen--Suslin theory for a structure theorem for the
elementary symplectic group,}  (2012), 1--15.


\bibitem{58} Tsit-Yuen Lam, {\it Serre's problem
on projective modules,}  Springer Verlag,
Berlin et al., 2006.


\bibitem{59} M.~Larsen, A.~Shalev, 
{\it Word maps and Waring type problems,}
 J. Amer. Math. Soc. {\bf22} (2009), 437--466.


\bibitem{60} M.~Larsen, A.~Shalev, Pham Huu Tiep,
{\it The Waring problem for finite simple groups,} Ann. Math. {\bf174} (2011), 1885--1950.

\bibitem{lavrenov} A.~V.~Lavrenov, The unitary Steinberg group is
centrally closed. {\it St.~Petersburg Math. J.}, {\bf 24} (2012)
to appear.


\bibitem{61}  F.~Li, {\it The structure of symplectic
group over arbitrary commutative rings,}
Acta Math. Sinica 
{\bf 3} (1987), no. 3, 247--255.


\bibitem{62}  F.~Li,
{\it The structure of orthogonal groups over arbitrary
commutative rings,}   Chinese Ann. Math. Ser.~B
{\bf10} (1989), no. 3, 341--350.


\bibitem{63} F.~Li, M.~Liu,
{\it Generalized sandwich theorem,} $\K$-Theory
{\bf 1} (1987), 171--184.


\bibitem{64} M.~Liebeck, E.~A.~O'Brien, A.~Shalev,
Pham Huu Tiep, {\it The Ore conjecture,}
 J. Europ. Math. Soc. {\bf12} (2010), 939--1008.


\bibitem{65} M.~Liebeck, E.~A.~O'Brien, A.~Shalev,
Pham Huu Tiep, {\it Commutators in finite quasisimple groups,}
Bull. London Math. Soc. {\bf43} (2011), 1079--1092.


\bibitem{66} M.~Liebeck, E.~A.~O'Brien, A.~Shalev,
Pham Huu Tiep,
{\it Products of squares in finite simple groups,}
Proc. Amer. Math. Soc. {\bf43}  (2012), no. 6, 1079--1092.


\bibitem{67} A.~Yu.~Luzgarev, {\it Overgroups of\/
$E(F_4,R)$ in\/ $G(E_6,R)$,}  St. Petersburg J. Math. {\bf20}
(2008),  no. 5, 148--185.


\bibitem{68} A.~Yu.~Luzgarev, A.~K.~Stavrova, {\it
Elementary subgroups of isotropic reductive groups are perfect,}
 St.~Peters\-burg Math. J. {\bf24} (2012), no. 5,
881--890.


\bibitem{69} A.~W.~Mason,
{\it A note on subgroups of\/ $\GL(n,A)$ which are
generated by commutators,}
J. London Math. Soc. {\bf11} (1974), 509--512.


\bibitem{70} A.~W.~Mason, {\it On subgroups of\/
$\GL(n,A)$ which are generated by commutators. {\rm II},}
 J.~reine angew.~Math. {\bf322} (1981), 118--135.


\bibitem{71} A.~W.~Mason,
{\it A further note on subgroups of\/ $\GL(n,A)$ which are
generated by commutators,}
Arch. Math. {\bf37} (1981), no. 5, 401--405.


\bibitem{72} A.~W.~Mason, W.~W.~Stothers,
{\it On subgroups of\/ $\GL(n,A)$ which are generated by commuta\-tors,}
Invent. Math. {\bf23} (1974), 327--346.


\bibitem{73} H.~Matsumoto, {\it Sur les sous-groupes
arithm\'etiques des groupes semi-simples d\'eploy\'es,}
Ann.\ Sci.\ \'Ecole Norm.~Sup. (4) {\bf2} (1969), 1--62.

\bibitem{milnor70} J. Milnor,  {\it Algebraic $K$-theory and
quadratic forms,} Invent. Math. {\bf 9} (1970), 318--344.

\bibitem{milnor} J. Milnor,  Introduction to algebraic
$K$-theory, Princeton Univ. Press, Princeton, N. J., 1971.


\bibitem{moore} C. Moore, {\it Group extensions of $p$-adic
and adelic linear groups}, Publ. Math. Inst. Hautes Etudes Sci.
{\bf 35} (1968), 157--222.


\bibitem{74} D.~W.~Morris,
{\it Bounded generation of\/ $\SL(n,A)$ {\rm(}after D.~Carter,
G.~Keller, and E.~Paige{\rm)},} New York J. Math. {\bf13} 2007,
383--421.


\bibitem{75} V.~A.~Petrov, 
{\it Overgroups of unitary groups,} $K$-Theory {\bf29}
(2003), 147--174.


\bibitem{76}  V.~A.~Petrov, {\it Odd unitary groups,}
 J. Math. Sci. {\bf130} (2003),  no. 3, 4752--4766.

\bibitem{77} V.~A.~Petrov,
{\it Overgroups of classical groups,}
Doktorarbeit Univ. St.-Petersburg,  2005, pp. 1--129,
(in Russian).


\bibitem{78} V.~A.~Petrov, A.~K.~Stavrova, {\it
Elementary subgroups of isotropic reductive groups}
 St.~Peters\-burg Math. J.  {\bf20}  (2008), no. 3,
160--188.

\bibitem{Quillen76} D. Quillen, {\it Projective modules over polynomial rings}, Invent. Math., {\bf36} (1976), 166-172.

\bibitem{Quillen72} D. Quillen,  Higher algebraic K-theory. I. Algebraic K-theory, I: Higher K-theories, Lecture Notes in Math., Vol. {\bf 341}, Springer, Berlin 1973, 85--147.

\bibitem{rosenberg} J. Rosenberg, Algebraic $K$-Theory and
its Applications, Springer-Verlag, New York, 1994.



\bibitem{79} Sh.~Rosset,
{\it The higher lower central series,}
Israel J. Math. {\bf73} (1991), no. 3, 257--279.


\bibitem{80} A.~Sivatski, A.~Stepanov, {\it On the word
length of commutators in\/ $\GL_n(R)$,} $\K$-Theory {\bf17} (1999),
295--302.


\bibitem{81} A.~Shalev,
{\it Commutators, words, conjugacy classes, and character methods,}
Turk. J. Math.  {\bf31} (2007), 131--148.


\bibitem{82} A.~Shalev,
{\it Word maps, conjugacy classes, and a noncommutative
Waring-type theorem,}
Ann. Math. {\bf170} (2009), no. 3, 1383--1416.


\bibitem{83} A.~Smolensky, B.~Sury, N.~A.~Vavilov,
{\it Gauss decomposition for Chevalley groups revisited,}
Intern. J. Group Theory {\bf1} (2012), no. 1,
3--16.


\bibitem{84} A.~Stavrova, {\it Homotopy invariance
of non-stable $\K_1$-functors,} 
(2012), 1--24 (to appear).


\bibitem{85} M.~R.~Stein,
{\it Generators, relations and coverings of Chevalley
groups over commutative rings,} Amer. J.~Math. {\bf93}
(1971), no.  4, 965--1004.


\bibitem{86} A.~V.~Stepanov, {\it Universal
localisation in algebraic groups}, 
\url{ http://alexei.stepanov.spb.ru/~publicat.html}
(2010) (to appear).


\bibitem{87} A.~V.~Stepanov, N.~A.~Vavilov, {\it
Decomposition of transvections: a theme with variations,} 
$\K$-Theory {\bf19} (2000), 109--153.


\bibitem{88} A.~V.~Stepanov, N.~A.~Vavilov, {\it
On the length of commutators in Chevalley groups,} 
Israel J. Math. {\bf185} (2011), 253--276.


\bibitem{89} A.~V.~Stepanov, N.~A.~Vavilov, H. You, {\it
Overgroups of semi-simple subgroups: localisation ap\-proach}
(2012), 1--43 (to appear).


\bibitem{steinberg62} R. Steinberg,  {\it G\'en\'erateurs,
r\'elations et rev\^etements des groupes alg\'ebriques}, 
Colloque Th\'eorie des Groupes Alg\'ebriques (Bruxelles---1962),
113--127.


\bibitem{steinberg67} R. Steinberg,  Lectures on Chevalley groups
Yale University, 1967.




\bibitem{90} A.~A.~Suslin, {\it The structure of the
special linear group over polynomial rings,}  Math. USSR Izv.
{\bf11} (1977),  no. 2, 235--253.


\bibitem{91} A.~A.~Suslin, V.~I.~Kopeiko, {\it Quadratic
modules and orthogonal groups over polynomial rings,}  J.~Sov.\
Math. {\bf20} (1982),  no. 6, 2665--2691.


\bibitem{92} K.~Suzuki, {\it Normality of the elementary
subgroups of twisted Chevalley groups over commu\-tative
rings,} J.~Algebra {\bf175} (1995), no. 3, 526--536.


\bibitem{93} R.~G.~Swan, {\it Excision in algebraic
$\K$-theory,}  J. Pure Appl. Algebra {\bf1} (1971), no. 3, 221--252.


\bibitem{94} G.~Taddei, {\it Sch\'emas de
Chevalley--Demazure, fonctions re\-pr\'e\-sen\-ta\-ti\-ves et
th\'e\-or\`e\-me de nor\-malit\'e,}  Th\`ese, Univ.\ de Gen\`eve,
 1985. 


\bibitem{95} G.~Taddei, {\it Normalit\'e des groupes
\'el\'ementaire dans les groupes de Che\-val\-ley sur un anneau,}
Contemp.\ Math. {\bf55} (1986), no. 2, 693--710.

\bibitem{tang} G.~Tang, {\it Hermitian groups and $K$-theory}, 
$K$-Theory, {\bf 13}:3 (1998), 209--267.

\bibitem{96} J.~Tits, {\it Syst\`emes g\'en\'erateurs de
groupes de congruences,} C.~R.~Acad. Sci. Paris, S\'er A
{\bf283} (1976), 693--695.


\bibitem{97} M.~S.~Tulenbaev,
{\it The Steinberg group of a polynomial ring}
Math.\ U.S.S.R.\ Sb. {\bf45} (1983), no. 1, 139--154.


\bibitem{TUL} M.S. Tulenbaev, {\it The Schur multiplier of the group of elementary matrices of finite order}, J. Sov. Math, {\bf17}:4 (1981), 2062--2067.

\bibitem{98} L.~N.~Vaserstein,  {\it On the normal
subgroups of the\/ $\GL_n$ of a ring,} Algebraic $\K$-Theory,
Evans\-ton 1980,
Lecture Notes in Math., vol. 854, Springer,
 Berlin et al.,  1981, pp. 454--465.


\bibitem{99} L.~N.~Vaserstein,  {\it The subnormal structure of
general linear groups,} Math. Proc. Cambridge Phil. Soc. {\bf99}
(1986), 425--431.


\bibitem{100} L.~N.~Vaserstein, {\it On normal subgroups of
Chevalley groups over commutative rings,} T\^ohoku Math.
J. {\bf36} (1986), no. 5, 219--230.


\bibitem{101} L.~N.~Vaserstein, {\it Normal subgroups of
orthogonal groups over commutative rings,}  Amer. J. Math.
{\bf110} (1988),  no. 5, 955--973.


\bibitem{102} L.~N.~Vaserstein, {\it Normal subgroups of
symplectic groups over rings,}  $K$-Theory  {\bf 2} (1989), no. 5,
647--673.


\bibitem{103} L.~N.~Vaserstein,  {\it The subnormal structure of
general linear groups over rings,}  Math. Proc. Cambridge Phil. Soc. {\bf108}
(1990), no. 2, 219--229.


\bibitem{104} L.~N.~Vaserstein, A.~A.~Suslin,
{\it Serre's problem on projective modules over polynomial rings,
and algebraic $K$-theory,}  Math. USSR Izv.
{\bf10} (1978), 937--1001.


\bibitem{105} L.~N.~Vaserstein, H. You, {\it Normal
subgroups of classical groups over rings,}  J. Pure Appl.
Algebra {\bf105} (1995), no. 1, 93--106.


\bibitem{106} N.~A.~Vavilov, {\it A note on the subnormal
structure of general linear groups,}
Math. Proc. Cam\-brid\-ge Phil. Soc. {\bf107}  (1990), no. 2,  193--196.


\bibitem{107} N.~A.~Vavilov, {\it Structure of Chevalley
groups over commutative rings,}  Proc.\ Conf.\
Non-as\-so\-ci\-a\-ti\-ve algebras and related topics
(Hiroshima -- 1990),  World Sci.\ Publ.,  London
et al.,   1991, pp.  219--335.

\bibitem{108} N.~A.~Vavilov, {\it A third look at weight
diagrams,} Rendiconti del Rend. Sem. Mat. Univ. Padova {\bf204} (2000), no. 1,  201--250.

\bibitem{109} N.~A.~Vavilov, A.~Luzgarev, A.~Stepanov,
{\it Calculations in exceptional groups over rings}
J.~Math. Sci. {\bf373} (2009), 48--72.


\bibitem{110} N.~A.~Vavilov, V.~A.~Petrov,
{\it Overgroups of $\Ep(n,R)$,} 
St. Petersburg J. Math. {\bf 15} (2004), no. 4, 515--543.


\bibitem{111} N.~A.~Vavilov, E.~B.~Plotkin, {\it Chevalley
groups over commutative rings. {\rm I.} Elementary
calcu\-la\-tions,} Acta Applicandae Math. {\bf45}
(1996), 73--115.


\bibitem{112} N.~A.~Vavilov, A.~V.~Stepanov, {\it
Standard commutator formula,} Vestnik St.~Petersburg Univ., ser.1
{\bf41} (2008), no. 1, 5--8.

\bibitem{113} N.~A.~Vavilov, A.~V.~Stepanov, {\it
Overgroups of semi-simple groups,}
Vestnik Samara State Univ., Ser. Nat. Sci. (2008), no. 3, 51--95 (in Russian).


\bibitem{114} N.~A.~Vavilov, A.~V.~Stepanov, {\it
Standard commutator formulae, revisited,}
Vestnik St.~Peters\-burg State Univ., ser.1,
{\bf43} (2010),  no. 1, 12--17.


\bibitem{115} N.~A.~Vavilov, A.~V.~Stepanov, {\it
Linear groups over general rings {\rm I}. Generalities,}
J. Math. Sci.  (2012), 1--107.


\bibitem{116} N.~A.~Vavilov, Z. Zhang, {\it
Subnormal subgroups of Chevalley groups. {\rm I}.
Cases $\E_6$ and $\E_7$,}
(2012), 1--24.


\bibitem{117} T.~Vorst, {\it Polynomial extensions
and excision $\K_1$,} Math. Ann.  {\bf244} (1979),
193--204.


\bibitem{118} M.~Wendt,
{\it ${\Bbb A}^1$-homotopy of Chevalley groups,}
J. $K$-Theory  {\bf 5} (2010), no.  2, 245--287.


\bibitem{119} M.~Wendt,  {\it On homotopy
invariance for homology of rank two groups,}
(2012), 1--16 (to appear).


\bibitem{120} H.~You, {\it On the solution to
a question of D.~G.~James,}
J. Northeast Normal Univ. (1982), no. 2, 39--44, (Chinese).


\bibitem{121} H.~You, {\it On subgroups of Chevalley
groups which are generated by commutators,}
J. North\-east Normal Univ. (1992), no. 2, 9--13.


\bibitem{122} H.~You, {\it Subgroups of classical
groups normalised by relative elementary groups,}
J. Pure Appl. Algebra {\bf216} (2012), 1040--1051.


\bibitem{123} Z.~Zhang, {\it Lower\/ $K$-theory of
unitary groups,}  Doktorarbeit Univ. Belfast, 2007, 1--67.


\bibitem{124} Z.~Zhang, {\it Stable sandwich classification
theorem for classical-like groups,}  Math. Proc.
Cambridge Phil. Soc. {\bf143} (2007), no. 3, 607--619.


\bibitem{125} Z.~Zhang, {\it Subnormal structure of
non-stable unitary groups over rings,}  J. Pure Appl. Algebra {\bf214}
(2010), 622--628.


\end{thebibliography}
\end{document}